\documentclass[12pt,reqno]{amsart}

\usepackage{amsmath,amsthm,amssymb, amscd, amsfonts,enumerate,array,bbm,bm}
\usepackage{pb-diagram,lamsarrow,pb-lams}
\usepackage{hyperref}
\usepackage[mathscr]{eucal}

\allowdisplaybreaks[2]
\makeatletter
\def\mhline{\noalign{\ifnum0=`}\fi\hrule height 4\arrayrulewidth \futurelet
   \@tempa\oxhline}
\def\oxhline{\ifx\@tempa\hline\vskip \doublerulesep\fi
      \ifnum0=`{\fi}}

\def\Frac{\mathscr{F}{\rm rac}}

\def\Im{\operatorname{Im}}

\makeatother

\numberwithin{equation}{section}

\newtheorem{theorem}{Theorem}[section]
\newtheorem{proposition}[theorem]{Proposition}
\newtheorem{conjecture}[theorem]{Conjecture}
\newtheorem{corollary}[theorem]{Corollary}
\newtheorem{lemma}[theorem]{Lemma}

\theoremstyle{definition}

\newtheorem{remark}[theorem]{Remark}
\newtheorem{example}[theorem]{Example}
\newtheorem{definition}[theorem]{Definition}

\newtheorem{problem}[theorem]{Problem}
\newcommand{\qbinom}[3][q]{\genfrac[]{0pt}0{#2}{#3}_{#1}}

\DeclareMathOperator{\ad}{ad}

\newcommand{\lie}[1]{\mathfrak{#1}}

\newcommand{\tensor}{\otimes}
\renewcommand{\eqref}[1]{{\rm (\ref{#1})}}

\def\citem#1*#2{\cite[#2]{#1}}
\def\cites#1{\cite{#1}}

\def\ad{\operatorname{ad}}
\DeclareMathOperator{\End}{End}

\def\AA{\mathbb{A}}

\def\ZZ{\mathbb{Z}}
\def\CC{\mathbb{C}}

\def\NN{\mathbb{N}}

\def\id{\operatorname{id}}

\def\kk{\Bbbk}

\def\bi{{\mathbf i}}\def\bj{{\mathbf j}}
\newcommand{\lr}[1]{{\langle #1\rangle}}

\def\lieg{\mathfrak{g}}

\def\nn{\mathfrak{n}}

\def\ii{\mathbf{i}}

\begin{document}

\title{Quantum folding}

\author{Arkady Berenstein}
\address{\noindent Department of Mathematics, University of Oregon,
Eugene, OR 97403, USA} \email{arkadiy@math.uoregon.edu}

\author{Jacob Greenstein}
\address{Department of Mathematics, University of
California, Riverside, CA 92521.} 
 \email{jacob.greenstein@ucr.edu}

\thanks{This work was partially supported by the NSF grants DMS-0800247 (A.~B.) and DMS-0654421~(J.~G.)}

\begin{abstract}
In the present paper we introduce a quantum analogue of the classical folding of a simply-laced Lie algebra~$\lie g$
to the non-simply-laced algebra~$\lie g^\sigma$ along a Dynkin diagram automorphism~$\sigma$ of~$\lie g$.
For each quantum folding we replace $\lie g^\sigma$ by its Langlands dual $\lie g^\sigma{}^\vee$ and
construct a nilpotent Lie algebra $\lie n$ which 
interpolates between the nilpotnent parts of~$\lie g$ and $\lie g^\sigma{}^\vee$, together with its
quantized enveloping algebra $U_q(\lie n)$ and a Poisson structure on~$S(\lie n)$. 
Remarkably, for the pair $(\lie g,\lie g^\sigma{}^\vee)=(\lie{so}_{2n+2},\lie{sp}_{2n})$, the algebra $U_q(\lie n)$
admits an action of the Artin braid group~$Br_n$ and 
contains a new algebra of quantum $n\times n$ matrices with an adjoint action of $U_q(\lie{sl}_n)$, which 
generalizes the algebras constructed by K.~Goodearl and M.~Yakimov in~\cite{GY}. The hardest case of quantum folding 
is, quite expectably, the pair $(\lie{so}_8,G_2)$ for which the PBW presentation of $U_q(\lie n)$ and 
the corresponding Poisson bracket on $S(\lie n)$ contain more than 700 terms each.
\end{abstract}
\maketitle

\tableofcontents
\section{Introduction and main results} 
This work is motivated by the classical ``folding" result for a  simply laced semisimple Lie algebra $\lieg$  and an
{\it admissible}  diagram automorphism $\sigma:\lieg\to \lieg$ (in the sense of~\citem{Lus}*{\S12.1.1}, see~Section~\ref{subsect:classical folding})
\begin{equation}
\label{eq:class fold}
\text{The fixed Lie algebra $\lieg^\sigma=\{x\in \lieg\,:\, \sigma(x)=x\}$ is also semisimple.} 
\end{equation}

Our goal is to find a quantum version of this result. Note, however, that the embedding of associative algebras $U(\lie g^\sigma)\hookrightarrow 
U(\lie g)^\sigma\subset U(\lie g)$ 
induced by the inclusion $\lie g^\sigma\hookrightarrow \lie g$ does 
not admit a naive quantum deformation (see Appendix~\ref{sec:non-existent}). On the other hand, there exists a ``crystal'' version of the desired homomorphism. 
Namely, let $B_\infty(\lie g)$ be the famous Kashiwara crystal introduced in~\cite{Kas}. The following result was proved by G.~Lusztig in~\citem{Lus}*{Section 14.4}.
\begin{proposition} 
\label{pr:crystal folding} 
Let $\sigma$ be an admissible diagram automorphism of~$\lieg$. Then $\sigma$ acts on $B_\infty(\lieg)$ and the fixed point set
$B_\infty(\lieg)^\sigma$ is naturally isomorphic to $B_\infty({\lieg^\sigma}^\vee)$, where ${\lieg^\sigma}^\vee$ is the Langlands dual Lie algebra of $\lieg^\sigma$.
 
\end{proposition}

Note that one can identify (in many ways) the $\CC(q)$-linear span of $B_\infty(\lieg)$ with the quantized enveloping algebra $U_q^+(\lieg)$ of $\lieg_+$, where $\lieg_+$ stands
for the ``upper triangular'' Lie subalgebra of $\lieg$.
This leads us to the following definition.

\begin{definition} A {\it quantum folding} of $\lieg$ is a $\CC(q)$-linear embedding  (not necessarily algebra homomorphism!) 
\begin{equation}
\label{eq:first iota}
\iota:U_q^+({\lieg^\sigma}^\vee) \hookrightarrow U_q^+(\lieg)^\sigma\subset U_q^+(\lieg) 
\end{equation}
(here  $U_q^+({\lieg^\sigma}^\vee)$ comes with powers of $q$ depending on $\sigma$; see Section \ref{subsect:quantum enveloping} for the details).
\end{definition}
 
We construct all relevant quantum foldings below (Proposition \ref{pr:general folding}) and now focus on a rich algebraic structure that can be attached to each quantum folding.
 
\begin{definition}
\label{def:liftable}
We say that a $\kk$-algebra $A$ generated by a totally ordered set $X_A$ is {\it Poincare-Birkhoff-Witt} ({\it PBW}) if the set $M(X_A)$ of all ordered monomials 
in $X_A$ is a basis of $A$. More generally, we say that  $A$ is {\it sub-PBW} if  $M(X_A)$ spans $A$ as a $\kk$-vector space (but $M(X_A)$ is not necessarily 
linearly independent). 

For a given sub-PBW algebra $A$ we say that an algebra $U=U(A,X_A)$ is a {\it uberalgebra} of $A$ if 
\begin{enumerate}[(a)]
\item $U$ is generated by $X_A$ and is a PBW algebra with these generators;

\item The identity map $X_A\to X_A$ extends to a surjective algebra homomorphism $U\twoheadrightarrow A$.
\end{enumerate}
\end{definition}
In general it is not clear whether a given sub-PBW algebra $A$ admits a uberalgebra. A criterion for uniqueness is based on the following notion of tameness of $(A,X_A)$. We need some notation. First,  consider the natural filtration  $\kk=A_0\subset A_1\subset \cdots$ given by $A_k=\operatorname{\operatorname{Span}}\{1,X_A,X_A\cdot X_A,\ldots,(X_A)^k\}$, $k\in\ZZ_{\ge 0}$. Next, for each $X,X'\in X_A$ with $X<X'$ let $d(X,X')$ be the smallest number $d$ such that $X'X\in A_d$. We also denote by $d_0=d_0(A,X_A)$ the maximum of all the $d(X,X')$.

\begin{definition} We say that a sub-PBW algebra $(A,X_A)$ is {\it tame} if the set $M(X_A)\cap A_{d_0}$ is linearly independent.
\end{definition}

\begin{lemma} A tame sub-PBW algebra $(A,X_A)$ admits at most one (up to isomorphism) uberalgebra $U(A,X_A)$.
\end{lemma}

In what follows, we will construct uberalgebras $U(\iota)$ for several quantum foldings $\iota$ as in \eqref{eq:first iota}, and these uberalgebras will depend on $\lieg$ and $\sigma$ 
rather than on a particular choice of $\iota$ and most algebras generated by the image of~$\iota$ will be tame. We need more notation.

\begin{definition} 
\label{def:liftable iota}
Let $A$ and $B$ be PBW-algebras and let $\iota:A\hookrightarrow B$ be an injective map (not necessarily an algebra homomorphism). We say that $\iota$ is {\it liftable} if:
\begin{enumerate}[(i)]
\item\label{def:liftable iota.i} For any ordered monomial $X=X_1^{m_1}\cdots X_N^{m_N}\in M(X_A)$ one has 
$$\iota(X)=\iota(X_1)^{m_1}\cdots \iota(X_N)^{m_N}.$$

\item There exists a finite subset $Z_0\subset \langle A\rangle_\iota$, where $\langle A\rangle_\iota$ is the subalgebra of $B$ generated by $\iota(A)$, 
such that $\langle A\rangle_\iota$ is sub-PBW with respect to $\iota(X_A)\cup Z_0$ (with some ordering of $\iota(X_A)\cup Z_0$ compatible with the ordering of $X$).

\item There exists a uberalgebra $U(\iota):=U(\langle A\rangle_\iota,\iota(X_A)\cup Z_0)$ for $\langle A\rangle_\iota$ and
a surjective homomorphism $\mu:=\mu_\iota:U(\iota)\to A$ such that for all $x\in X_A$, $\mu(\iota(x))=x$ and $\mu(Z_0)=0$.
\end{enumerate}
\end{definition}
If $\iota$ is liftable and $(\langle A\rangle_\iota,\iota(X_A)\cup Z_0)$ is tame (hence $U(\iota)$ is unique), in what follows we refer to an $\iota$ as a {\it tame} liftable quantum folding.

For each liftable $\iota$ we have a diagram
\begin{equation}
\label{eq:cone over iota}
\begin{diagram}
 \node{U(\iota)}\arrow{e,t}{\hat\iota}\arrow{s,lr}{\tilde\iota\Bigg\uparrow}{\mu_\iota} \node{B}\\
 \node{A}\arrow{ne,l}{\iota}
\end{diagram}
\end{equation}
satisfying $\mu_\iota\circ \tilde\iota=\id_A$ and $\hat\iota\circ \tilde\iota=\iota$ where
\begin{itemize}

\item $\tilde \iota:A\hookrightarrow U(\iota)$ is the canonical splitting of~$\mu$ given by
$\tilde\iota(X_1^{m_1}\cdots X_N^{m_N})=\iota(X_1)^{m_1}\cdots \iota(X_N)^{m_N}$ in the notation of Definition~\ref{def:liftable iota}\eqref{def:liftable iota.i}.

\item $\hat \iota:U(\iota)\to B$ is the structural algebra homomorphism given by $\hat\iota(X)=X$ for $X\in \iota(X_A)\cup Z_0$ (e.g., the image of $\hat\iota$ is $\langle A\rangle_\iota$).
\end{itemize}
Note the following easy Lemma.
\begin{lemma}\label{lem:change of basis}
In the notation of Definition~\ref{def:liftable iota}, write $X_A=\{X_1,\dots,X_N\}$ as an ordered set.
Fix $1\le k\le N$ and $f\in\sum_{M\in M(X_A\setminus \{X_k\})}\kk M$. If the uberalgebra $U(\iota)$ is optimal PBW
(in the sense of Definition~\ref{def:specialization}) then 
\begin{enumerate}[{\rm(i)}]
 \item $(A,X'_A)$, where $X'_A=\{X_1,\dots,X_k':=X_k+f,X_{k+1},\dots,X_N\}$, is a PBW algebra.
 \item The injective linear map $\iota':A\to B$ given by the formula $$
 \iota'(X_1^{m_1}\cdots  (X_k')^{m_k}X_{k+1}^{m_{k+1}}\cdots X_N^{m_N})=
 \iota(X_1)^{m_1}\cdots  \iota(X_k')^{m_k}\iota (X_{k+1}^{m_{k+1}})\cdots \iota(X_N^{m_N})$$
 is liftable with $\lr{A}_{\iota'}=\lr{A}_\iota$,
 $U(\iota')=U(\iota)$ and $\mu_\iota=\mu_{\iota'}$.
\end{enumerate}
\end{lemma}

The following is our first main result (see Section \ref{sect:D_{n+1}} for greater details). 

\begin{theorem} 
\label{th:first folding}
For the pair $(\lieg,{\lieg^\sigma}^\vee)=(\lie{so}_{2n+2},\lie{sp}_{2n})$, $n\ge 3$ there exists a tame liftable quantum folding $\iota:
U_q^+(\lie{sp}_{2n})\hookrightarrow U_q^+(\lie{so}_{2n+2})$. The corresponding uberalgebra $U(\iota)$ is isomorphic  to $S_q(V\otimes V)\rtimes U_q^+(\lie{sl}_n)$, where 
$V$ is the standard $n$-dimensional $U_q(\lie{sl}_n)$-module, and $S_q(V\otimes V)$ is a quadratic  PBW-algebra in the category of  $U_q(\lie{sl}_n)$-modules. More precisely: 

\begin{enumerate}[{\rm(i)}]
\item\label{th:first folding.i}
The algebra $S_q(V\otimes V)$ is isomorphic to $T(V\otimes V)/\langle (\Psi-1)(V^{\otimes 4})\rangle$, where $\Psi:V^{\otimes 4}\to V^{\otimes 4}$ is a $\CC(q)$-linear map given by:
\begin{equation}
\label{eq:psi braiding}
\Psi=\Psi_2\Psi_1\Psi_3\Psi_2+(q-q^{-1})(
    \Psi_1\Psi_2\Psi_1+ \Psi_1\Psi_3\Psi_2)+(q-q^{-1})^2 \Psi_1\Psi_2
\end{equation}
where $\Psi_i:V^{\otimes 4}\to V^{\otimes 4}$, $1\le i\le 3$ is, up to a power of $q$, the  braiding operator in the category of $U_q(\lie{sl}_n)$-modules that acts 
in the $i$-th and $(i+1)$st factors and satisfies the normalized  Hecke equation $(\Psi_i-q^{-1})(\Psi_i+q)=0$.

\item\label{th:first folding.ii}
The covariant $U_q(\lie{sl}_n)$-action on the algebra $S_q(V\otimes V)$  is  determined by the natural action of the Hopf algebra $U_q(\lie{sl}_n)$ on $V\otimes V$.

\item\label{th:first folding.iii} The algebra $S_q(V\otimes V)$ is PBW with respect to any ordered  basis of~$V\otimes V$.
\end{enumerate}
\end{theorem}

\begin{remark} Strictly speaking, the cross product $S_q(V\otimes V)\rtimes U_q^+(\lie{sl}_n)$ is ``braided''  in the sense of Majid (\cite{Maj}) because $U_q(\lie{sl}_n)$ is a braided Hopf algebra (see \cite{Lus} and Section \ref{subsect:quantum enveloping}). 

\end{remark}

We prove Theorem \ref{th:first folding} in Section~\ref{sect:D_{n+1}}. In particular, the key ingredient in our proof of part~\eqref{th:first folding.iii} is
the following surprising result.
\begin{proposition}\label{prop:PsiProp} The map $\Psi$ satisfies:
\begin{enumerate}[{\rm(i)}]
\item\label{prop:PsiProp.i} The braid equation in $(V\otimes V)^{\otimes 3}$:
\begin{equation}
\label{eq:braid equation}
(\Psi\otimes 1)(1\otimes \Psi)(\Psi\otimes 1)=(1\otimes \Psi)(\Psi\otimes 1)(1\otimes \Psi).
\end{equation}
\item\label{prop:PsiProp.ii} The cubic version of the Hecke equation:
$$(\Psi-1)(\Psi+q^2)(\Psi+q^{-2})=0.
$$
In particular, $\Psi$ is invertible and:
$$\Psi^{-1}=\Psi_2\Psi_1\Psi_3\Psi_2+(q-q^{-1})(
    \Psi_2\Psi_3\Psi_2+ \Psi_1\Psi_3\Psi_2)+(q-q^{-1})^2 \Psi_3\Psi_2\ .
$$

\item\label{prop:PsiProp.iii} $\dim (\Psi-1)(V\otimes V)=\dim \Lambda^2 V$.
\end{enumerate}
\end{proposition}
 
This and the following general fact that we failed to find in the literature, although
numerous special cases are well-known (cf. for example~\cites{P-P,Gur,D-S}),
settle Theorem \ref{th:first folding}\eqref{th:first folding.iii} (see Section \ref{sect:D_{n+1}} for details).
\begin{theorem} 
\label{thm:flat quadratic algebra} Let $Y$ be a finite-dimensional $\CC(q)$-vector space and let $\Psi$ be an invertible $\CC(q)$-linear map $Y\otimes Y\to Y\otimes Y$ satisfying the braid equation. 
Assume that: 
\begin{enumerate}[{\rm(i)}]
\item\label{thm:flat quadratic algebra.i} the specialization $\Psi|_{q=1}$ of $\Psi$ is the permutation of factors $\tau:Y\otimes Y\to Y\otimes Y$,

\item\label{thm:flat quadratic algebra.ii}  $\dim (\Psi-1)(Y\otimes Y)=\dim \Lambda^2 Y$. 
\end{enumerate}
Then the  algebra $S_\Psi(Y)=T(Y)/\langle (\Psi-1)(Y\otimes Y)\rangle$ is a flat deformation of the symmetric algebra $S(Y)$ (hence $S_\Psi(Y)$ is PBW for any ordered basis of $Y$).
\end{theorem}
We prove Theorem~\ref{thm:flat quadratic algebra} in Section~\ref{subsect:diamond}.

An explicit PBW presentation of both  $S_q(V\otimes V)$ and $U(\iota)$ is more cumbersome, so we postpone it until 
Proposition~\ref{pr:SqPres}. Below we provide a presentation of $U(\iota)$ by a minimal set of Chevalley-like generators satisfying Serre-like relations. 

\begin{theorem}\label{th:Dn Chevalley} The uberalgebra $U(\iota)=S_q(V\otimes V)\rtimes U_q^+(\lie{sl}_n)$, $n\ge 2$ is generated by $u_1,\ldots,u_{n-1}$, $w$, and $z$ subject to the following relations (for all relevant $i,j$): 
\begin{align*}
&[u_i,[u_i,u_j]_q]_{q^{-1}}=0,\,\text{if $|i-j|=1$},\quad u_iu_j=u_ju_i,\, \text{if $|i-j|>1$},
\\
&u_i w=w u_i,\, \text{if $i\ne 1$},\quad u_i z=z u_i,\,\text{if $i\ne 2$}, \quad zw=wz,
\\
&[u_{1},[u_{1},[u_{1},w]]_{q^2}]_{q^{-2}}=0,\qquad [u_{2},[u_{2},z]_q]_{q^{-1}}=0,
\\
&[w,[w,u_{1}]_{q^2}]_{q^{-2}}=-h wz,\qquad [z,[u_2,[u_1,w]_{q^2}]_q]_q=[w,[u_1,[u_2,z]_q]_q]_{q^2},
\\
&2[z,[z,u_{2}]_q]_{q^{-1}}=h 
(z [u_{1},u_{2}]_q w+w [u_{2},u_{1}]_{q^{-1}} z+w u_{1}[z,u_{2}]_q
+[u_2,z]_{q^{-1}}u_1w),
\end{align*}
where $h=q-q^{-1}$ and we abbreviate $[a,b]_v=ab-v ba$ and $[a,b]=[a,b]_1=ab-ba$ (with the convention that $u_2=0$ if $n=2$). 
\end{theorem}

\begin{remark} \label{rem:various S_q}
Under the decomposition $V\otimes V=S^2_q V\oplus \Lambda^2_q V$ in the category of $U_q(\lie{sl}_n)$-modules the generator $w$ (respectively $z$) of $U(\iota)$  
is a lowest weight vector in the simple $U_q(\lie{sl}_n)$-module $S^2_q V$ (respectively $\Lambda^2_q V$), with the convention that $u_i$ equals the $(n-i)$th standard Chevalley generator $E_{n-i}$ of $U_q^+(\lie{sl}_n)$. 

It is easy to show that $S_q(V\otimes V)/\langle \Lambda_q^2 V\rangle \cong S_q(S^2 V)$, $S_q(V\otimes V)/\langle S_q^2 V\rangle \cong S_q(\Lambda^2 V)$, where $S_q(S^2 V)$ and $S_q(\Lambda^2 V)$ are respectively the algebras of quantum symmetric and quantum exterior matrices studied in~\cites{FRT,Nou,Kam,Str}. 
Due to this and the canonical identification $S(V\otimes V)=S(\Lambda^2 V\oplus S^2 V)=S(\Lambda^2 V)\otimes S(S^2 V)$, we can view $S_q(V\otimes V)$ as a deformation of the braided (in the category of $U_q(\lie{sl}_n)$-modules) tensor product $S_q(\Lambda^2 V)\otimes S_q(S^2 V)$  (see also Remark~\ref{rem:various S_q Poisson} for the Poisson version of this discussion). This point of view is  supported by the observation that our braiding operator $\Psi$ given by \eqref{eq:psi braiding} is a deformation of the braiding $\Psi':=\Psi_2\Psi_1\Psi_3\Psi_2$ of $V\otimes V$ with itself in the category of $U_q(\lie{sl}_n)$-modules. Note, however, that latter braiding $\Psi'$ does not satisfy the condition~\eqref{thm:flat quadratic algebra.ii} of Theorem~\ref{thm:flat quadratic algebra}, therefore, the quadratic algebra $S_{\Psi'}(V\otimes V)$ (as defined in Theorem~\ref{thm:flat quadratic algebra}) is not a flat deformation of $S(V\otimes V)$. 
\end{remark}

\begin{remark} In all quantum foldings we constructed so far the image of $\iota$ is contained in $U_q(\lieg)^{gr~\sigma}$, where $(\cdot)^{gr~\sigma}$ is the graded fixed point algebra defined for any graded algebra $A=\bigoplus_{\gamma\in  \Gamma} A_\gamma$ and any automorphism $\sigma$ of $A$ by:
$A^{gr~\sigma}=\textstyle\bigoplus_{\gamma\in \Gamma} \{a\in A_\gamma:\sigma(a)=a\}$ (in our case, $\Gamma$ is the root lattice of $\lieg$).
One can show that the subalgebra of $U_q^+(\lie{so}_{2n+2})$ generated by the image of $\iota:
U_q^+(\lie{sp}_{2n})\hookrightarrow U_q^+(\lie{so}_{2n+2})$ is isomorphic to $U_q^+(\lie{so}_{2n+2})^{gr~\sigma}$, 
but we do not expect that this to happen in general (e.g., it fails for the pair $(\lie g,\lie g^\sigma{}^\vee)=(\lie{so}_8,G_2)$). We will discuss the relationship between quantum foldings and graded fixed points of diagram auromorphisms in a separate publication.
\end{remark} 

Theorem \ref{th:first folding} implies that the ``classical limit'' $S(V\otimes V)$ of $S_q(V\otimes V)$ has a quadratic Poisson bracket which we present in the following
\begin{corollary} 
\label{cor:deformed Yakimov's bracket} In the notation of Theorem \ref{th:first folding}, let  $\{X_i\}$, $i=1,\ldots,n$ be the standard basis of~$V$. Then
the formulae (for all $1\le i\le j\le k\le l\le n$, where we abbreviated $X_{ij}=X_i\otimes X_j$ for $1\le i,j\le n$): 
\begin{align*}
&\{X_{ij},X_{kl}\}=(\delta_{ik}+\delta_{il}+\delta_{jk}+\delta_{jl}) X_{ij}
    X_{kl}-2 (X_{il} X_{kj}+X_{ki} X_{lj})\\
&\{X_{ij}, X_{lk}\}=(\delta_{ik}+\delta_{il}+\delta_{jk}+\delta_{jl}) X_{ij}
    X_{lk}-2 (X_{kj} X_{li}+ X_{ik}
    X_{lj})\\
&\{X_{ji},X_{kl}\}=(\delta_{ik}+\delta_{il}+\delta_{jk}+\delta_{jl}) X_{ji}
    X_{kl}-2 (X_{jl} X_{ki}+X_{kj} X_{li})\\
&\{X_{ji}, X_{lk}\}=(\delta_{ik}+\delta_{il}+\delta_{jk}+\delta_{jl})
    X_{ji} X_{lk}-2 (X_{jk}X_{li}+X_{ki} X_{lj})\\
&\{X_{ik},X_{jl}\}=(\delta_{ij}+\delta_{il}-\delta_{jk}+\delta_{kl}) X_{ik}
    X_{jl}-2 (X_{il} X_{jk}-X_{ij} X_{kl}+X_{ji}
    X_{lk})\\
&\{X_{ik},X_{lj}\}=(\delta_{ij}+\delta_{il}-\delta_{jk}+\delta_{kl}) X_{ik}
    X_{lj}-2 X_{jk}X_{li}\\
&\{X_{ki},X_{jl}\}=(\delta_{ij}+\delta_{il}-\delta_{jk}+\delta_{kl}) X_{jl}
    X_{ki}-2 X_{jk} X_{li}\\
&\{X_{ki}, X_{lj}\}=(\delta_{ij}+\delta_{il}-\delta_{jk}+\delta_{kl}) X_{ki}
    X_{lj}-2 X_{kj}X_{li}
\\
&\{X_{il},X_{jk}\}=(\delta_{ij}+\delta_{ik}-\delta_{jl}-\delta_{kl}) X_{il}
    X_{jk}+2( X_{ij} X_{lk}-X_{ji} X_{kl})\\
&\{X_{il},X_{kj}\}=(\delta_{ij}+\delta_{ik}-\delta_{jl}-\delta_{kl}) X_{il} X_{kj}+2(
    X_{ik} X_{lj}-X_{jl}
    X_{ki})\\
&\{X_{li},X_{kj}\}=(\delta_{ij}+\delta_{ik}-\delta_{jl}-\delta_{kl}) X_{kj}
    X_{li}\\
&\{X_{li},X_{jk}\}=(\delta_{ij}+\delta_{ik}-\delta_{jl}-\delta_{kl}) X_{jk}
    X_{li}
\end{align*}
define a Poisson bracket on $S(V\otimes V)$.
\end{corollary}

\begin{remark} \label{rem:various S_q Poisson}
In~\cite{GY} K.~Goodearl and M.~Yakimov constructed quadratic Poisson brackets on $S(\Lambda^2 V)$ and $S(S^2 V)$. 
In parallel with  Remark \ref{rem:various S_q}, one can show that the ideal of $S(V\otimes V)$ generated by $\Lambda^2 V=\operatorname{\operatorname{Span}}\{X_{ij}-X_{ji}\}$ (respectively by $S^2 V=\operatorname{\operatorname{Span}}\{X_{ij}+X_{ji}\}$)  is Poisson   hence the  quotient of $S(V\tensor V)$ by this ideal is the 
Poisson algebra $S(S^2 V)$ (respectively $S(\Lambda^2 V)$) from \cite{GY}. Therefore, we can view the bracket given by Corollary \ref{cor:deformed Yakimov's bracket}
as a certain deformation of 
the  Poisson bracket on $S(V\otimes V)$ obtained by lifting  the brackets on $S(\Lambda^2 V)$ and $S(S^2 V)$. 
\end{remark}

We construct more liftable quantum foldings when $\sigma$ is an involution.
\begin{theorem} 
\label{th:first folding 2}
If $(\lieg,\lieg^\sigma{}^\vee)=(\lie{sl}_n\times \lie{sl}_n,\lie{sl}_n)$, $n=3,4$, then there exists a tame liftable quantum folding $\iota:U_q^+(\lieg^\sigma{}^\vee)\hookrightarrow
U_q^+(\lieg)$ 
such that $U(\iota)$ is a $q$-deformation of the universal enveloping algebra $U(V_n\rtimes (\lie{sl}_n)_+)$, $n=3,4$, 
where $V_n$ is a finite-dimensional module (regarded as an abelian Lie algebra) over $(\lie{sl}_n)_+$. More precisely,

\begin{enumerate}[{\rm(i)}]

\item\label{th:first folding 2.i}  For $(\lie{sl}_3\times \lie{sl}_3,\lie{sl}_3)$, $V_3={\bf 1}$ is the trivial one-dimensional $(\lie{sl}_3)_+$-module and
the uberalgebra $U(\iota)$ is generated by $u_1,u_2$, and $z$ subject to the following relations

$\bullet$ $z$ is central;

$\bullet$  $u_i^2 u_j-(q^2+q^{-2})u_i u_j u_i+u_j u_i^2=(q-q^{-1})u_i z$ for  $\{i,j\}=\{1,2\}$. 

\item\label{th:first folding 2.ii}
For $(\lie{sl}_4\times \lie{sl}_4,\lie{sl}_4)$,   the $(\lie{sl}_4)_+$-module $V_4$ has a basis $z_{12},z_{13},z_{23},z_{1,23},z_{12,3},z_{12,23}$, the action of Chevalley generators $e_1,e_2,e_3$ of $(\lie{sl}_4)_+$ on $V_4$ is given by the following
diagram
$$
\divide\dgARROWLENGTH by 3
\begin{diagram}
\node{z_{12}}\arrow{s,l}{e_3}\node{z_{13}}\node{z_{23}}\arrow{s,r}{e_1}\\
\node{z_{12,3}}\arrow{se,r}{e_2}\node{}\node{z_{1,23}}\arrow{sw,r}{e_2}\\
\node{}\node{z_{12,23}}
\end{diagram}
$$
where an arrow from $z$ to $z'$ labeled by $e_i$ means that $e_i(z)=z'$, while $e_j(z)=0$ for all $j\not=i$. The uberalgebra $U(\iota)$ is a quantized enveloping algebra of the Lie algebra $V_4\rtimes (\lie{sl}_4)_+$ and it is generated by $u_1,u_2,u_3$, $z_{12}=z_{21}$, $z_{23}=z_{32}$, and $z_{13}$ subject to the relations:
\begin{itemize}
\item $u_iz_{ij}=z_{ij}u_i$,  $i< j$, $u_1u_3=u_3u_1$,

\item $[u_i,[u_i,u_j]_{q^2}]_{q^{-2}}=hu_i z_{ij}$,  $|i-j|=1$, 

\item $[u_i,[u_i,z_{j2}]_{q^2}]_{q^{-2}} =hu_i z_{13}$ for $\{i,j\} = \{1,3\}$, 

\item $(q+q^{-1}) [z_{12},z_{23}]=[u_2,[z_{12},u_3]_{q^{-2}}]_{q^2}-[u_2,[z_{23},u_1]_{q^{-2}}]_{q^2}$,

\item $[u_2,z_{13}]+[z_{i2},z_{j2}]=h(z_{j2}u_i u_2-u_2 u_i z_{j2})$ for $\{i,j\} = \{1,3\}$,

\item $2[z_{i2},[z_{i2},u_j]_{q^{-2}}]_{q^2}+[z_{i2},[u_i,z_{2j}]_{q^{-2}}]_{q^2}+[u_i,[z_{i2},z_{2j}]_{q^{-2}}]_{q^2}=$

$=h \{z_{i2},z_{13}+(q^2+1+q^{-2})u_i u_2 u_j-u_j u_2 u_i-u_iu_ju_2-u_2u_iu_j\}$ for $\{i,j\}=\{1,3\}$, 

\item $
[z_{i2},z_{13}]-[u_i,[u_2,[u_j,z_{i2}]]]=h(u_i z_{j2}z_{i2}-z_{i2}z_{j2}u_i)+h^2(u_i u_2 u_j z_{i2}-z_{i2}u_j u_2 u_i)
$
\end{itemize}
for $\{i,j\}=\{1,3\}$, where we abbreviated $[a,b]_v=ab-vba$,  $[a,b]=ab-ba$, $\{a,b\}=ab+ba$, and $h=q-q^{-1}$.
\end{enumerate}
\end{theorem}

\begin{remark} In case of $(\lie{sl}_3\times \lie{sl}_3,\lie{sl}_3)$ the uberalgebra $U(\iota)$ is PBW on the ordered set $\{u_1,u_2,u_{21}=u_1u_2-q^{-2}u_2u_1-z,z\}$
subject to the following relations

\begin{itemize}

\item  the element $z$ is central, 
\item  $u_1 u_{21}=q^{2} u_{21}u_1$,  $u_2 u_{21}=q^{-2} u_{21}u_2$,
\item  $u_1 u_2=q^{-2}u_2 u_1+u_{21}+z$.
\end{itemize}

In particular, $S({\bf 1}\rtimes (\lie{sl}_3)_+)$ is generated by $\tilde u_1,\tilde u_2,\tilde u_{12},\tilde z$ and the following

\begin{itemize}
 \item 
$\{\tilde u_1,\tilde z\}=\{\tilde u_2,\tilde z\}=\{\tilde u_{21},\tilde z\}=0$,
 
\item $\{\tilde u_1,\tilde u_2\}=-2 \tilde u_2\tilde u_1 +4 \tilde u_{21}+2 \tilde z$,

\item  $\{\tilde u_1,\tilde u_{21}\}=2 \tilde u_{21} \tilde  u_1,~\{\tilde u_2,\tilde u_{21}\}=-2 \tilde u_{21}\tilde u_2$
\end{itemize}
defines a  Poisson bracket on $S({\bf 1}\rtimes (\lie{sl}_3)_+)$.

\end{remark}
  
The PBW-presentation of the uberalgebra $U(\iota)$ for the folding $(\lie{sl}_4\times\lie{sl}_4,\lie{sl}_4)$ 
is more cumbersome (see Theorem~\ref{thm:sl4 diag}). Similarly to the previous discussion, the PBW property of $U(\iota)$ defines a Poisson bracket  on $S(V_4\rtimes (\lie{sl}_4)_+)$  which, unlike that on  $S(V_3\rtimes (\lie{sl}_3)_+)$, includes cubic terms (Theorem~\ref{thm:sl4 diag Poisson}). It would be interesting to construct both the uberalgebra and the corresponding Poisson bracket for the folding $(\lie{sl}_n\times \lie{sl}_n,\lie{sl}_n)$, $n\ge 4$. 

Now we will explicitly construct all tame liftable quantum foldings $\iota$  used in Theorems \ref{th:first folding} and \ref{th:first folding 2} along with their (yet conjectural) generalizations to all semisimple Lie algebras. We need some notation.

Given a semisimple simply laced Lie algebra $\lieg$ with an admissible diagram automorphism~$\sigma$, let $I$ be the set of vertices of the Dynkin diagram of $\lieg$ and  we denote by the same letter $\sigma$ 
the induced bijection $\sigma:I\to I$. 
Denote by $s_i$,  $i\in I$ (respectively, by $s'_r$, $r\in I/\sigma$) the simple reflections  
of the root lattice of $\lieg$  (respectively, of ${\lieg^\sigma}^\vee$).  Let 
$W(\lieg)=\langle s_i: i\in I\rangle$ (respectively, $W({\lieg^\sigma}^\vee)=\langle s'_r: r\in I/\sigma\rangle$) be the corresponding Weyl group.

Denote by $\hat w_\circ$ (respectively, $w_\circ$) the longest element of $W(\lieg)$ (respectively, of $W({\lieg^\sigma}^\vee)$). Furthermore, denote by $R(w_\circ)$ the set of all reduced decompositions of $w_\circ$ i.e. of all sequences $\ii=(i_1,\ldots,i_m)\in (I/ \sigma)^m$ where $m=\ell(w_\circ)$  is the Coxeter length of $w_\circ$ such that $s_{i_1}\cdots s_{i_m}=w_\circ$. Similarly, one defines the set $R(\hat w_\circ)$ of all reduced decompositions of $\hat w_\circ$.

Note that each admissible diagram automorphism $\sigma$ defines an automorphism of  $W(\lieg)$ via $s_i\mapsto s_{\sigma(i)}$ and its fixed subgroup $W(\lieg)^\sigma$ is isomorphic to  
$W(\lieg^\sigma)= W({\lieg^\sigma}^\vee)$ via  $s_r\mapsto \hat s_r=\prod_{i\in {\mathcal O}_r} s_i$ where  ${\mathcal O}_r\subset I$ is the $r$-th $\sigma$-orbit in $I$ (see Proposition~\ref{prop:WeylArtin folding}). We denote this natural  isomorphism $W({\lieg^\sigma}^\vee)\widetilde \to W(\lieg)^\sigma$ by $w\mapsto \hat w$. 
 
Thus, one can assign to each $\ii=(i_1,\ldots,i_m)\in R(w_\circ)$   its {\it lifting}  $\hat{\ii}\in R(\hat w_\circ)$ via:
$$\hat \ii=({\mathcal O}_{i_1},\ldots,{\mathcal O}_{i_m})$$ 
(in fact, $\hat \ii$ is unique up to reordering of each set ${\mathcal O}_{r_k}$).

Following Lusztig (\citem{Lus}*{\S40.2}), for each $\hat \ii\in R(\hat w_\circ)$ (respectively, $\ii\in R(w_\circ)$ one defines a {\it modified} PBW-basis $M(X_{\hat \ii})$ of $U_q^+(\lieg)$ (respectively, $M(X_{\ii})$ of $U_q^+({\lieg^\sigma}^\vee$)), see Section~\ref{subs:PBW bases} for details (this modification will
ensure the commutativity of the triangle in~\eqref{eq:cone over iota}).

One can show (see Lemma~\ref{le:iota PBW}) 
that for any  $\hat \ii\in R(\hat w_\circ)$ the PBW basis $M(X_{\hat \ii})$ does not depend on the choice of a lifting $\hat \ii\in R(\hat w_\circ)$ of $\ii\in R(w_\circ)$. Moreover, the action of $\sigma$ on $U_q^+(\lieg)$ preserves $M(X_{\hat \ii})$ for each such lifting $\hat \ii$.

The following result serves as a definition of  quantum folding for all $\lieg$ and $\sigma$ (see Lemma \ref{le:iota PBW} for details). 

\begin{proposition} 
\label{pr:general folding}
Given an admissible diagram automorphism $\sigma$ of $\lieg$, for each  $\ii\in R(w_\circ)$ there is a natural injective $\CC(q)$-linear map
\begin{equation}
\label{eq:generalized quantum folding}
\iota_\ii:U_q^+({\lieg^\sigma}^\vee)\hookrightarrow U_q^+(\lieg)^\sigma\subset U_q^+(\lieg)
\end{equation}
which maps the modified PBW-basis $M(X_\ii)$ bijectively onto the fixed point set $M(X_{\hat \ii})^\sigma$ of~$M(X_{\hat \ii})$. 
\end{proposition}

In fact, the tame liftable foldings $\iota$ used in Theorems~\ref{th:first folding} and~\ref{th:first folding 2} were of the form $\iota_\ii$, $\ii\in R(w_\circ)$.

\begin{theorem} \label{th:folding ii}
Let $\lieg$ be a simply laced semisimple Lie algebra and let $\sigma$ be its admissible diagram automorphism of order $2$. Then
for any reduced decompositions $\ii,\ii'$ of $w_\circ$ the subalgebras of $U_q^+(\lieg)$ generated by the images of $\iota_\ii$ and $\iota_{\ii'}$  are isomorphic.
\end{theorem} 

This theorem is proved in~Section~\ref{subs:PBW bases}.

However, if the order of $\sigma$ is at least $3$, it frequently happens that the image of $\iota$ generates a non-sub-PBW algebra hence the uberalgebra $U(\iota_\ii)$ does not always exists (see Section~\ref{subs:inf generated A(z)}). In order to restore the (sub-)PBW behavior of the algebras in question, we propose the modification, which we  refer to as the {\it enhanced} uberalgebra $\hat U(\iota)$.    

Indeed, in the assumptions of Definition \ref{def:liftable} let us relax the assumption that $Z_0\subset \langle A\rangle_\iota$ in Definition \ref{def:liftable iota}.
Suppose that $B$ is  PBW domain. Then we take $Z_0$ to be a finite subset of $\Frac(\iota(A))\cap B$, where $\Frac(\iota(A))\subset \Frac(B)$ is the skew-subfield of the skew-filed $\Frac(B)$ generated by $\iota(A)$ ($B$ is an Ore domain so its skew-field of fractions $\Frac(B)$ is well-defined, see \citem{BZ1}*{Appendix~A} for details). 

We will refer to a map $\iota$ satisfying  Definition \ref{def:liftable iota} ``relaxed'' in such a way as {\it enhanced} liftable and to its uberalgebra (which we denote by $\hat U(\iota)$) as an {\it enhanced} uberalgebra of $\iota$. (A {\it tame} enhanced liftable $\iota$ is  introduced accordingly).
By construction, $\hat U(\iota)$ satisfies the diagram \eqref{eq:cone over iota}, however, it need not be
generated by $A$ (unlike all known $U(\iota)$ for liftable~$\iota$).

\begin{theorem} 
\label{th:second folding A_2 cube} Let $n\ge 3$ and let $(\lieg,{\lieg^\sigma}^\vee)=(\lie{sl}_3^{\times n},\lie{sl}_3)$ where $\lie{sl}_3^{\times n}=\lie{sl}_3\times \cdots \times \lie{sl}_3$  and $\sigma$ is a cyclic permutation of factors. Then for both reduced decompositions $\ii_1=(121)$ and  $\ii_2=(212)$ of  $w_\circ\in W({\lieg^\sigma}^\vee)$   the quantum folding $\iota_{\ii_r}$, $r=1,2$ is enhanced liftable and the enhanced uberalgebras  $\hat U(\iota_{\ii_1})$ and  $\hat U(\iota_{\ii_2})$ are isomorphic. More precisely, 
\begin{enumerate}[{\rm(i)}]

\item \label{th:second folding A_2 cube.i}
$\hat U(\iota_{\ii_1})$ is generated by Chevalley-like generators $u_1,u_2$, and $z_1,\ldots,z_{n-1}$, subject to Serre-like relations

$\bullet$ $z_k z_l=z_l z_k$ for $k,l=1,\ldots,n-1$,

$\bullet$ $u_i z_{k,i}=q^{n-2k} z_{k,i} u_i$ for $i=1,2$, $k=1,\ldots,n-1$,

$\bullet$ $u_i^2 u_j-(q^n+q^{-n})u_i u_j u_i+u_j u_i^2=(q^{-1}-q) u_i \sum_{k=1}^{n-1} q^k z_{k,i}$ 
 for $\{i,j\}=\{1,2\}$

\noindent
where we abbreviated $z_{k,1}=z_{n-k,2}=z_k$.

\item \label{th:second folding A_2 cube.ii}
The enhanced uberalgebra $\hat U(\iota_{i_1})$ is a  PBW algebra in the totally ordered set of generators  $\{u_2, u_{21},u_1, z_1,\ldots,z_{n-1}\}$,  
where $u_{21}=u_1 u_2-q^{-n} u_2 u_1-\sum\limits_{k=1}^{n-1} \frac{q-q^{-1}}{q^k-q^{-k}} z_k$,   subject to the commutation relations:

$\bullet$ $z_k z_l=z_l z_k$ for all $1\le k,l\le n-1$,

$\bullet$ $u_1 u_{21}=q^n u_{21} u_1,u_2 u_{21}=q^{-n} u_{21}u_2$, 

$\bullet$ $u_1 z_k=q^{n-2k} z_{k}u_1$, $u_2 z_k=q^{2k-n} z_k u_2$, $u_{21} z_k=z_k u_{21}$ for all $1\le k\le n-1$,

$\bullet$ $u_1 u_2=q^{-n} u_2 u_1+u_{21}+\sum\limits_{k=1}^{n-1} \displaystyle\frac{q-q^{-1}}{q^k-q^{-k}}\,z_k$.

\item\label{th:second folding A_2 cube.iii}
$\hat U(\iota_{\ii_1})$ a quantum deformation of the enveloping algebra  $U(\mathbf 1^{n-1}\rtimes (\lie{sl}_3)_+)$, where  $\mathbf 1$ is the trivial one-dimensional $(\lie{sl}_3)_+$-module.
\end{enumerate}
\end{theorem}

We prove Theorem~\ref{th:second folding A_2 cube} in Section~\ref{subs:diag sl3}.

\begin{remark} It follows from Theorem~\ref{th:second folding A_2 cube} that the following defines a Poisson bracket on $S(\mathbf 1^{n-1}\rtimes (\lie{sl}_3)_+)$ 

$\bullet$ $\{\tilde u_1,\tilde u_{21}\}=n \tilde u_1 \tilde u_{21},~\{u_2,\tilde u_{21}\}=-n \tilde u_1 \tilde u_{21}$,

$\bullet$ $\{\tilde u_1,\tilde z_k\}=(n-2k) \tilde u_1 \tilde z_k,~\{\tilde u_2,\tilde z_k\}=(k-2n) \tilde u_2 \tilde z_k$ for $k=1,\ldots,n-1$,

$\bullet$ $\{\tilde u_1,\tilde u_2\}=n(2 \tilde u_{21}-\tilde u_1 \tilde u_2)+2\sum_{k=1}^{n-1}  \tilde z_k$,

\noindent
where $\tilde u_1$,$\tilde u_{21}$, $\tilde u_2$, and $\tilde z_k$, $k=1,\ldots,n-1$ are PBW generators of $S(\mathbf 1^{n-1}\rtimes (\lie{sl}_3)_+)$ obtained by certain specialization at $q=1$ from generators of $U(\iota_{\ii_1})$. Note that 
the quotient by the Poisson ideal generated by $\tilde z_1,\ldots,\tilde z_{n-1}$ is the Poisson algebra $S(\lie{sl}_3^+)$ with the standard Poisson bracket multiplied by~$n$.

\end{remark}

\begin{theorem} 
\label{th:second folding G2} 
Let $(\lieg,{\lieg^\sigma}^\vee)=(\lie{so}_8,G_2)$ where $\sigma$ is a cyclic permutation of $3$ vertices of Dynkin diagram of type $D_4$.  Then for both reduced decompositions $\ii_1=(121212)$ and  $\ii_2=(212121)$ of  $w_\circ\in W({\lieg^\sigma}^\vee)$  the quantum folding $\iota_{\ii_k}$, $k=1,2$ is enhanced liftable and the enhanced uberalgebras  $\hat U(\iota_{\ii_1})$ and  $\hat U(\iota_{\ii_2})$ are isomorphic to each other and to a  quantum deformation of  the universal enveloping algebra $U(\nn_{G_2}\rtimes (\lie{sl}_2)_+)$, where $\nn_{G_2}$ is a certain nonabelian nilpotent $13$-dimensional Lie algebra with the covariant $(\lie{sl}_2)_+$-action. More precisely, 

\begin{enumerate}[{\rm(i)}]
\item\label{th:G2.i} $\nn_{G_2}\rtimes (\lie{sl}_2)_+$ is generated by $u=e_1$, $w$, $z_1$, $z_2$ subject to the following relations 
\begin{itemize}
\item $[u,[u,[u,[u,w]]]]=[w,[w,u]]=0$; 
\item $[u,[u,z_i]]=[z_i,[z_i,[z_i,u]]]=[w,z_i]=0$, $[w,[z_i,u]]=[z_1,z_2]$ for $i=1,2$,
\item $[z_i,[u,z_i]]=[z_1,[u,z_2]]+[z_2,[u,z_1]]$ for $i=1,2$.
\end{itemize}

\item\label{th:G2.ii} $\nn_{G_2}$ is the Lie ideal in $\nn_{G_2}\rtimes (\lie{sl}_2)_+$ with the basis $w_i$, $1\le i\le 5$ and $z_i$, $1\le i\le 8$ and
the multiplication table (only non-zero Lie brackets are shown)

\begin{itemize}
\item $[w_1,w_4]=-3 w_5$,  $[w_2,w_3]=w_5$, 
\item $[w_1,z_3]=[w_1,z_4]=3z_5,~[w_2,z_1]=[w_2,z_2]=[z_2,z_1]=- z_5$,
\item $[w_2,z_3]=[w_2,z_4]=[z_1,w_3]=[z_2,w_3]=2z_6$,
\item $[w_3,z_3]= [w_3,z_4]=[z_3,z_4]=z_7$,~$[w_4,z_1]=[w_4,z_2]=-3 z_7$,
\item $[z_1,z_3]=[z_2,z_4]=2 z_8$,
\item $[z_1,z_4]=z_6+z_8$, $[z_2,z_3]=-z_6+z_8$.
\end{itemize}

\item\label{th:G2.iii} $\hat U(\iota_{\ii_1})$ is generated by Chevalley-like generators $u,w$, and $z_1,z_2$ and satisfies the following Serre-like relations (the list is incomplete):
\begin{itemize}
\item $[u,[u,[u,[u,w]_{q^{-3}}]_{q^{-1}}]_q]_{q^3}=0$
\item $[w,[w, u]_{q^{-3}}]_{q^3}=[w,z_1]_{q^3}=[z_2,w]_{q^3}=
q([z_1,w]_{q}+[w,z_2]_{q})$
\item $[z_1,z_2]=(q+q^{-1}) [z_1,[w,u]_{q^{-3}}]-[z_2,[w,u]_{q^{-3}}]$,
\item $[z_1,[u,w]_q]=[[w,u]_q,z_2]$,~$[z_1,[z_1,u]_{q^{-1}}]_q=[z_2,[z_2,u]_{q^{-1}}]_q$,
\item $[z_1,[z_1,u]_{q^{-1}}]_q=q ([z_1,[z_2,u]_{q^{-1}}]_q+[z_2,[z_1,u]_{q^{-1}}]_q)$

$+(q-q^{-1})(q [z_1,u]_{q^{-1}} z_2-z_1 [z_2,u]_{q^{-1}})$
\end{itemize}
\end{enumerate}

\end{theorem}

We prove Theorem \ref{th:second folding G2} in Section \ref{sect:G2}. 

\begin{remark}
The non-tameness of the quantum folding assigned to $(\lie{so}_4,G_2)$ causes serious computational problems for the corresponding uberalgebra and the Poisson bracket on $S(\nn_{G_2})$. At the moment the Poisson bracket involves around $700$ terms and the PBW presentation of $U(\iota_{\ii_1})$ is even more complicated 
(they can be found at \href{http://ishare.ucr.edu/jacobg/G2.pdf}{http://ishare.ucr.edu/jacobg/G2.pdf}).
 
This is one of the reasons why Theorem \ref{th:second folding G2}(iii) contains  only a partial Serre-like presentation of $U(\iota_{\ii_1})$ in Chevalley-like generators $u,w,z_1,z_2$. We dropped here the most notorious relations involving more than 30 terms each (see the above mentioned webpage).

\end{remark}

Taking into account  Theorems \ref{th:first folding},  \ref{th:folding ii}, \ref{th:second folding A_2 cube},  and \ref{th:second folding G2}, we we propose the following conjecture.

\begin{conjecture} 
\label{conj:n_w >2}
Let $\sigma$ be any admissible diagram automorphism  of $\lieg$  such that ${\lieg^\sigma}^\vee$ has no Lie ideals of type $G_2$. Then there exists a (unique)  $\lieg_+$-module $V_\lieg$ such that:

\begin{enumerate}[{\rm(i)}]
 \item\label{conj:n_w >2.i} 
for any   $\ii\in R(w_\circ)$, the folding $\iota_\ii$ is tame enhanced liftable, 

\item\label{conj:n_w >2.ii} the corresponding enhanced uberalgebra $\hat U(\iota_\ii)$ is a flat deformation of both the universal enveloping algebra $U(\nn\rtimes \lieg_+)$ and the symmetric algebra $S(V_\lieg\rtimes \lieg_+)$,

\item\label{conj:n_w >2.iii} The skew field of fractions $\Frac(U(\iota_\ii))$ is generated by $\tilde \iota_\ii(E_r)$, $r\in I/\sigma$, where $E_r$ are Chevalley generators of $U_q^+({\lieg^\sigma}^\vee)$  (and $\tilde \iota_\ii:U_q^+({\lieg^\sigma}^\vee)\hookrightarrow U(\iota_\ii)$ is the lifting of $\iota_\ii$ given by \eqref{eq:cone over iota}).
\end{enumerate}
\end{conjecture}

If $\sigma$ is an involution, we drop ``enhanced'' in Conjecture \ref{conj:n_w >2} because we expect that $\hat U(\iota_\ii)=U(\iota_\ii)$.

In particular, the conjecture implies that one can canonically assign to each simply laced Lie algebra $\lieg$ a finite-dimensional 
$\lieg_+$-module $V_\lieg^{(k)}$ for each $k\ge 2$ (by taking $\lieg^{\times k}$ and its natural diagram automorphism $\sigma$, the cyclic permutation of factors so that  ${(\lieg^{\times k})^\sigma}^\vee=\lieg)$). 
Theorem~\ref{th:first folding 2} implies that such a $V_\lieg^{(k)}$ will be rather non-trivial even for $\lieg=\lie{sl}_n$. It would be also 
interesting to explicitly compute the Poisson bracket on $S(V_\lieg^{(k)}\rtimes \lieg_+)$ predicted by Conjecture \ref{conj:n_w >2}.
It should be noted that if $\lieg$ has a diagram automorphism $\sigma'$, then the corresponding uberalgebra also admits an automorphism extending~$\sigma'$.
For example, in the notation of Theorems~\ref{th:second folding A_2 cube} and~\ref{th:first folding 2},
the uberalgebra for the folding $(\lieg,{\lieg^\sigma}^\vee)=(\lie{sl}_3^{\times k},\lie{sl}_3)$ has an automorphism $\sigma'$ defined by
$u_1\mapsto u_2$, $u_2\mapsto u_1$, $z_i\mapsto z_{k-i}$, $1\le i\le k-1$, while the uberalgebra for the 
folding $(\lie{sl}_4\times\lie{sl}_4,\lie{sl}_4)$ has an automorphism defined by $e_1\mapsto e_3$, 
$z_{12}\mapsto z_{32}$, $z_{32}\mapsto z_{12}$ and $e_2$, $z_{13}$ are fixed.

Note also that the part~\eqref{conj:n_w >2.iii} of Conjecture \ref{conj:n_w >2} holds for all cases we considered so far,  in particular, for the folding $(\lieg,{\lieg^\sigma}^\vee)=(\lie{sl}_3^{\times n},\lie{sl}_3)$, the skew-field $\Frac(U(\iota))$ is generated by $u_1$ and $u_2$ (one can show that each $z_k$, $1\le k\le n-1$ in Theorem \ref{th:second folding A_2 cube}\eqref{th:second folding A_2 cube.i} is a rational ``function" of $u_1$ and $u_2$; see Lemma~\ref{lem:Z0 fract}) and for $(\lieg,{\lieg^\sigma}^\vee)=(\lie{so}_8,G_2)$  the skew-field $\Frac(U(\iota))$ is generated by $u$ and $w$ (both $z_1$ and $z_2$ in Theorem \ref{th:second folding G2}\eqref{th:G2.iii} are rational ``functions" of $u$ and $w$). 

{\bfseries Acknowledgments.}
An important part of this work was done while both authors were visiting the University of Geneva and it is our pleasure to
thank Anton Alekseev for his hospitality. The authors thank Bernhard Keller, Bernard Leclerc, Nicolai Reshetikhin, Olivier Schiffmann 
and Milen Yakimov for stimulating discussions.

\section{General properties of quantum foldings and PBW algebras}\label{sect:prelim}

\subsection{Folding of semisimple Lie algebras}
\label{subsect:classical folding}

Recall that each semisimple Lie algebra $\lieg=\langle e_i,f_i\,:\,i\in I\rangle$ is determined by its Cartan matrix $A=(a_{ij})_{i,j\in I}$ (see e.g.~\cite{Serre}) via:
$$
(\ad e_i)^{1-a_{ij}}e_j=0=(\ad f_i)^{1-a_{ij}}f_j,\qquad [e_i,f_j]=0,\qquad i\not=j
$$
and
$$
[[e_i,f_i],e_j]=a_{ij}e_j,\qquad [[e_i,f_i],f_j]=-a_{ij} f_j,\qquad i,j\in I
$$
Denote by $\lieg_+$ the Lie subalgebra of~$\lieg$ generated 
by the $e_i$, $i\in I$.

We say that a bijection 
$\sigma:I\to I$ is a diagram automorphism of~$\lieg$ if $a_{\sigma(i),\sigma(j)}=a_{ij}$ for all $i,j\in I$. It is well-known that such $\sigma$ defines a unique  automorphism, which we also denote by $\sigma$, of the Lie algebra $\lieg$ via 
$$\sigma(e_i)=e_{\sigma(i)},~\sigma(f_i)=f_{\sigma(i)},\qquad i\in I.$$
After~\citem{Lus}*{\S12.1.1}, a diagram automorphism $\sigma$ is said to be {\em admissible} if for all $i\in I$, $k\in\ZZ$, $a_{i,\sigma^k(i)}=0$, whenever $\sigma^k(i)\ne i$. 

In what follows we denote by $I/\sigma$ the quotient set of $I$ by the equivalence relation which consists of all pairs $(i,\sigma^k(i))$. In other words, we use $I/\sigma$ as the indexing set for orbits of the cyclic group $\langle \sigma \rangle=\{1,\sigma,\sigma^2,\ldots\}$ action on $I$.

The following result is well-known (cf. for example
\citem{Kac}*{Proposition~7.9})
\begin{theorem}\label{th:classical folding}
Let $\sigma$ be an admissible diagram automorphism of~$\lie g$. Then the fixed Lie subalgebra $\lieg^\sigma=\{x\in\lieg\,:\, \sigma(x)=x\}$ 
of $\lieg$ is semi-simple, with:
\begin{itemize}
\item the Chevalley generators
 $e_r',f_r'$, $r\in I/\sigma$ 
given by
$$e_r'= \sum_{i\in {\mathcal O}_r} e_i, \quad f_r'= \sum_{i\in  {\mathcal O}_r} f_i\ ,$$
where $ {\mathcal O}_r$ is the $r$-th orbit of the $\langle \sigma \rangle$-action on $I$.

\item the Cartan matrix $A'=(a'_{r,s})$, $r,s\in I/\sigma$ given by
\begin{equation}
\label{eq:folded Cartan matrix}
a_{r,s}'=\sum_{i\in  {\mathcal O}_r} a_{i,j}
\end{equation} 
for all  $j\in  {\mathcal O}_s$, $r,s\in I/\sigma$. 

\end{itemize}
\end{theorem}

\subsection{Quantized enveloping algebras and Langlands dual folding}
\label{subsect:quantum enveloping}
For any indeterminate $v$  and for any $m\le n\in\ZZ_{\ge 0}$, set
$$
[n]_v:=\frac{v^n-v^{-n}}{v-v^{-1}},\qquad [n]_v!:=\prod_{j=1}^n [j]_v,\qquad \qbinom[v]{n}{m}:=\frac{[n]_v!}{[m]_v![n-m]_v!}.
$$

For each semisimple Lie algebra $\lieg$ we fix symmetrizers $d_i\in \NN$ such that $d_ia_{ij}=a_{ij}d_j$ for all $i,j\in I$. 
Then denote by $C=(d_ia_{ij})$ the {\it symmetrized Cartan matrix} of $\lieg$ (it depends on the choice of symmetrizers) and let $q_i:=q^{d_i}$.

Let  $U_q(\lieg)$ be the quantized universal enveloping algebra of $\lieg$ which is a  $\CC(q)$-algebra generated by the elements 
$E_i,F_i,K_i^{\pm 1}$, $i\in I$ subject to the relations
$$
[E_i,F_j]=\delta_{ij}\frac{K_i-K_{i}^{-1}}{q_i-q_i^{-1}},\qquad 
K_i E_jK_i^{-1}=q_i^{a_{ij}} E_j,\qquad K_i F_jK_i^{-1}=q_i^{-a_{ij}} F_j,
$$
as well as {\em quantum Serre relations}
\begin{equation}\label{q.Serre}
\sum_{b=0}^{1-a_{ij}} (-1)^r \qbinom[q_i]{1-a_{ij}}{b} E_i^r E_j^{} E_i^{1-a_{ij}-b}=0=\sum_{b=0}^{1-a_{ij}} (-1)^r \qbinom[q_i]{1-a_{ij}}{b} F_i^r F_j^{} F_i^{1-a_{ij}-b}
\end{equation}
for all $i\not=j$.

We denote by $U_q^+(\lieg)$ (resp. by $U_q^{\le 0}(\lie g)$) the  subalgebra of $U_q(\lieg)$ generated by the $E_i$, $i\in I$. 
(resp. by the $F_i,K_i^{\pm 1}$, $i\in I$).
Note that $U_q(\lieg)$ and $U_q^+(\lieg)$ are completely determined by the symmetrized Cartan matrix $C$.

We now define the folding of symmetrized Cartan matrices for a given admissible diagram automorphism $\sigma$.
For each $I\times I$ symmetric matrix $C$ and a bijection $\sigma:I\to I$
denote by $C^\sigma=(c_{r,s}^\sigma)$ the $I/\sigma\times I/\sigma$ symmetric matrix with the entries:
\begin{equation}
\label{eq:folded symmetrized Cartan matrix} c_{r,s}^\sigma=\sum_{i\in  {\mathcal O}_r,j\in {\mathcal O}_s} c_{i,j}
\end{equation} 
for all  $j\in  {\mathcal O}_s$, $r,s\in I/\sigma$. 

\begin{lemma} 
\label{le:Langlands dual}
Let $C=A$ be the Cartan matrix of a simply-laced semisimple Lie algebra~$\lieg$ with an admissible diagram automorphism~$\sigma$. Then $C^\sigma$ is a symmetrized 
Cartan matrix of ${\lieg^\sigma}^\vee$, where ${\lieg^\sigma}^\vee$ is the Langlands dual Lie algebra of the semisimple Lie algebra $\lieg^\sigma$. More precisely, $C^\sigma=D^\sigma(A')^T$ where $A'$ is the Cartan matrix of $\lieg^\sigma$ (given by \eqref{eq:folded Cartan matrix}) and $D^\sigma$ is the diagonal matrix $\operatorname{diag}(|{\mathcal O_r}|,r\in I/\sigma)$.
\end{lemma}  
\begin{proof}
By~\eqref{eq:folded Cartan matrix}, \eqref{eq:folded symmetrized Cartan matrix} and the symmetry of~$A$ we have for all $r,s\in I/\sigma$
\begin{equation*}
c^\sigma_{r,s}=\sum_{i\in\mathcal O_r,j\in\mathcal O_s} a_{i,j}=\sum_{i\in \mathcal O_r}\sum_{j\in\mathcal O_s} a_{j,i}=
\sum_{i\in\mathcal O_r} a'_{s,r}=
|\mathcal O_r| a'_{s,r}.\qedhere
\end{equation*}
\end{proof}

This motivates the following notation. For each $\lieg$ and $\sigma$ as above denote by $U_q({\lieg^\sigma}^\vee)$ the quantized enveloping algebra determined by the matrix $C^\sigma$ from Lemma~\ref{le:Langlands dual}.

\subsection{Braid groups and their folding}
\label{subs:braid group folding}
Given a semisimple Lie algebra $\lieg$ with the Cartan matrix $A=(a_{ij})_{i,j\in I}$, let $Q=\bigoplus_{i\in I} \ZZ\alpha_i$ be the root lattice of $\lieg$. Recall that the {\it Weyl group} $W(\lieg)$ is generated by the simple reflections $s_i:Q\to Q$ given by:
$$s_i(\alpha_j)=\alpha_j-a_{ij}\alpha_i$$
for $i,j\in I$. It is well-known that $W(\lieg)$ is a Coxeter group with the presentation 

\begin{equation}
\label{eq:coxeter relation}
(s_is_j)^{m_{ij}}=1,  ~\text{where}~ m_{ij}=\begin{cases} 
1 & \text{if $i=j$}\\ 
2 & \text{if $a_{ij}=0$}\\ 
3 & \text{if $a_{ij}=a_{ji}=-1$}\\
4 & \text{if $a_{ij}a_{ji}=2$}\\
6 & \text{if $a_{ij}a_{ji}=3$}
\end{cases}.
\end{equation}
For each $w\in W(\lie g)$ denote by~$R(w)$ the set of all reduced decompositions $\ii=(i_1,\ldots,i_m)\in I^m$ such that
$$w=s_{i_1}\cdots s_{i_m}$$
and $m$ is minimal (this $m$ is the Coxeter length $\ell(w)$). We denote by $w_\circ$ the longest element of $W(\lie g)$. 

The Artin braid group $Br_\lieg$ is generated by the $T_i$, $i\in I$ subject to the relations (for all $i,j\in I$):
\begin{equation}
\label{eq:braid relation}
\underbrace{T_iT_j\cdots}_{m_{ij}} =\underbrace{T_jT_i\cdots}_{m_{ij}} \ .
\end{equation}

To each $w\in W(\lieg)$ one associates the element $T_w\in Br_\lieg$ such that 
\begin{equation}
\label{eq:Tw}T_w=T_{i_1}\cdots T_{i_m}
\end{equation}
for each $\ii=(i_1,\ldots,i_m)\in R(w)$ (it follows from relations \eqref{eq:braid relation} that $T_w$ is well-defined).

\begin{lemma}[{\citem{Del}*{Theorem~4.21} and~\citem{BS}*{Lemma~5.2}}]
\label{le:center of braid group} 

For each $w\in W$ we have 
$$T_{w_\circ} T_w T_{w_\circ}^{-1}=T_{w_\circ w w_\circ}\ .$$
In particular, the element 
$C_\lieg=\begin{cases} 
T_{w_\circ} & \text{if $w_\circ$ is in the center of $W$}\\
T_{w_\circ}^2 & \text{if $w_\circ$ is not in the center of $W$}\\
\end{cases}
$ is in the center of $Br_\lieg$. Moreover, the center of $Br_\lieg$ is generated by all $C_{\lieg'}$, where $\lieg'$ runs over the simple Lie ideals of $\lieg$.
\end{lemma}

Now we return to the folding situation. Let $\lieg$ be a semisimple Lie algebra and let $\sigma$ be its admissible diagram automorphism.

Note that $\sigma$ defines an automorphism of  $W(\lieg)$ (respectively of $Br_\lieg$) via $\sigma(s_i)= s_{\sigma(i)}$ (respectively 
$\sigma(T_i)=T_{\sigma(i)}$) for $i\in I$. Denote by $\hat w_\circ$ (respectively, $w_\circ$) the longest element of $W(\lieg)$ 
(respectively, of $W({\lieg^\sigma}^\vee)$). Since $\sigma$ preserves the Coxeter length, it follows that $\sigma(\hat w_\circ)=\hat w_\circ$.
The following result provides a ``folding''  isomorphism of the corresponding Weyl and braid groups.

\begin{proposition}\label{prop:WeylArtin folding} 
For each semisimple simply laced Lie algebra $\lieg$ and its admissible diagram automorphism $\sigma$ we have:
\begin{enumerate}[{\rm(i)}]
\item\label{prop:WeylArtin folding.i} The assignment
\begin{equation}
\label{eq:Weyl folding}
s_r\mapsto \hat s_r=\prod_{i\in {\mathcal O}_r} s_i,\qquad r\in I/\sigma
\end{equation}
extends to an isomorphism of groups $\hat{\cdot}:W({\lieg^\sigma}^\vee)\widetilde \to W(\lieg)^\sigma\subset W(\lieg)$.

\item\label{prop:WeylArtin folding.ii} The assignment 
\begin{equation}
\label{eq:braid folding}
T_r\mapsto \hat T_r=\prod_{i\in {\mathcal O}_r} T_i,\qquad r\in I/\sigma
\end{equation}
extends to an isomorphism of groups $Br_{{\lieg^\sigma}^\vee}\widetilde \to (Br_\lieg)^\sigma\subset Br_\lieg$.
Under this isomorphism the element $T_w$ of $Br_{{\lieg^\sigma}^\vee}$ is mapped to the element $T_{\hat w}$ of $Br_{\lieg}$.
\end{enumerate}
\end{proposition}

\begin{proof} It is easy to see that \eqref{eq:Weyl folding} defines a group homomorphism because it respects the Coxeter relations~\eqref{eq:coxeter relation}. The injectivity also follows. Let us prove surjectivity, i.e. that each element $w\in W(\lieg)^\sigma$ factors into a product of the $\hat s_r$,
$r\in I/\sigma$. We proceed by induction on
the Coxeter length of $w$, the induction base being trivial. We need the following well-known result.
\begin{lemma} 
\label{le:cancelation Weyl}
Let $w\in W(\lieg)$ and $i\ne j$ be such that $\ell(s_iw)=\ell(s_jw)=\ell(w)-1$. Then there exists $w'$ such that $w=\underbrace{s_is_j\cdots}_{m_{ij}} \cdot w'$ and $\ell(w')=\ell(w)-m_{ij}$. In particular, if $I_0\subset I$ satisfies 
\begin{itemize}
\item $\ell(s_iw)=\ell(w)-1$ for each $i\in I_0$,

\item $s_is_{i'}=s_{i'}s_i$ for all $i,i'\in I_0$,
\end{itemize}
then there exists $w''$ such that $w=(\prod_{i\in I_0} s_i)\cdot w''$ and $\ell(w'')=\ell(w)-|I_0|$.
\end{lemma} 

Indeed, let $w\in  W(\lieg)^\sigma$. Then there exits $i\in I$ such that  $w=s_iw'$ for some $w'$ with $\ell(w')=\ell(w)-1$. Applying $\sigma^k$, we obtain: $s_iw'=s_{\sigma^k(i)}\sigma^k(w')$, hence $\ell(s_{\sigma^k(i)}w)=\ell(w)-1$. Thus the set $I_0=\{i,\sigma(i),\sigma^2(i)\ldots\}={\mathcal O}_r$ satisfies the conditions of Lemma~\ref{le:cancelation Weyl}. Therefore, $w=\hat s_rw''$ with $\ell(w'')=\ell(w)-|{\mathcal O}_r|$. In particular,  $w''\in W(\lieg)^\sigma$ and 
$\ell(w'')<\ell(w)$ so we finish the proof by induction. This proves~\eqref{prop:WeylArtin folding.i}.

To prove~\eqref{prop:WeylArtin folding.ii} note that \eqref{eq:braid folding} defines a group homomorphism because it respects the Coxeter relations \eqref{eq:braid relation}. The injectivity also follows. Let us prove surjectivity, i.e., that each element $g\in (Br_\lieg)^\sigma$ factors into a product of 
the $\hat T_r^{\pm 1}$, $r\in I/\sigma$. 

 Following~\cites{BS,Del}, denote by $Br^+_\lieg$ the {\it positive braid monoid}, i.e, the monoid generated by the $T_i$, $i\in I$ subject to \eqref{eq:braid relation}.

\begin{lemma}[{\citem{BS}*{Proposition 5.5},\citem{Del}*{Proposition~4.17}}] 
\label{le:embedding braid}
\noindent
\begin{enumerate}[{\rm(i)}]
\item\label{le:embedding braid.i}The assignment $T_i\mapsto T_i$ defines an injective homomorphism of monoids $Br^+_\lieg\hookrightarrow Br_\lieg$. In other words, $Br^+_\lieg$ is naturally a submonoid of $Br_\lieg$.

\item\label{le:embedding braid.ii} For each $g\in Br_\lieg$ there exists an element $g^+\in Br^+_\lieg$ such that $g=Cg^+$ for some central element $C$ of $Br_\lieg$.
\end{enumerate}
\end{lemma}

Note that the central element $C_\lieg=T_{\hat w_\circ}^2$ from Lemma \ref{le:center of braid group}  is the product of all generators $C_{\lieg'}$ of the center of $Br_\lieg^+$, where $\lieg'$ runs over simple Lie ideals of $\lieg$. This and Lemma \ref{le:embedding braid} imply that for each $g\in Br_\lieg$ there exists $N\ge 0$ such that $C_\lieg^N \cdot g\in Br_\lieg^+$. Taking into account that $T_{\hat w_\circ}$ and hence $C_\lieg$ 
is fixed under $\sigma$, it suffices to prove that any element $g^+\in (Br_\lieg^+)^\sigma$ factors into a product of the $\hat T_r$, $r\in I/\sigma$. 

We need the following result which is parallel to Lemma \ref{le:cancelation Weyl}.

\begin{lemma} [\citem{BS}*{Lemma 2.1}]
\label{le:cancelation braid}
Let $g^+\in Br_\lieg^+$ and $i\ne j$ be such that $g^+=T_i\cdot g_i^+=T_j\cdot g_j^+$ for some $g_i^+,g_j^+\in Br_\lieg^+$. Then there exists $h^+\in Br_\lieg^+$ such that $g^+=\underbrace{T_iT_j\cdots}_{m_{ij}} \cdot h^+$. In particular, if $I_0\subset I$ satisfies 

$\bullet$ $g^+\in T_i\cdot Br_\lieg^+$ for each $i\in I_0$,

$\bullet$ $T_iT_{i'}=T_{i'}T_i$ for all $i,i'\in I_0$,

\noindent
then there exists $h^+\in Br_\lieg^+$ such that $g^+=(\prod_{i\in I_0} T_i)\cdot h^+$.

\end{lemma}

We proceed by induction on length of elements in $Br_\lieg^+$.
Indeed, let $g^+\in  (Br_\lieg^+)^\sigma$. Then there exits $i\in I$ such that  $g^+=T_i\cdot g_i^+$ for some $g_i^+\in Br_\lieg^+$. Applying $\sigma^k$, we obtain: $T_i\cdot g^+_i=T_{\sigma^k(i)}\cdot \sigma^k(g^+)$, where $\sigma^k(g^+)\in Br_\lieg^+$. Thus the set $I_0=\{i,\sigma(i),\sigma^2(i)\ldots\}={\mathcal O}_r$ satisfies the conditions of Lemma \ref{le:cancelation braid}. Therefore, $g^+=\hat T_r \cdot h^+$ for some $h^+\in Br_\lieg^+$. In particular,  $h^+\in (Br_\lieg^+)^\sigma$ and is shorter than $g^+$ so we finish the proof by  induction. This proves~\eqref{prop:WeylArtin folding.ii}.
\end{proof}

\subsection{PBW bases and quantum folding}\label{subs:PBW bases}

G.~Lusztig proved in~\cite{Lus} that  $Br_\lieg$ acts on $U_q(\lieg)$ by algebra automorphisms via:  
\begin{equation}\label{eq:Lustig braid action}
\begin{gathered}
T_i(K_j^{\pm 1})=K_j^{\pm 1}K_i^{\mp a_{ij}},\quad T_i(E_i) = - K_i^{-1}F_i,\quad T_i(F_i)=-E_i K_i,\quad i,j\in I\\
T_i(E_j)= \mskip-5mu\sum_{s+r=-a_{ij}}\mskip-8mu (-1)^r q_i^{-r} E_i^{(r)}E_j E_i^{(s)},\, T_i(F_j)=\mskip-5mu\sum_{s+r=-a_{ij}}\mskip-8mu (-1)^r q_i^{r} F_i^{(s)}F_j F_i^{(r)},\,
i\not=j
\end{gathered}
\end{equation}
where $Y_i^{(k)}=\frac{1}{[k]_{q_i}!}\,Y_i^k$ (these automorphisms $T_i$ are denoted~$T'_{i,-1}$ in~\cite{Lus}).
G. Lusztig also proved in \cite{Lus} that $T_w(E_i)\in U^+_q(\lieg)$ if and only if $\ell(ws_i)=\ell(w)+1$. Using this,  
for each $\ii=(i_1,\ldots,i_m)\in R(w_\circ)$, define the ordered set $X_\ii=\{X_1,X_2,\ldots,X_m\}\subset U^+_q(\lieg)$ by:
$$
X_k=X_{\bi,k}=c_k^{-1} T_{i_1}\cdots T_{i_{k-1}}(E_{i_k}),
$$
where $c_k=\gamma(s_{i_1}\cdots s_{i_{k-1}}(\alpha_{i_k})-\alpha_{i_k})$ and $\gamma:Q\to \CC(q)^\times$ is the unique group homomorphism 
defined by $\gamma(\alpha_i)=q_i-q_i^{-1}$. It should be noted that for any $w,w'\in W(\lie g)$, $i,i'\in I$ such that $w\alpha_i=w'\alpha_i'$,
$\gamma(w\alpha_i-\alpha_i)=\gamma(w'\alpha_{i'}-\alpha_{i'})$.

We will need the following useful Lemma which is, most likely, well known. 
\begin{lemma}\label{lem:lus.br.id}
Suppose that $i,j\in I$ and $w\in W(\lie g)$ satisfy $w\alpha_i=\alpha_j$. Then ${T_w(E_i)=E_j}$.
\end{lemma}
\begin{proof}
We use induction on~$\ell(w)$, the induction base being trivial.
Since $w\alpha_i=\alpha_j$ we have $\ell(ws_i)=\ell(w)+1$ and $w s_i w^{-1}=s_j$. Then by~\citem{BZ}*{Lemma~9.9}, there exist
$k\in I$ and
a $\bi\in R(w)$ such that $\bi$ terminates with $(\dots,i,k)\in R(w_\circ(i,k)s_i)$ where
$w_\circ(i,k)$ denotes the 
longest element of the subgroup of~$W(\lieg)$ generated by $s_i,s_k$. Since
by~\citem{Lus}*{\S\S 39.2.2--4} 
\begin{equation}\label{eq:lus.br.id}
T_{w_\circ(i,k)s_i}(E_i)=\begin{cases}E_k,&a_{ik}=a_{ki}=-1\\
                          E_i,& \text{otherwise,}
                         \end{cases}
\end{equation}
we conclude that either $T_w(E_i)=T_{w'}(E_i)$ with~$w'\alpha_i=\alpha_j$ or~$T_w(E_i)=T_{w'}(E_k)$ with
$w'\alpha_k=\alpha_j$ and in both cases $\ell(w')<\ell(w)$.
\end{proof}

For each $a=(a_1,\ldots,a_m)\in \ZZ_{\ge 0}^m$ define the monomial $X_\ii^{a}\in U_q^+(\lieg)$ by:
$$X_\ii^{a}=X_1^{a_1}\cdots X_{m}^{a_m}.$$ 
The following result is well-known.
\begin{proposition} [\citem{Lus}*{Corollary~40.2.2}] The set $M(X_\ii)$ of all monomials $X_\ii^{a}$ is a PBW-basis of $U_q^+(\lieg)$.

\end{proposition}

\begin{remark} The basis $M(X_\ii)$ differs from Lusztig's PBW basis from \cite{Lus} in that we do not divide the monomials by $q$-factorials
but rather by some factors which vanish at $q=1$.
\end{remark}

Let  $\sigma$ be an admissible diagram automorphism of  a semisimple simply laced Lie algebra $\lieg$.
For each $r\in I/\sigma$ define the element $\hat E_r:=\prod\limits_{i\in {\mathcal O}_r} E_i\in U_q^+(\lieg)$ and the set $\hat E_r^\bullet\subset U_q^+(\lieg)$ of all monomials $\prod\limits_{i\in {\mathcal O}_r} E_i^{a_i}$, $a_i\in \ZZ_{\ge 0}$. Note that $\hat E_r$ is fixed under the action of $\sigma$ on $U_q^+(\lieg)$ and the set $\hat E_r^\bullet$ is $\sigma$-invariant. Moreover, $(\hat E_r^\bullet)^\sigma=\{\hat E_r^k\,|\,k\in \ZZ_{\ge 0}\}$. The following result is obvious. 
\begin{lemma} 
\label{le:iota PBW}
Assume that $\sigma$ is an admissible diagram automorphism of $\lieg$ and let  $\ii=(r_1,\ldots,r_m)\in R(w_\circ)$. Let
$\hat\ii\in R(\hat w_\circ)$ be any lifting of~$\ii$ (as defined in the Introduction).
Then 

\begin{enumerate}[{\rm(i)}]
\item\label{le:iota PBW.i} $M(X_{\hat\bi})=\hat X_1^\bullet\cdots \hat X_m^\bullet$ 
up to multiplication by non-zero scalars, where $\hat X_k^\bullet=\hat T_{r_1} \cdots \hat T_{r_{k-1}}(\hat E_{r_k}^\bullet)$, $1\le k\le m$.

\item\label{le:iota PBW.ii} The basis $M(X_{\hat\ii})$ is invariant under the action of $\sigma$ on $U_q^+(\lieg)$ and the fixed point set 
$M(X_{\hat\ii})^\sigma$ coincides, up to scalars, with the set
$$\hat X_\ii^a=\hat X_1^{a_1}\cdots \hat X_m^{a_m}\ ,$$
where 
$$
\hat X_k=\hat c_k^{-1}\hat T_{r_1}\hat T_{r_2}\cdots \hat T_{r_{k-1}}(\hat E_{r_k})
$$
and $\hat c_k=\prod_{i\in \mathcal O_{r_k}}\gamma(\hat s_{r_1}\cdots \hat s_{r_{k-1}}(\alpha_i)-\alpha_i)$.
\end{enumerate}
\end{lemma}

This motivates the following definition. 

\begin{definition}[$\ii$-th quantum folding]\label{defn:quantum folding} For each $\ii\in R(w_\circ)$ define an injective linear map $\iota_\ii: U_q^+({\lieg^\sigma}^\vee) \hookrightarrow U_q^+(\lieg)^\sigma\subset U_q^+(\lieg)$ by the formula
$$\iota_\ii(X_\ii^a)=\hat X_\ii^a$$
for $a=(a_1,\ldots,a_m)\in \ZZ_{\ge 0}^m$.
\end{definition}

\begin{remark}\label{rem:lusztig bijection} Combinatorially, $\iota_\ii$ is a bijection $M(X_\ii)\widetilde \to M(X_{\hat\ii})^\sigma$, which can be interpreted as a certain bijection of Kashiwara crystals. More precisely, we can view  $\iota_\ii$ as the composition of the canonical isomorphism $B_\infty({\lieg^\sigma}^\vee)\cong B_\infty(\lieg^\vee)$ with the Lusztig's $\sigma$-equivariant identification 
$\hat L_\ii:B_\infty(\lieg)\widetilde\to M(X_{\hat\ii})$ so that  the following diagram commutes, in view of Proposition \ref{pr:crystal folding} and Lemma \ref{le:iota PBW}:
$$\begin{CD}
B_\infty(\lieg) @>\hat L_\ii>>M(X_{\hat\ii})\\
@AAA@AAA\\
B_\infty({\lieg^\sigma}^\vee)@>\hat L_\ii^\sigma>>M(X_{\hat\ii})
\end{CD}
$$
where the vertical arrows are natural inclusions of $\sigma$-fixed point subsets.
\end{remark}
\begin{remark}
The Definition~\ref{defn:quantum folding} makes sense for any $\bi\in R(w)$ where~$w$ is any element of the Weyl group. In that case
we replace $U_q^+(\lie g)$ by the algebra $U_q(w)$ introduced in~\cite{DCKP,Lus} and extensively studied in~\cite{Y,Y1}.
One can also define quantum foldings for some pairs $(\lieg,\lieg')$ not related by a diagram automorphism (e.g., for $(\lie{sp}_6,G_2)$). Namely, following \cite{Kas95}, one can embed $B_\infty(\lieg')$ into $B_\infty(\lieg)$ and then extend the embedding linearly to $\iota:U_q(\lieg')\hookrightarrow U_q(\lieg)$  using Lusztig's identifications  $B_\infty(\lieg')\cong M(X_{\bi})$ and  $B_\infty(\lieg)\cong M(X_{\hat \bi})$. 
\end{remark}

We conclude this section with the proof of Theorem~\ref{th:folding ii}.
\begin{proof}[Proof of Theorem~\ref{th:folding ii}]
It is sufficient to prove the statement for two reduced decompositions $\ii,\ii'\in R(w_\circ)$ which differ by one braid relation involving 
$r,s\in I/\sigma$
and
thus it suffices to consider the rank~2 case. We have the following two possibilities:

$1^\circ$. $|\mathcal O_r|=2$, $|\mathcal O_s|=1$ and $\ii=(r,s,r,s)$, $\ii'=(s,r,s,r)$.
Let $\mathcal O_r=\{i,j\}$, $\mathcal O_s=\{k\}$, with $a_{i,k}=a_{j,k}=-1$. 
Then $\hat\ii=(i,j,k,i,j,k)$, $\hat\ii'=(k,i,j,k,i,j)$ and the 
elements $\hat X_s$ for these two decompositions are, respectively,
$$\hat X_1=E_i E_j,\,\hat X_2=\frac{[E_i,[E_j,E_k]_{q^{-1}}]_{q^{-1}}}{(q-q^{-1})^2},\,\hat X_3= \frac{[E_i,E_k]_{q^{-1}} [E_j,E_k]_{q^{-1}}}{(q-q^{-1})^2},\,\hat X_4=E_k,$$
while $\hat X_1'=\hat X_4$, $\hat X_2'=\hat X_3^*$, $\hat X_3'=\hat X_2^*$ and~$\hat X_4'=\hat X_1^*$, where $*$ is the unique
anti-automorphism of~$U_q^+(\lie g)$ such that  $E_i^*=E_i$ for $i\in I$. 
A straightforward computation shows that 
$$
\hat X_2'=q^{-2} (\hat X_3+(q-q^{-1})^{-1}[\hat X_4,\hat X_1]\hat X_4),\qquad \hat X_3'=\hat X_2+q^{-1}(q-q^{-1})^{-1}[\hat X_4,\hat X_1].
$$
Therefore, the subalgebra of $U_q^+(\lie g)$ generated by $\hat X_r$ contains all elements $\hat X_r'$. Applying 
$*$ we obtain the opposite inclusion.

$2^\circ.$ $|\mathcal O_r|=|\mathcal O_s|=2$ and $\ii=(r,s,r)$, $\ii'=(s,r,s)$. Let $\mathcal O_r=\{i,j\}$,
$\mathcal O_s=\{k,l\}$ with $a_{i,k}=-1=a_{j,l}=-1$ and $a_{i,l}=a_{j,k}=0$. Then we have $\hat X_1=E_i E_j=\hat X_3'$, $\hat X_3=E_k E_l=\hat X_1'$ and 
$$
\hat X_2=\frac{[E_i,E_k]_{q^{-1}} [E_j,E_l]_{q^{-1}}}{(q-q^{-1})^2},\quad \hat X_2'=\frac{[E_k,E_i]_{q^{-1}} [E_l,E_j]_{q^{-1}}}{(q-q^{-1})^2}.
$$
It is easy to check that 
$$
\hat X_2'=\hat X_2+q^{-1}(q-q^{-1})^{-1}[\hat X_3, \hat X_1],
$$
hence $\hat X_2'$ is contained in the subalgebra of~$U_q^+(\lie g)$ generated by $\hat X_k$, $1\le k\le 3$. Interchanging the 
role of~$r$ and~$s$ completes the proof.
\end{proof}

\def\AA{\mathbb A}
\subsection{Diamond Lemma and specializations of PBW algebras} 
\label{subsect:specialization}
We will use the following version of Bergman's Diamond Lemma (\cite{Berg}).
Let $A$ be an associative $\kk$-algebra and suppose that $A$ is generated by a totally ordered set~$X_A$. Let $M(X_A)$
be the set of ordered monomials on the~$X_A$.
\begin{proposition}\label{prop:diamond lemma}
Assume that the defining relations for $(A,X_A)$ are
$$
X'X=\sum_{M\in M(X_A)} c_{X,X'}^M M,\qquad X<X',
$$
where $c_{X,X'}^{XX'}\in\kk^\times$ and for any $M\in M(X_A)$, $c_{X,X'}^M\not=0$ implies that $M<X'X$ in the lexicographic order. 
If for all $X<X'<X''$ there exist a unique $S\subset M(X_A)$ and a unique $\{a_{M}\,:\, M\in S\}\subset \kk^\times$ such that 
$$
X''X'X=\sum_{M\in S} a_M M, 
$$
then $(A,X_A)$ is a PBW algebra 
\end{proposition}
Note that, unlike~\cite{P-P}, we do not require $(A,X_A)$ to be quadratic. In fact, in most cases where we will need to apply the Diamond 
Lemma, this will not be the case.

We will now list some elementary properties of specializations which will be needed later. 
The simplest instance of specialization is given by the following definition. Throughout this subsection, let $\kk=\CC(t)$ (later on, we set $t=q-1$) and denote by $\kk_0$ the set of all $f=f(t)\in \kk$ such that $f(0)$ is defined. Clearly, $\kk_0$ is a (local) subalgebra of $\kk$ and for each non-zero $f\in \kk$ either $f\in \kk_0$ or $f^{-1}\in \kk_0$.

\begin{definition}
Let $U$ and $V$ be $\kk$-vector spaces with bases $B_U$ and $B_V$ respectively. Let  $F:U\to V$ be a $\kk$-linear map  that all matrix coefficients $c_{b,b'}=c_{b,b'}(t)\in \kk$ defined by:
$$F(b)=\sum_{b'\in B_V} c_{b',b} b'$$
for $b\in B_U$, belong to $\kk_0$. Then the {\it specialization} $F_0$ of $F$ is a $\kk$-linear map $U\to V$ given by:
$$F_0(b)=\sum_{b'\in B_V} c_{b',b}(0) b'$$
for $b\in B_U$ (here, unlike in the literature on deformation theory, we preserve the ground field $\kk=\CC(t)$ after the specialization because it is more convenient to view both $F$ and $F_0$ as $\kk$-linear maps $U\to V$).
\end{definition}

Similarly, let $(A,X_A)$ be a PBW algebra. Then, in the notation of Definition~\ref{def:liftable}, it has a unique presentation:
$$X'X=\sum_{M\in M(X_A)} c_{X,X'}^M M$$
for all $X,X'\in X_A$ such that $X< X'$, where all $c_{X,X'}^M\in \kk$. 
\begin{definition}

\label{def:specialization}
We say that the PBW algebra $(A,X_A)$ is {\it specializable} if all the $c_{X,X'}^M$ belong to $\kk_0$ 
and in that case define the specialization $(A_0,X_A)$ to be the associative $\kk$-algebra with the unique presentation:
$$X'X=\sum_{M\in M(X_A)} c_{X,X'}^M(0) M$$
for all $X,X'\in X_A$ with  $X< X'$, where all $c_{X,X'}^M\in \kk$. 

We say that a specializable PBW algebra $(A,X_A)$ is {\it optimal} if  $(A_0,X_A)$ is just the polynomial algebra~$\kk[X_A]$, that
is, if~$c_{X,X'}^M(0)=\delta_{M,XX'}$ for all relevant $X,X',M$ (i.e., the defining relations in~$(A_0,X_A)$ are $X'X=XX'$).
In that case we define a bi-differential bracket $\{\cdot,\cdot\}$ on $\kk[X_A]$ by:
\begin{itemize}
\item (Leibniz rule) $\{xy,z\}=x\{y,z\}+y\{x,z\}$, $\{x,yz\}=\{x,y\}z+y\{x,z\}$ for all $x,y,z\in \kk[X_A]$.

\item (skew symmetry) $\{x,y\}=-\{y,x\}$ for all  $x,y\in \kk[X_A]$.

\item for all $X,X'\in X_A$  with $X< X'$
\begin{equation}
\label{eq:explicit poisson}
\{X',X\}=\sum\limits_{M\in M(X_A)} \frac{\partial c_{X,X'}^M}{\partial t}\Big|_{t=0} M.
\end{equation}   
\end{itemize}
\end{definition}

\begin{proposition} \label{prop:Specialization}
Let $(A,X_A)$ be a specializable PBW-algebra. 
Then:

\begin{enumerate}[{\rm(i)}]
\item\label{prop:Spec.i} Its specialization $(A_0,X_A)$ is also PBW.

\item\label{prop:Spec.ii} If, additionally, $(A,X_A)$ is  optimal, then  the bracket $\{\cdot,\cdot\}$ on $A_0=\kk[X_A]$ is Poisson. 
\end{enumerate}
\end{proposition}
\begin{proof}
Denote $A'_0=\sum_{M\in M(X_A)} \kk_0 M$. Clearly, this is a $\kk_0$-subalgebra of $A$ and a free $\kk_0$-module. 
Taking into the account that $(t)$ is a (unique) maximal ideal in $\kk_0$, we see that the specialization $A_0$ of $A$, 
is canonically isomorphic to $\kk\otimes_{\kk_0} A_0'/t A_0'$. 

To prove~\eqref{prop:Spec.i}, suppose that in~$A_0$ we have 
$\sum\limits _{M\in M(X_A)} c_M M=0$, where all $c_M\in \CC$. This implies that $\sum\limits_{M\in M(A_X)} c_M M\in t A_0'$, hence
$\sum\limits_{M\in M(X_A)} c_M M=\sum\limits_{M\in M(X_A)} t c'_{M} M$ for some $c'_M\in\kk_0$. Since~$A_0'$ is a free $\kk_0$-module,
this implies that $c_M=c'_M=0$ for all $M$ and completes the proof of~\eqref{prop:Spec.i}.

Now we prove~\eqref{prop:Spec.ii}. Optimality of $(A,X_A)$ implies that the commutator $[a,b]$ of any $a,b\in A'_0$ belongs to the ideal $tA_0$. For any $a_0,b_0\in A'_0/tA'_0$ 
denote:
$$\{a_0,b_0\}=\pi\Big( \frac{[a,b]}{t}\Big)$$
where $\pi:A'_0\twoheadrightarrow A'_0/tA'_0$ is the canonical projection and $a,b\in A'_0$ are any elements such that $\pi(a)=a_0$, $\pi(b)=b_0$. Clearly, the bracket $\{a_0,b_0\}$ is well-defined (i.e., it does not depend on the choice of representatives $a$ and $b$).  This bracket is Poisson because the original commutator bracket was skew-symmetric, and satisfied both the Liebniz rule and Jacobi identities.

It remains to verify \eqref{eq:explicit poisson}. Indeed, let $X,X'\in X_A$ with $X<X'$. We have 
$$\{X',X\}=\pi\Big(\frac{X'X-XX'}{t}\Big)=\pi\Big(\frac{c_{X,X'}^{XX'}-1}{t}\Big) XX' +\sum_{M\not=XX'} \pi\Big(\frac{c_{X,X'}^M}{t}\Big) M\ .$$
This gives \eqref{eq:explicit poisson} because $\pi(f)=f(0)$ and $\pi(\frac{f-f(0)}{t})=\frac{\partial f}{\partial t}\Big|_{t=0}$ for any $f\in \kk_0$.
The proposition is proved.
\end{proof}

\subsection{Nichols algebras and proof of Theorem \ref{thm:flat quadratic algebra}}\label{subsect:diamond}
We will now prove Theorem \ref{thm:flat quadratic algebra}, which allows to establish the 
PBW property when an algebra is quadratic and is defined in terms of a braiding. Retain the notation of Section~\ref{subsect:specialization}.

\begin{proof}
Let $Y$ be a $\kk$-vector space and 
 $\Psi:Y\otimes Y\to Y\otimes Y$ be linear map. For each $k\ge 2$ we define the linear maps $\Psi_i:Y^{\otimes k}\to Y^{\otimes k}$, $i=1,\ldots,k-1$ by the formula: 
$$\Psi_i=1^{\otimes i-1}\otimes \Psi\otimes 1^{\otimes k-i-1} \ .$$

If $\Psi$ is invertible and satisfies the braid equation \eqref{eq:braid equation}, then for each $k\ge 2$ one obtains the representation of the braid group $Br_{\lie{sl}_k}$ on $Y^{\otimes k}$ (in the notation of Section \ref{subs:PBW bases}) via $T_i\mapsto \Psi_i$
for $i=1,\ldots,k-1$. 

Therefore, one can define the braided factorial  $[k]!_\Psi:Y^{\otimes k}\to Y^{\otimes k}$ by the formula:
$$[k]!_\Psi=\sum_{w\in S_k} \Psi_w$$ 
where $\Psi_w$ is the image of $T_w$ (given by \eqref{eq:Tw}) in $\End_\kk(Y^{\otimes k})$. 

It is well-known (see e.g.~\cite{W}) that $I_\Psi:=\bigoplus\limits_{k\ge 2}\ker[k]!_\Psi$ is a two-sided ideal in the tensor algebra $T(Y)$. The quotient algebra ${\mathcal B}_\Psi(Y):=T(Y)/I_\Psi$ is called the {\em Nichols-Woronowicz 
algebra} of the braided vector space $(Y,\Psi)$.

For each linear map $\Psi:Y\otimes Y\to Y\otimes Y$ denote $A_\Psi(Y):=T(Y)/\langle \ker  (\Psi+1)\rangle$. 

We need the following result. 

\begin{proposition} 
\label{pr:flat Nichols}
Let $Y$ be a $\kk$-vector space and let $\Psi:Y\otimes Y\to Y\otimes Y$ be an invertible $\kk$-linear map satisfying the braid equation. 
Assume that:
\begin{enumerate}[{\rm(i)}]
\item\label{pr:flat Nichols.i} The specialization $\Psi_0=\Psi|_{t=0}$ of $\Psi$ is a well-defined (with respect to a basis of $Y$) invertible linear map $Y\otimes Y\to Y\otimes Y$ satisfying the braid equation.

\item\label{pr:flat Nichols.ii} The Nichols-Woronowicz 
algebra ${\mathcal B}_{\Psi_0}(Y)$ is isomorphic to $A_{\Psi_0}(Y)$ as a graded vector space. 

\item\label{pr:flat Nichols.iii} $\dim \ker(\Psi+1)=\dim \ker(\Psi_0+1)$.
\end{enumerate}

\noindent Then ${\mathcal B}_\Psi(Y)\cong {\mathcal B}_{\Psi_0}(Y)$ as a graded vector space and 
one has an isomorphism of algebras 
$${\mathcal B}_\Psi(Y)\widetilde \to A_{\Psi}(Y) \ .$$ 

\end{proposition}

\begin{proof} 
We will need two technical results.
\begin{lemma} 
\label{le:inequalities on dimension} 
Let $U$ and $V$ be a finite-dimensional $\CC(t)$ vector spaces  and $F:U\to V$ be a linear map such that its specialization $F_0$ at $t=0$ is a well-defined map $U\to V$ (with respect to some bases of $B_U$ and $B_V$).  
Then
\begin{enumerate}[{\rm(a)}]
\item\label{le:ineq.a} $\dim F_0(U)\le \dim F(U)$ and $\dim \ker  F\le \dim \ker  F_0$. 

\item\label{le:ineq.b} Assume that $\dim \ker  F= \dim \ker  F_0$. Then  there is a linear map $G:V\to V$ such that: 
\begin{enumerate}[{\rm(i)}]
\item\label{le:ineq.b.i} $G(V)=\ker F$, 

\item\label{le:ineq.b.ii} the specialization $G_0$ of $G$ at $t=0$ is well-defined,  

\item\label{le:ineq.b.iii} $G_0(V)=\ker F_0$.
\end{enumerate}
\end{enumerate}
\end{lemma}

\begin{proof} Fix bases $B_U$ and $B_V$ and identify $U$ with $\kk^n$, $V$ with $\kk^m$, and $F:\kk^n\to \kk^n$ with its $m\times n$ matrix.

It is well-known (and easy to show) that for each non-zero $F\in \operatorname{Mat}_{m\times n}(\kk)$ there exist
$g_t\in GL_m(\kk_0)$ and $h_t\in GL_n(\kk_0)$ such that 
\begin{equation}
\label{eq:affine grassmannian decomposition}
F=g_tPh_t\ ,
\end{equation}
where $P$ is an $m\times n$-matrix such that $P_{ij}= 0$ unless $(i,j)\in \{(1,1),\ldots,(r,r)\}$, and $P_{ii}=t^{\lambda_i}$ for $i=1,\ldots,r$, where $r=\operatorname{rank}(F)$ and $\lambda_i\in \ZZ$.  Therefore, $F_0$ is well-defined if and only if all $\lambda_i\ge 0$. 

In particular, $F_0=g_0P_0h_0$ and $\operatorname{rank}(F_0)\le k$, where $P_0$ is the specialization of $P$ at $t=0$. That is,  
$$\dim F_0(U)=\operatorname{rank}(M_0)\le \operatorname{rank}(M_t)=\dim F(U) \ .$$ 
This in conjunction with the equality $\dim \ker  F+\operatorname{rank}(F)=\dim V$ proves~\eqref{le:ineq.a}

Now we prove~\eqref{le:ineq.b}. Clearly, the condition $\dim \ker  F= \dim \ker  F_0$  is equivalent to  $\operatorname{rank}(P_0)=r$, i.e.,  in the decomposition \ref{eq:affine grassmannian decomposition} one has $\lambda_1=\cdots=\lambda_r=0$, i.e., $P=P_0$ is the matrix (not depending on $t$) of the standard projection $\kk^n\to \kk^r\subset \kk^m$.

Let $P^\perp\in \operatorname{Mat}_{n\times n}(\CC)$ be the standard projection $\kk^n\to \operatorname{Span}\{e_{r+1},\ldots,e_n\}\subset \kk^n$ (e.g., $P P^\perp=0$). Denote $G:=h_t^{-1}P^\perp$ so that the specialization  $G_0$ of $G$ at $t=0$ is well-defined and given by  $G_0=h_0^{-1}P^\perp$.

Clearly,  
$$G(\kk^m)=h_t^{-1}(\operatorname{Span}\{e_{r+1},\ldots,e_n\})=\ker  Ph_t=\ker  F \ .$$
Similarly, $G_0(\kk^m)=\ker  F$.
This proves~\eqref{le:ineq.b}. 
\end{proof}

\begin{lemma}\label{lem:spec of ideals}
Let $\mathcal F$ be a free $\kk$-algebra on $y_i$, $i\in I$ where~$I$ is a finite set. Fix
a grading on~$\mathcal F$ with $\deg y_i\in\ZZ_{>0}$. Fix any finite subset $B_t$ of 
specializable (with respect to the natural monomial basis of~$\mathcal F$) homogeneous elements in~$\mathcal F$.
Then $\dim \lr{B_t}_n\ge \dim\lr{B_{0}}_n$, where $\lr{B_t}_n$ (respectively, $\lr{B_{0}}_n$) is the $n$th homogeneous component 
of the ideal in~$\mathcal F$ generated by $B_t$ (respectively, by the specialization $B_{t=0}$ of~$B_t$).
\end{lemma}
\begin{proof}
Clearly $$\lr{B_t}_n=\bigoplus\limits_{i+j+k=n}\sum\limits_{b\in B_t\,:\, \deg b=j} \mathcal F_i b \mathcal F_k.
$$
Define $\widehat{\lr{B_t}_n}=\bigoplus\limits_{i+j+k=n}\bigoplus\limits_{b\in B_t\,:\, \deg b=j} \mathcal F_i b \mathcal F_k$ and
let $F:\widehat{\lr{B_t}_n}\to \mathcal F_n$ be the natural map which is the identity on each summand.
Clearly the specialization of $F_0$ at $t=0$ with respect to the natural monomial basis in both spaces is well-defined and
the image of~$F$ (respectively, of $F_0$) is $\lr{B_t}_n$ (respectively, $\lr{B_0}_n$). Then the assertion 
follows from Lemma~\ref{le:inequalities on dimension}\eqref{le:ineq.a}.
\end{proof}

The algebra $\mathcal B_\Psi(Y)$ is graded and $\mathcal B_\Psi(Y)_k$ is isomorphic to $[k]_\Psi!(Y^{\tensor k})$ as a vector space 
Therefore,
\begin{equation}\label{eq:dim ineq.1}
\dim(\mathcal B_{\Psi_0})_k\le \dim (\mathcal B_\Psi)_k
\end{equation}
by Lemma~\ref{le:inequalities on dimension}\eqref{le:ineq.a} with $V=Y^{\tensor k}$ and $F=[k]!_\Psi$.

On the other hand, note that  $\ker  (\Psi+1)= \ker  [2]!_\Psi$ and so 
we have a structural homomorphism of graded algebras $A_\Psi(Y)\twoheadrightarrow \mathcal B_\Psi(Y)$. In particular,
we obtain a surjective homomorphism of vector spaces $Y^{\tensor k}/\lr{B_t}_k\twoheadrightarrow \mathcal B_\Psi(Y)_k$ where
$B$ is the image of natural basis in~$Y^{\tensor k}$ under the map~$G$ from Lemma~\ref{le:inequalities on dimension}\eqref{le:ineq.b}. 
It follows from Lemma~\ref{lem:spec of ideals} that $\dim A_{\Psi_0}(Y)_k\ge \dim A_{\Psi}(Y)_k$.  Combining this with~\eqref{eq:dim ineq.1}
and the obvious inequality $\dim A_\Psi(Y)_k\ge \dim \mathcal B_\Psi(Y)_k$ we obtain $\dim A_\Psi(Y)_k\ge \dim A_{\Psi_0}(Y)_k$ which implies
$\dim A_\Psi(Y)_k= \dim \mathcal B_{\Psi}(Y)_k=\dim \mathcal B_{\Psi_0}(Y)_k$ for all~$k$. This completes the proof of Proposition~\ref{pr:flat Nichols}.
\end{proof}

Now we can complete the proof of Theorem~\ref{thm:flat quadratic algebra}. First, in the notation of Proposition we use Proposition \ref{pr:flat Nichols} with  $Y^*$  and  $-\Psi^*:Y^*\otimes Y^*\to Y^*\otimes Y^*$.

Taking $t=q-1$ and $\Psi_0^*=\tau$ in Proposition \ref{pr:flat Nichols}, we see that ${\mathcal B}_{-\Psi_0^*}(Y^*)=\Lambda(Y^*)=A_{-\Psi_0^*}(Y^*)$ hence  $A_{-\Psi^*}(Y^*)\cong {\mathcal B}_{-\Psi^*}(Y^*)$ is a flat deformation of the exterior algebra $\Lambda(Y^*)$. Taking into account that $(\ker  (-\Psi^*+1))^\perp=(\Psi-1)(Y\otimes Y)$ we see that the quadratic dual $A_{-\Psi^*}(Y^*)^!=S_\Psi(Y)$ is a flat deformation of $S(Y)$. 

Theorem~\ref{thm:flat quadratic algebra} is proved. 
\end{proof}

\subsection{Module algebras and semi-direct products}\label{subs:mod alg semidirect}
It is well-known that $U_q(\lieg)$ is a Hopf algebra with:

$\bullet$ The coproduct $\Delta:U_q(\lieg)\to U_q(\lieg)\otimes U_q(\lieg)$ given by:
$$\Delta(E_i)=E_i\otimes 1+K_i\otimes E_i,\qquad  \Delta(F_i)=F_i\otimes K_i^{-1}+1\otimes F_i.$$

$\bullet$ The counit $\varepsilon:U_q(\lieg)\to \CC(q)$ given by: 
$$\varepsilon(E_i)=\varepsilon(F_i)=0,\qquad \varepsilon(K_i^{\pm 1})=1.$$

$\bullet $ The antipode $S:U_q(\lieg)\to U_q(\lieg)$ given by
$$S(F_i)=-F_iK_i,\qquad S(E_i)=-K_i^{-1}E_i, \qquad S(K_i)=K_i^{-1}.$$

In particular, $U_q(\lieg)$ admits the (left) adjoint action on itself, which we denote by $(u,v)\mapsto (\ad u)(v)$. The action is given  by:
$$(\ad u)(v)=u_{(1)}vS(u_{(2)})\ ,$$
where $ \Delta(u)=u_{(1)}\otimes u_{(2)}$ in the Sweedler notation. By definition, 
$$(\ad K_i)(u)=K_iuK_i^{-1},\,(\ad E_i)(u)=E_iu-K_iuK_i^{-1}E_i,\, (\ad F_i)(u)=(F_iu-uF_i)K_i.$$
In particular, the quantum Serre relations can be written as
\begin{equation}\label{eq:qSerre-ad}
(\ad E_i)^{1-a_{ij}}(E_j)=0.
\end{equation}
We will also need the right action of~$U_q(\lie g)$ on itself.
Let $*$ be the unique anti-automorphism of $U_q(\lie g)$ defined by $E_i^*=E_i$, $F_i^*=F_i$ and $K_i^*=K_i^{-1}$. Then 
we define $\ad^* u=*\circ \ad u\circ *$.
In particular, we have 
$$
(\ad^* K_i)(u)=K_i u K_i^{-1},\, (\ad^* E_i)(u)=u E_i-E_i K_iu K_i^{-1},\, (\ad^* F_i)(u)=K_i^{-1}[u,F_i].
$$
and
\begin{equation}\label{eq:braid.ad}
T_i(E_j)=(\ad^* E_i)^{(-a_{ij})}(E_j).
\end{equation}
It is easy to see that $\ad^* u$ is in fact the right adjoint action for a different co-product on~$U_q(\lie g)$. 
Note that for all $i,j\in I$ 
\begin{equation}\label{eq:ad*-ad}
(\ad E_i)(E_j)=(\ad^* E_j)(E_i),
\end{equation}
while for all $i,j\in I$ and $w\in W$ such that $w\alpha_i=\alpha_j$ we have, by Lemma~\ref{lem:lus.br.id}
\begin{equation}\label{eq:Tw ad}
T_w((\ad E_i)(u))=(\ad E_j)(T_w(u)),\qquad T_w((\ad^* E_i)(u))=(\ad^*E_j)(T_w(u)).
\end{equation}

Given a bialgebra $U$, refer to an algebra in the category $U$-mod as a {\it module algebra} over $U$. The following Lemma is obvious.
\begin{lemma}\label{lem:modalgebra} Let ${\mathcal A}$ be a left module algebra over $U_q(\lieg)$. Then the action of Chevalley generators on ${\mathcal A}$ satisfies:
$$K_i(ab)=K_i(a)K_i(b),\, E_i(ab)=E_i(a)b+K_i(a)E_i(b),\,F_i(ab)=F_i(a)K_i^{-1}(b)+a F_i(b)$$  
for all $a,b\in {\mathcal A}$ and $i\in I$. 
\end{lemma}

\begin{definition} For any bialgebra $B$ and its module algebra ${\mathcal A}$ define the cross product ${\mathcal A}\rtimes B$ to be the the vector space ${\mathcal A}\otimes B$ with the associative product given by:
$$(a\otimes b)(a'\otimes b')=a\cdot (b_{(1)}(a))\otimes  b_{(2)}\cdot b'$$
for all $a\in {\mathcal A}$, $b\in B$ (where $\Delta(b)=b_{(1)}\otimes b_{(2)}$ in Sweedler notation). In what follows, we suppress tensors and will write  $a\cdot b$ instead of $a\otimes b$ and $b\cdot a$ instead of $(1\otimes b)(a\otimes 1)$ in the algebra  ${\mathcal A}\rtimes B$.
\end{definition}

Similarly, one can replace $B$ by a {\it braided bialgebra}, i.e., a bialgebra in a braided category ${\mathscr C}$ and ${\mathcal A}$ by a module algebra over $B$ in ${\mathscr C}$. Our main example is when  ${\mathscr C}_Q$ is the category of $Q$-graded vector spaces with the braiding $\Psi_{U,V}:U\otimes V\to V\otimes U$ for  $U,V\in \operatorname{Ob} {\mathscr C}_Q$  given by 
$$\Psi_{U,V}(u\tensor v)=q^{(\mu,\nu)}v\otimes u,\qquad u\in U_\mu,\,v\in V_\nu$$
where $(\cdot,\cdot)$ is the inner product on $Q$ given by 
\begin{equation}
\label{eq:killing form}
(\alpha_i,\alpha_j)=d_ia_{ij}
\end{equation}
where $C=(d_ia_{ij})_{i,j\in I}$ is a symmetrized Cartan matrix of a semisimple lie algebra $\lieg$. 

Let $Y=\CC(q)\otimes_\ZZ Q$. As G. Lusztig proved in \cite{Lus}, the Nichols-Woronowicz algebra ${\mathcal B}_{\Psi_{Y,Y}}(Y)$ (see Section~\ref{subsect:diamond})
is naturally isomorphic to $U_q^+(\lieg)$. The following obvious fact is parallel to Lemma~\ref{lem:modalgebra}.
\begin{lemma}\label{pr:braided semidirect}
Let ${\mathcal A}=\bigoplus\limits_{\nu\in Q} \mathcal A_\nu$ be a module algebra over (the braided bialgebra) $U_q^+(\lieg)$. Then 

\begin{enumerate}[{\rm(i)}]
\item\label{pr:braided semidirect.i}For any $a\in {\mathcal A}_\nu$, $b\in {\mathcal A}$, $i\in I$ one has
$$E_i(ab)=E_i(a)b+q^{(\alpha_i,\nu)}aE_i(b).$$

\item\label{pr:braided semidirect.ii}
The braided cross product ${\mathcal A}\rtimes U^+_q(\lieg)$ is the algebra generated by ${\mathcal A}$ and $U_q^+(\lieg)$ (and isomorphic to ${\mathcal A}\otimes U_q^+(\lieg)$ 
as a vector space) subject to the relations
$$E_i\cdot a=E_i(a) +q^{(\alpha_i,\nu)}a \cdot E_i$$
for all $a\in {\mathcal A}_\nu$, $i\in I$. 
In particular, if $\mathcal A$ is a PBW algebra, then so is~$\mathcal A\rtimes U^+_q(\lieg)$. 
\end{enumerate}
\end{lemma}

\begin{remark} In fact if ${\mathcal A}=\bigoplus_{\nu} {\mathcal A}_\nu$ is a $U_q(\lieg)$-module algebra with $K_i|_{{\mathcal A}_\nu}=q^{(\alpha_i,\nu)}$, then the braided cross product $
{\mathcal A}\rtimes U_q^+(\lieg)$ is simply the subalgebra of the ordinary cross product ${\mathcal A}\rtimes U_q(\lieg)$ generated by  ${\mathcal A}$ and $U_q^+(\lieg)$. Moreover, $\mathcal A\rtimes U_q(\lie g)\cong \mathcal A\rtimes U_q^+(\lie g)\tensor U_q^{\le 0}(\lie g)$ as a vector space.
\end{remark}

Using Lemma~\ref{lem:modalgebra} and~\citem{Lus}*{\S 3.1.5} we obtain for any $U_q(\lie g)$-module algebra $\mathcal A$ and $a,b\in \mathcal A$
\begin{equation}\label{eq:ad nilp.1}
\begin{split}
&E_i^{(r)}(a b)=\sum_{p=0}^r q_i^{p(r-p)} E_i^{(r-p)}(K_i^p(a))E_i^{(p)}(b) \\
&F_i^{(r)}(a b)=\sum_{p=0}^r q_i^{-p(r-p)} F_i^{(p)}(a) K_i^{-p}(F_i^{(r-p)}(b)),
\end{split}
\end{equation}
where $Y_i^{(r)}:=([r]_{q_i}!)^{-1} Y_i^r$. In particular, if $K_i(a)=v_i a$, then
\begin{equation}\label{eq:ad nilp.2}
E_i^{(r)}(a b)=\sum_{p=0}^r q_i^{p(r-p)}v_i^p E_i^{(r-p)}(a)E_i^{(p)}(b)
\end{equation}
Given $\bi=(i_1,\dots,i_k)\in I^k$ and $m=(m_1,\dots,m_k)\in\ZZ_{\ge 0}^k$, let $E_\bi^{(m)}=E_{i_1}^{(m_1)}\cdots E_{i_k}^{(m_k)}$. 
Using an obvious induction, we immediately obtain from~\eqref{eq:ad nilp.2} 
\begin{equation}\label{eq:ad nilp.3}
E_\bi^{(m)}(ab)=\sum_{m',m''\in\ZZ_{\ge 0}^k\,:\, m'+m''=m} q^{\frac12 \sum_{1\le r,s\le k} m'_r m''_s(\alpha_{i_r},\alpha_{i_s})} v^{m''}
E_\bi^{(m')}(a) E_\bi^{(m'')}(b), 
\end{equation}
where
$v^{m''}=\prod_{r=1}^k v_{i_r}^{m''_r}$.

Following M.~Kashiwara and
G.~Lusztig (\cites{Lus,Kas}), define for all $i\in I$ and for all~$u\in U_q^+(\lieg)$ the elements 
$r_i(u), {}_i r(u)\in U_q^+(\lie g)$ by
\begin{equation}\label{eq:defn-quasi-der}
 [u,F_i]=\frac{K_i \cdot {}_i r(u)- r_i(u)\cdot K_i^{-1}}{q_i-q_i^{-1}}.
\end{equation}
\begin{lemma}[\citem{Lus}*{Lemma~1.2.15, Proposition~3.1.6}]\label{lem:lusztig}
For all $i,j\in I$, $u,v\in U_q^+(\lie g)$
\begin{enumerate}[{\rm(i)}]
 \item\label{lem:lusztig.iii} ${}_i r(u)=r_i(u^*)^*$ and $r_i({}_j r(u))={}_j r(r_i(u))$;
 \item\label{lem:lusztig.ii'} $r_i(E_j)=\delta_{ij}$ and $r_i(u v)=r_i(u)K_i^{-1}vK_i+u r_i(v)$;
 \item\label{lem:lusztig.ii} If $u\in\ker {}_i r$ then $(\ad F_i)(u)=(q_i-q_i^{-1})^{-1} r_i(u)$;
 \item\label{lem:lusztig.iv} $
 {}_i r(\ad E_j(u))=q_i^{-a_{ij}} E_j\cdot {}_i r(u)- q_i^{a_{ij}} K_j\cdot {}_i r(u)\cdot K_j^{-1}\cdot E_j$;
\item\label{lem:lusztig.i} $\bigcap_{i\in I} \ker {}_i r=\bigcap_{i\in I} \ker r_i=\CC(q)$.
\end{enumerate}
\end{lemma}
Recall that for any $J\subset I$ the parabolic subalgebra $\lie p_J$ of~$\lie g$ is the Lie subalgebra generated by 
$\lie g_+$ and $f_j$, $j\in J$. Let $U_q(\lie p_J)$ be the subalgebra of $U_q(\lie g)$ generated by $U_q^+(\lie g)$
and by the $F_j,K_j^{\pm1}$, $j\in J$. Let 
$U_q(\lie g_J)=\lr{E_j,F_j,K_j^{\pm1}\,:\, j\in J}$ and define
$$
U_q(\lie r_J)=\{ u\in U_q^+(\lie g)\,:\, {}_j r(u)=0,\,\forall\, j\in J\}.
$$
Clearly $U_q(\lie p_J)$ is a Hopf subalgebra of~$U_q(\lie g)$, $U_q(\lie g_J)$ is a Hopf subalgebra of $U_q(\lie p_J)$, 
and $U_q(\lie r_J)$ is a subalgebra of $U_q(\lie p_J)$.
The following corollary is an immediate consequence of Lemma~\ref{lem:lusztig}.
\begin{corollary}[Quantum Levi factorization]\label{lem:quantum levi fact} 
\noindent
\begin{enumerate}[{\rm(i)}]
 \item 
$U_q(\lie r_J)$ is preserved by the adjoint action of $U_q(\lie g_J)$ on~$U_q(\lie p_J)$. In particular,
$U_q(\lie r_J)$ is a $U_q(\lie g_J)$-module algebra. 
\item $U_q(\lie p_J)= U_q(\lie r_J)\rtimes U_q(\lie g_J)$. 
\end{enumerate}
\end{corollary}

\section{Folding \texorpdfstring{$(\lie{so}_{2n+2},\lie{sp}_{2n})$}{(so\_2n+2,sp\_2n)} and
proof of Theorems~\ref{th:first folding} and~\ref{th:Dn Chevalley}}\label{sect:D_{n+1}}

\subsection{The algebras \texorpdfstring{$\mathcal U_{q,n}^+$}{U\_q+} and \texorpdfstring{$\mathcal U_{q,n}$}{U\_q,n}}\label{subs:alg Uqn}
In what follows we take $I=\{-1,0,\ldots,n-1\}$ for $\lieg=\lie{so}_{2n+2}$ so that $\sigma$ interchanges $-1$ and $0$ and fixes each $i=1,\ldots,n-1$. Accordingly, we set $I/\sigma:=\{0,1,\ldots,n-1\}$ for ${\lieg^\sigma}^\vee=\lie{sp}_{2n}$.

Let $\mathcal U_{q,n}$ be the associative $\CC(q)$-algebra generated by $U_q(\lie{sl}_n)$ and $w,z$ subject to the following relations
\begin{subequations}
\begin{gather}
[F_i,w]=0=[F_i,z],\qquad K_i w K_i^{-1}=q^{-2\delta_{i,1}} w,\qquad K_i z K_i^{-1}=q^{-\delta_{i,2}} z\label{eq:D_n rls.2}\\
[E_i, w]=0, \quad i\ne 1,\quad 
[E_i, z]=0,\quad i\ne 2,\qquad  [z,w]=0\label{eq:D_n rls.1}\\
[E_{1},[E_{1},[E_{1},w]_{q^{-2}}]]_{q^2}=0=[E_{2},[E_{2},z]_q]_{q^{-1}},\label{eq:D_n rls.3}\\
[w,[w,E_{1}]_{q^2}]_{q^{-2}}=-h wz,\quad [z,[E_{2},[E_{1},w]_{q^2}]_q]_q=[w,[E_{1},[E_{2},z]_q]_q]_{q^2},\label{eq:D_n rls.4}\\
2[z,[z,E_{2}]_q]_{q^{-1}}=h
(z [E_{1},E_{2}]_q w+w [E_{2},E_{1}]_{q^{-1}} z+ w E_{1}[z,E_{2}]_q
+[E_{2},z]_{q^{-1}}E_{1}w)\label{eq:D_n rls.5}
\end{gather}
\end{subequations}
where $h=q-q^{-1}$ and we abbreviate $[a,b]_v=ab-v ba$ and $[a,b]=[a,b]_1=ab-ba$ (with the convention that $E_2=0$ if $n=2$) and
the $E_i,F_i,K_i^{\pm1}$, $1\le i\le n-1$, are Chevalley generators of~$U_q(\lie{sl}_n)$.
Let $\mathcal U^+_{q,n}$ be the 
associative $\CC(q)$-algebra generated by the $E_i$, $1\le i\le n-1$ and by $w,z$ subject to the relations \eqref{eq:D_n rls.1}--\eqref{eq:D_n rls.5}.
Let $V$ be the standard $U_q(\lie{sl}_n)$-module. We denote $S_q(V\otimes V):=S_\Psi(V\otimes V)$ in the notation of Theorem \ref{thm:flat quadratic algebra}, where $\Psi:V^{\otimes 4}\to V^{\otimes 4}$ is the $U_q(\lie{sl}_n)$-equivariant map given by~\eqref{eq:psi braiding}.

The  following theorem is the main result of the section.
\begin{theorem} \label{th:D_n} 
\noindent
\begin{enumerate}[{\rm(i)}]
 \item\label{th:D_n.i} $S_q(V\tensor V)$ is a PBW algebra on any ordered basis of~$V\tensor V$.
 \item\label{th:D_n.ii}
 The algebra ${\mathcal U}_{q,n}^+$ is isomorphic to the braided cross product $S_q(V\tensor V)\rtimes U_q^+(\lie{sl}_n)$ and
 in particular is PBW.
 
 \item\label{th:D_n.ii'} The algebra $\mathcal U_{q,n}$ is isomorphic to the cross product $S_q(V\tensor V)\rtimes U_q(\lie{sl}_n)$
 and also to the tensor product of $\mathcal U_{q,n}^+$ and the subalgebra of $U_q(\lie{sl}_n)$ generated by the $F_i,K_i^{\pm 1}$, $i\in I$.
 In particular, $\mathcal U_{q,n}^+$ is a subalgebra of~$\mathcal U_{q,n}$.

\item\label{th:D_n.iv} The assignment 
$w\mapsto E_0$, $z\mapsto 0$ defines a homomorphism $\mu:{\mathcal U}_{q,n}\to U_q(\lie{sp}_{2n})$. Its image is 
the (parabolic) subalgebra of $U_q(\lie{sp}_{2n})$ generated by $U_q^+(\lie{sp}_{2n})$ and $K_i^{\pm 1}, F_i$, $1\le i\le n-1$.

\item\label{th:D_n.v}
The assignment 
$w\mapsto E_0 E_{-1}$ and
$$
z\mapsto 
\frac{1}{q h}([E_{-1},[E_{1},E_0]_q]_q+[E_{0},[E_{1},E_{-1}]_q]_q)
$$
defines an algebra homomorphism 
$\hat\iota:{\mathcal U}_{q,n}\to U_q(\lie{so}_{2n+2}).$ Its image is contained in
the (parabolic) subalgebra of $U_q(\lie{so}_{2n+2})$ generated by $U_q^+(\lie{so}_{2n+2})$ and $K_i^{\pm 1}, F_i$, $1\le i\le n-1$.
\item\label{th:D_n.iii} The assignments 
\begin{equation}\label{eq:D_n braid formulae}
T_{i}(w)=\begin{cases}(q+q^{-1})^{-1} [[w,E_1]_{q^{-2}},E_1],&i=1\\
          w,&i\not=1
         \end{cases}
\qquad 
T_{i}(z)=\begin{cases}[z,E_2]_{q^{-1}},&i=2\\
z,&i\not=2
         \end{cases}
\end{equation}
extend Lusztig's action~\eqref{eq:Lustig braid action} of the braid group $Br_{\lie{sl}_n}$ on $U_q(\lie{sl}_n)$ to an action on~$\mathcal U_{q,n}$ by algebra automorphisms. Moreover, $\mu$ and $\hat\iota$ are
$Br_{\lie{sl}_n}$-equvivariant.
\end{enumerate}
\end{theorem}
This theorem is proved in the rest of Section~\ref{sect:D_{n+1}}. 

\begin{remark}\label{rem:another quotient} It is interesting to observe a complement to Theorem \ref{th:D_n}\eqref{th:D_n.iv}: the quotient algebra ${\mathcal U}_{q,n}/\langle w \rangle$ is isomorphic to the (parabolic) subalgebra of $U_q(\lie{so}_{2n})$ generated by $U_q^+(\lie{s
0}_{2n})$ and $K_i,K_i^{-1}, F_i$, $i=1,\ldots,n-1$. 

\end{remark} 
\begin{remark}
Note that the subalgebra $S_q(V\tensor V)$ of~$\mathcal U_{q,n}$ is not preserved by the action of~$Br_{\lie{sl}_n}$. For example,
$T_3^2(w),T_2^2(z)\notin S_q(V\tensor V)$.
\end{remark}
\begin{remark}
The image of~$S_q(V\tensor V)$ under $\mu$ is isomorphic to $U_q(\lie r_J)\subset U_q(\lie{sp}_{2n})$, as defined in Section~\ref{subs:mod alg semidirect}, where 
$J=\{1,\dots,n-1\}$. Furthermore, $\hat\iota(S_q(V\tensor V))$ is a quantum deformation of the coordinate ring of $\mathcal M_{\le 2}$,
where $\mathcal M_{\le 2}$ is the variety of all matrices with the symmetric part of rank at most~$2$. 
Moreover, both homomorphisms are compatible with the cross-product structure, e.g. $\hat\iota(\mathcal U_{q,n})=\hat\iota(S_q(V\tensor V))\rtimes 
U_q(\lie{sl}_n)$.
\end{remark}

\subsection{Structure of algebra~\texorpdfstring{$S_q(V\tensor V)$}{S\_q(V(x)V)}.}\label{subs:struct Sq}
Let $\{v_i\}$, $1\le i\le n$ be the standard basis of the $n$-dimensional $U_q(\lie{sl}_n)$ module $V$. Let
$X_{i,j}=v_i\tensor v_j$ be the standard basis of~$V\tensor V$. In particular, we have
\begin{equation}\label{eq:sln Vtens V.1}
\begin{aligned}
&E_i(v_j)=\delta_{i,j-1} v_{j-1},&\qquad& E_i(X_{j,k})=\delta_{i,j-1} X_{j-1,k}+\delta_{i,k-1} q^{\delta_{i,j}-\delta_{i,j-1}} X_{j,k-1}
\\
&F_i(v_j)=\delta_{i,j}v_{j+1},&\qquad& F_i(X_{j,k})=\delta_{i,j} q^{\delta_{i,k-1}-\delta_{i,k}} X_{j+1,k}+\delta_{i,k} X_{j,k+1}
\end{aligned}
\end{equation}
for all $1\le i<n$ and for all $1\le j,k\le n$. 
Let $T:V\tensor V\to V\tensor V$ be the $\CC(q)$-linear map defined by 
$$
T(X_{ij})=q^{-\delta_{ij}}X_{ji},\qquad T(X_{ji})=q^{\delta_{ij}}X_{ij}-(q-q^{-1})X_{ji}
$$
for all $1\le i\le j\le n$.
It is well-known that $T$ is an isomorphism of $U_q(\lie{sl}_n)$-modules, satisfies $(T-q^{-1})(T+q)=0$ and the braid equation on
$V^{\tensor 3}$. Define 
$\Psi_i:V^{\tensor k}\to V^{\tensor k}$ by $\Psi_i=1^{\tensor i-1}\tensor T\tensor 1^{\tensor k-i-1}$. Then
$\Psi_i$ are isomorphisms of $U_q(\lie{gl}_n)$-modules.
Let
$$
\Psi=\Psi_2\Psi_1\Psi_3\Psi_2+(q-q^{-1})(
    \Psi_1\Psi_2\Psi_1+ \Psi_1\Psi_3\Psi_2)+(q-q^{-1})^2 \Psi_1\Psi_2.
$$
It will be convenient for us to regard~$\Psi$ as an element of the Hecke algebra~$H(S_n)$. Recall that
$H(S_n)$ is the quotient of the group algebra over~$\CC(q)$ of the braid group~$Br_{\lie{sl}_n}$ by the ideal generated by 
$(T_i-q^{-1})(T_i+q)$, $1\le i\le n-1$. In particular, $V^{\tensor n}$ is an $H(S_n)$-module.
A well-known result of Jimbo (\citem{Jimbo}*{Proposition~3}) provides
a quantum analogue of Schur-Weyl duality, namely the image of~$U_q(\lie{gl}_n)$ in~$\End V^{\tensor n}$ is the centralizer of 
the image of~$H(S_n)$ and vice versa. It is also well-known that the Hecke algebra $H(S_n)$ is semi-simple.

\begin{proof}[Proof of Proposition~\ref{prop:PsiProp}]
Since $H(S_n)$ is semi-simple, to prove part~\eqref{prop:PsiProp.i} (respectively, part~\eqref{prop:PsiProp.ii}) of
Proposition~\ref{prop:PsiProp}, 
it is sufficient to show that these identities hold in any simple finite dimensional representation 
of the Hecke algebra $H(S_6)$ (respectively, $H(S_4)$). For, we use a realization of the multiplicity 
free direct sum of all simple finite dimensional $H(S_n)$-modules, known as the Gelfand model, constructed
in~\citem{APR}*{Theorem~1.2.2}, which we briefly review for the reader's convenience. 

Let $\mathscr I_n$ be the set of involutions 
in the symmetric group~$S_n$ and let $\mathscr I_{n,k}\subset \mathscr I_n$ be the set 
of all involutions containing~$k$ cycles of length~$2$ so each $\mathscr I_{n,k}$ 
is an orbit for the action of~$S_n$ on~$\mathscr I_n$.
Given $w\in\mathscr I_{n,k}$,
one defines $\hat\ell(w)=\min\{ \ell(v)\,:\, v w v^{-1}=\prod_{i=1}^k s_{2i-1}\}$. Let
$\mathscr{V}_{n}^{(k)}=\operatorname{Span}\{C_w\,:\, w\in \mathscr I_{n,k}\}$ and set~$\mathscr V_n=\bigoplus_{0\le k\le n/2} \mathscr V_n^{(k)}$.
Then 
$$
T_i(C_w)=\begin{cases}
          -q C_{w},& s_i w s_i=w,\, \ell(ws_i)<\ell(w)\\
          q^{-1} C_w, & s_i w s_i=w,\, \ell(w)<\ell(ws_i)\\
          q C_{s_i w s_i}-(q-q^{-1})C_{w},&s_i w s_i\not=w,\, \hat\ell(w)<\hat\ell(s_i w s_i)\\
          q^{-1} C_{s_i w s_i},& s_i w s_i\not=w,\, \hat \ell(s_i w s_i)<\hat\ell(w)
         \end{cases}
$$
defines a representation of the Hecke algebra~$H(S_n)$ on $\mathscr{V}_{n}$ which 
realizes the Gelfand model for~$H(S_n)$.
Clearly, $\mathscr V_n^{(k)}$ is an $H(S_n)$-submodule of~$\mathscr V_n$ and
$\mathscr{V}_{n}^{(0)}$ is 
the trivial $H(S_n)$-module.
A straightforward computation then shows that the matrix of~$\Psi$ on $\mathscr{V}_{4}^{(1)}$ with respect to the 
basis $C_{(i,j)}$, $1\le i<j\le 4$, is
$$
\begin{pmatrix}
    0 & q^3 h & -q^3 h & 0 &
      0 & q^4 \\
    0 & -q^2 & 0 & 0 & 0 & 0 \\
    -q^{-3}h & -q^{-2} h^2 &
      -h^2 & -q^{-2} & 0 & q
      h \\
    0 & 0 & -q^2 & 0 & 0 & 0 \\
    -q^{-3}h & -q^{-2}h^2 &
      -h^2 & 0 & -q^{-2} &
      -q^{-1}h \\
    q^{-4} & q^{-3}h & q^{-1}h & 0 &
      0 & 0
\end{pmatrix}
$$
while $\Psi|_{\mathscr{V}_{4}^{(2)}}=\id$. Here we abbreviate $h=q-q^{-1}$. Part~\eqref{prop:PsiProp.ii} is now straightforward. Part~\eqref{prop:PsiProp.i}
is checked similarly and we omit the details.

It remains to prove~\eqref{prop:PsiProp.iii}. Let $\tau=\tau_{V\tensor V,V\tensor V}$ be the permutation of factors.
Note that by the quantum Schur-Weyl duality, the vector subspace $(V^{\tensor m})^+$ of $U_q(\lie{sl}_n)$-highest weight vectors in~$V^{\tensor m}$ is isomorphic to the direct sum 
of simple $H(S_m)$-modules $\mathscr S^\lambda$, where $\lambda$ runs over the set of all partitions of 
$m$ with at most $\dim V$ non-zero parts. In particular,
if $\dim V\ge m$ then $(V^{\tensor m})^+\cong \mathscr V_m\cong \bigoplus_{\lambda} \mathscr S^\lambda$.
To complete the argument, we need the following result, which is an immediate consequence of Schur-Weyl duality.
\begin{lemma}
Let $\Psi'\in H(S_m)$ be such that $\Psi'$ is specializable at~$q=1$ on~$\mathscr V_m$ with respect to 
the basis $C_w$, $w\in \mathscr I_m$ and suppose that $\dim \Psi'(\mathscr S^\lambda)=\dim \Psi'|_{q=1}(\mathscr S^\lambda)$ for all
partitions $\lambda$ of~$m$. Then for any $V$, $\dim\Psi'(V^{\tensor m})=\dim\Psi'|_{q=1}(V^{\tensor m})$.
\end{lemma}
Recall that $\mathscr V_4^{(0)}=\mathscr S^{(4)}$ and it is easy to see that $\mathscr V_4^{(1)}=\mathscr S^{(2,1^2)}\oplus\mathscr S^{(3,1)}$
while $\mathscr V_4^{(2)}=\mathscr S^{(2,2)}\oplus \mathscr S^{(1^4)}$.
Therefore, $(\Psi-1)(\mathscr S^\lambda)=(\tau-1)(\mathscr S^\lambda)=0$
for $\lambda\in\{(4),(2,2),(1^4)\}$. Finally, one can easily show that $\dim (\Psi-1)(\mathscr S^\lambda)=\dim(\tau-1)(\mathscr S^\lambda)=2$ for $\lambda\in\{(2,1^2),(3,1)\}$.
\end{proof}

We can now prove the first part of Theorem~\ref{th:D_n}.
\begin{proof}[Proof of~part~\eqref{th:D_n.i} of Theorem~\ref{th:D_n}]
By~Proposition~\ref{prop:PsiProp}, $\Psi$ satisfies the braid relation and condition~\eqref{thm:flat quadratic algebra.ii} of~Theorem~\ref{thm:flat quadratic algebra}. 
Since $\Psi_i$ specializes to the transposition of factors with respect to the standard basis of~$V^{\tensor 4}$, it follows that $\Psi$ specializes 
to the permutation of factors in~$(V^{\tensor 2})^{\tensor 2}$. It remains to apply Theorem~\ref{thm:flat quadratic algebra}. 
\end{proof}

\begin{proposition}\label{pr:SqPres}
The algebra $S_q(V\tensor V)$ is generated by the elements $X_{ij}$, $1\le i,j\le n$, subject to the following relations
for all $1\le i\le j\le k\le l\le n$
{\footnotesize
\begin{align*}
&X_{ij} X_{kl}=q^-_{ikj} q^-_{ilj}
    X_{kl} X_{ij}+h q^{-d_{ij}} (q^-_{ikj}
    X_{kj} X_{il}+q^-_{ikl} q^-_{ilj} X_{lj}
    X_{ki})\\&\qquad\qquad+h^2 (q^-_{ikl} q^-_{jil} X_{jl}
    X_{ki}+q^-_{ikj} q^-_{jil} X_{kj}
    X_{li})+h^3 q^{-d_{ik}} q^-_{jil} X_{jk}
    X_{li}\\
&X_{ij} X_{lk}=q^-_{ikj}
    q^-_{ilj} X_{lk} X_{ij}+h q^{-d_{ij}} (q^-_{ilj}
    X_{lj} X_{ik}+q^-_{ikj}q^\pm_{kli} X_{kj} X_{li}
    )+h^2 q^-_{jil}q^\pm_{lki} X_{jk} X_{li}
    \\
&X_{ji}X_{kl}=q^-_{ikj} q^-_{ilj} X_{kl} X_{ji}+h
    q^{-d_{il}} (q^-_{ikl} X_{jl} X_{ki}+q^-_{ikj}
    X_{kj} X_{li})+h^2 q^-_{kil} X_{jk}
    X_{li}\\
&X_{ji} X_{lk}=q^-_{ikj}
    q^-_{ilj} X_{lk} X_{ji}+h q^{-d_{ik}} (q^-_{ilj}
    X_{lj} X_{ki}+q^\pm_{kli} X_{jk} X_{li}
    )\\
&X_{ik}X_{jl}=q^-_{jil}q^\pm_{jkl} X_{jl} X_{ik}
    +h q^{-d_{ik}} (-q^-_{ilj} X_{kl}
    X_{ij}+q^-_{jil} q^-_{jlk} X_{lk}
    X_{ji}+q^\pm_{kji} X_{jk} X_{il} )\\&\qquad\qquad+h^2
    q^-_{jik} (q^{-d_{il}+d_{jk}} X_{jk}
    X_{li}-X_{kj} X_{il}+q^-_{ilj} X_{kl}
    X_{ji})\\
&X_{ik}X_{lj}=q^-_{jil}q^\pm_{jkl} X_{lj} X_{ik}
    +h q^{-d_{ij}} q^\pm_{jki} q^\pm_{jli} X_{jk} X_{li}
    \\
&X_{ki}
    X_{jl}=q^-_{jil}q^\pm_{jkl} X_{jl} X_{ki}
    +h q^{d_{jk}} q^-_{jil} X_{jk}
    X_{li}\\
&X_{ki} X_{lj}=q^-_{jil}q^\pm_{jkl}
    X_{lj} X_{ki} +h q^{-d_{ij}}q^\pm_{jli}
    X_{kj} X_{li}
    \\
&X_{il}
    X_{jk}=q^\pm_{lji}
    q^\pm_{lki}X_{jk} X_{il} +h q^{-d_{il}} q^-_{ikj} (q^{-d_{ij}+d_{kl}}
    X_{kl} X_{ji}-X_{lk} X_{ij})\\&\qquad\qquad+h^2 (q^-_{jil}
    q^{-d_{ik}+d_{jl}} X_{jl} X_{ki}-X_{lj}
    X_{ik})\\
&X_{il} X_{kj}= q^\pm_{lji} q^\pm_{lki}X_{kj}
    X_{il}+h q^{d_{jk}}
    q^-_{jil} (q^{-d_{ik}+d_{jl}} X_{jl}
    X_{ki}-X_{lj} X_{ik})\\
&X_{li}
    X_{jk}=q^\pm_{lji}
    q^\pm_{lki} X_{jk} X_{li} \\
&X_{li}
    X_{kj}=q^\pm_{lji}
    q^\pm_{lki} X_{kj} X_{li} 
\end{align*}
}
\noindent
where $h=q-q^{-1}$, $q_{abc}^+=q^{\delta_{ab}+\delta_{bc}}$, $q_{abc}^-=(q_{abc}^+)^{-1}$ and $q^\pm_{abc}=q^{\delta_{ab}-\delta_{bc}}$.
\end{proposition}
\begin{proof}
One can show that for all $1\le i\le j\le k\le l\le n$
{\footnotesize
\begin{align*}
&\Psi(X_{ij} X_{kl})=q^-_{ikj} q^-_{ilj}
    X_{kl} X_{ij}+h q^-_{ikj} (q^{-\delta_{jl}}
    X_{ki} X_{lj}+q^{-\delta_{ij}} X_{kj} X_{il})+h^2 q^-_{ikj} X_{ki} X_{jl}\\
&\Psi(X_{ij}
    X_{lk})=q^-_{ikj} q^-_{ilj} X_{lk} X_{ij}+h
    q^-_{ilj} (q^{-\delta_{jk}} X_{li} X_{kj}+q^{-\delta_{ij}}
    X_{lj} X_{ik})+h^2 q^-_{ilj} X_{li} X_{jk}\\
&\Psi(X_{ik}
    X_{jl})=q^-_{ilk} q^\pm_{kji} X_{jl}X_{ik}+h (q^{-\delta_{kl}} q^\pm_{kji} X_{ji}X_{lk}+q^{-\delta_{ij}} q^\pm_{jki} X_{jk}X_{il}
    -q^{-\delta_{ik}} q^-_{ilj} X_{kl}X_{ij})\\&\qquad+h^2 (q^\pm_{kji} X_{ji}X_{kl}-q^{-\delta_{ik}-\delta_{jl}} X_{ki}X_{lj}
    -q^-_{jik} X_{kj} X_{il})-h^3 q^{-\delta_{ik}} X_{ki} X_{jl}\\
&\Psi(X_{ik} X_{lj})=q^-_{ilk} q^\pm_{kji} X_{lj}X_{ik}+h q^{\delta_{jk}} q^-_{ilk} X_{li}X_{jk}\\
&\Psi(X_{il}X_{jk})=q^-_{jik} X_{jk} X_{il} q^+_{jlk}+h (-q^{-\delta_{il}} q^-_{ikj} X_{lk} X_{ij}+q^{-\delta_{ij}}q^+_{jlk}X_{ji} X_{kl})\\&\qquad-h^2(q^{-\delta_{il}-\delta_{jk}} X_{li} X_{kj}+q^-_{jil}
    X_{lj} X_{ik})-h^3 q^{-\delta_{il}} X_{li}
    X_{jk}\\
&\Psi(X_{il}X_{kj})=q^-_{jik} X_{kj} X_{il} q^+_{jlk}-h
    (q^{-\delta_{il}} q^\pm_{kji} X_{lj}
    X_{ik}-q^{-\delta_{ik}}q^+_{jlk} X_{ki} X_{jl}
    )-h^2 q^{-\delta_{il}+\delta_{jk}} X_{li}
    X_{jk}\\
&\Psi(X_{ji}
    X_{kl})=q^-_{ikj} q^-_{ilj} X_{kl} X_{ji}+h
    q^-_{ikj} (q^{\delta_{ij}} X_{ki} X_{jl}+q^{-\delta_{il}}
    X_{kj} X_{li})\\
&\Psi(X_{ji}
    X_{lk})=q^-_{ikj} q^-_{ilj} X_{lk} X_{ji}+h
    q^-_{ilj} (q^{\delta_{ij}} X_{li} X_{jk}+q^{-\delta_{ik}}
    X_{lj} X_{ki})\\
&\Psi(X_{jk}
    X_{il})=q^\pm_{ijl}q^\pm_{ikl} X_{il} X_{jk}\\&\qquad+h (q^{\delta_{ij}}(q^\pm_{ikj} X_{ik} X_{jl}
    +q^\pm_{ikl} X_{ij} X_{lk}
    )-q^{-\delta_{il}} (q^-_{kjl} X_{kl}
    X_{ji}+q^\pm_{ikl}X_{jl} X_{ik} ))
    \\&\qquad+h^2 (X_{jk}
    X_{il}+q^{-\delta_{il}-\delta_{jk}} X_{kj}
    X_{li}+q^\pm_{ijk}X_{ki} X_{jl}
    +q^\pm_{ikl}X_{ji} X_{lk}-q^+_{jik}X_{ij} X_{kl})-h^3
    q^{\delta_{ik}} X_{ji} X_{kl}\\
&\Psi(X_{jk}
    X_{li})=q^-_{jlk}q^+_{jik} X_{li} X_{jk}
    \\
&\Psi(X_{jl} X_{ik})=q^\pm_{ijk} q^+_{ilk}X_{ik}
    X_{jl} +h (q^{-\delta_{jl}}
    q^-_{ikj} X_{lk} X_{ji}-q^{\delta_{ij}} X_{ij}
    X_{kl} q^+_{ilk}+q^{-\delta_{ik}} X_{jk} X_{il}
    q^+_{ilk})\\&\qquad-h^2 (q^{-\delta_{ik}-\delta_{jl}} X_{lj}
    X_{ki}+q^\pm_{ijl}X_{li} X_{jk}
    +q^+_{ilk}X_{ji} X_{kl})\\
&\Psi(X_{jl}X_{ki})=q^\pm_{ijk}
    q^+_{ilk}X_{ki} X_{jl} -h q^{-\delta_{jl}}q^+_{jik} X_{li} X_{jk}\\
&\Psi(X_{ki} X_{jl})=q^-_{ilk}
    q^\pm_{kji} X_{jl} X_{ki}+h (q^{-\delta_{il}}
    q^\pm_{kji} X_{jk} X_{li}-q^{-\delta_{ij}} q^-_{ilj}
    X_{kl} X_{ji}+q^{-\delta_{ij}} X_{ji} X_{kl}
    q^+_{ikj})\\&\qquad-h^2 (X_{ki} X_{jl}+q^-_{jil}
    X_{kj} X_{li})\\
&\Psi(X_{ki}
    X_{lj})=q^-_{ilk} q^\pm_{kji} X_{lj}
    X_{ki}+h q^{\delta_{ik}} q^-_{ilk} X_{li}
    X_{kj}\\
&\Psi(X_{kj}
    X_{il})=q^-_{jlk} q^+_{jik} X_{il} X_{kj}
    \\&\qquad-h
    (q^{\delta_{jk}} q^-_{ilk} X_{jl} X_{ki}+q^{\delta_{ij}}
    (q^-_{ilj} X_{kl} X_{ij}-q^+_{ikj} X_{ij}
    X_{kl})-q^{-\delta_{jl}}q^+_{jik} X_{ik} X_{lj})
    \\&\qquad-h^2 (q^{-\delta_{il}+\delta_{jk}} X_{jk}
    X_{li}+X_{kj} X_{il}-q^-_{ilj} X_{kl}
    X_{ji}+q^\pm_{ijl}X_{ki} X_{lj}+q^+_{ikj}X_{ji} X_{kl})+h^3
    q^{-\delta_{il}} X_{kj} X_{li}\\
&\Psi(X_{kj} X_{li})=q^-_{jlk} q^+_{jik}X_{li} X_{kj}\\
&\Psi(X_{kl} X_{ij})=q^+_{ikj} q^+_{ilj}X_{ij}
    X_{kl} -h (q^{-\delta_{kl}}
    q^\pm_{kji} X_{lj} X_{ki}+q^{-\delta_{ij}}q^+_{ilj}
    X_{kj} X_{il} )\\&\qquad-h^2 (q^\pm_{ikl}X_{li}
    X_{kj} +q^+_{ilj}X_{ki} X_{jl})\\
&\Psi(X_{kl}
    X_{ji})=X_{ji} X_{kl} q^+_{ikj} q^+_{ilj}-h
    (q^{\delta_{ij}}q^+_{ilj} X_{ki} X_{jl}+q^{-\delta_{kl}}q^+_{jik}
    X_{li} X_{kj} )\\
&\Psi(X_{li} X_{jk})=q^-_{jik} X_{jk} X_{li}
    q^+_{jlk}-h (q^{-\delta_{ij}} q^-_{ikj} X_{lk}
    X_{ji}-q^{-\delta_{ij}}q^+_{ilj} X_{ji} X_{lk}
    )\\&\qquad-h^2 (X_{li} X_{jk}+q^-_{jik}
    X_{lj} X_{ki})\\
&\Psi(X_{li}
    X_{kj})=q^-_{jik} X_{kj} X_{li} q^+_{jlk}-h
    (q^{-\delta_{ij}} q^\pm_{jki} X_{lj}
    X_{ki}-q^{-\delta_{ik}}q^+_{ilk} X_{ki} X_{lj})-h^2 X_{li} X_{kj}\\
&\Psi(X_{lj} X_{ik})=X_{ik} X_{lj}
    q^\pm_{ijk} q^+_{ilk}-h (q^{\delta_{ij}} q^-_{ikj}
    X_{lk} X_{ij}-q^{\delta_{ij}}q^+_{ilj} X_{ij} X_{lk}
    +q^{-\delta_{ik}}q^+_{jlk} X_{jk} X_{li})\\&\qquad-h^2 (X_{lj} X_{ik}-q^-_{ikj}
    X_{lk} X_{ji}+q^\pm_{ijk}X_{li} X_{kj}-q^+_{ilj}X_{ji} X_{lk})+h^3
    q^{-\delta_{ik}} X_{lj} X_{ki}\\
&\Psi(X_{lj} X_{ki})=q^\pm_{ijk} q^+_{ilk}X_{ki} X_{lj}
    -h q^{\delta_{ij}}q^\pm_{ikj} X_{li}
    X_{kj} \\
&\Psi(X_{lk}
    X_{ij})=q^+_{ikj} q^+_{ilj}X_{ij} X_{lk} -h
    (q^{-\delta_{ij}}q^+_{ikj} X_{lj} X_{ik}+q^{-\delta_{ij}}q^+_{jlk}
    X_{kj} X_{li} )\\&\qquad+h^2
    (q^\pm_{kji} X_{lj} X_{ki}-X_{li}
    X_{jk} q^+_{ikj}-X_{ki} X_{lj}
    q^+_{ilk})+h^3 q^{\delta_{ik}} X_{li}
    X_{kj}\\
&\Psi(X_{lk}
    X_{ji})=X_{ji} X_{lk} q^+_{ikj} q^+_{ilj}-h
    (q^{\delta_{ij}}q^+_{ikj} X_{li} X_{jk}+q^{\delta_{ij}}q^+_{ilk}
    X_{ki} X_{lj})+h^2 q^+_{jik}X_{li}
    X_{kj} 
\end{align*}
}
Since the quotient~$\mathcal S=(V\tensor V)/(\Psi-1)(V\tensor V)$ is a flat deformation of~$S^2(V\tensor V)$, the canonical images of 
the $X_{ab}X_{cd}$ with $(a,b)\preccurlyeq (c,d)$  (where the order is defined by 
$(a,b)\prec (c,d)$ if $\min(c,d)<\min(a,b)$ or $\min(a,b)=\min(c,d)$ and
$\max(c,d)<\max(a,b)$, 
while $(i,j)\preccurlyeq (j,i)$, for all $i\le j$) 
form a basis of $\mathcal S$. Using this basis we obtain the formulae in the Proposition from the above formulae
for~$\Psi$.
\end{proof}
\begin{remark}
It is easy to check that the quotient of $S_q(V\tensor V)$ by the ideal 
generated by the elements $X_{ij}-q X_{ji}$, $1\le i<j\le n$
(respectively, by the elements $X_{ij}+q^{-1} X_{ji}$, $1\le i<j\le n$ and $X_{ii}$, $1\le i\le n$)
is isomorphic to the algebra $S_q(S^2 V)$ (respectively, $S_q(\Lambda^2 V)$;
cf. \citem{Kam}*{Teorem~0.2} and~\citem{Str}*{(1.1)}, respectively, and also~\cites{Nou,FRT,Zwi}).
\end{remark}

We can now prove Corollary~\ref{cor:deformed Yakimov's bracket}.
\begin{proof}
The algebra $S_q(V\tensor V)$ is clearly optimal specializable with respect to its PBW basis on the~$X_{ij}$, $1\le i,j\le n$ with the
total order defined as in~Proposition~\ref{pr:SqPres}. It remains to apply Propositions~\ref{prop:Specialization}\eqref{prop:Spec.ii} and~\ref{pr:SqPres}.
\end{proof}

\subsection{Cross product structure of~\texorpdfstring{$\mathcal U_{q,n}^+$}{U\_q,n+} and~\texorpdfstring{$\mathcal U_{q,n}$}{U\_q,n}}\label{subs:Pr semidirect}
In this section we will use the usual numbering of nodes in the Dynkin diagram of $\lie{so}_{2n+2}$, that is, the simple root
$\alpha_{n-1}$ corresponds to the triple node. Retain the notations of Section~\ref{subs:struct Sq}.
\begin{proposition}[Theorem~\ref{th:D_n}(\ref{th:D_n.ii},\ref{th:D_n.ii'})]\label{pr:crossprod isom}
\noindent
\begin{enumerate}[{\rm(i)}]
 \item\label{pr:crossprod isom.i} The natural homomorphism $\mathcal U_{q,n}^+\to \mathcal U_{q,n}$ is injective and as vector spaces
 $\mathcal U_{q,n}\cong \mathcal U_{q,n}^+\tensor U_q^{\le 0}(\lie{sl}_n)$, where 
 $U_q^{\le 0}(\lie{sl}_n)$ is defined as in Section~\ref{subsect:quantum enveloping}.
 \item\label{pr:crossprod isom.ii} The assignment
$w\mapsto w':=X_{n,n}$, $z\mapsto z':=q(X_{n-1,n}-q X_{n,n-1})$ defines isomorphisms of algebras 
$\psi:\mathcal U_{q,n}\to S_q(V\tensor V)\rtimes U_q(\lie{sl}_n)$
and $\psi_+:\mathcal U_{q,n}^+\to S_q(V\tensor V)\rtimes U_q^+(\lie{sl}_n)$.
\end{enumerate}
\end{proposition}
\begin{proof}
First we prove that the elements $w',z'$ satisfy the relations~\eqref{eq:D_n rls.1}--\eqref{eq:D_n rls.3}.
It follows from~\eqref{eq:sln Vtens V.1} that
$w'$ (respectively, $z'$) is a lowest weight vector of the $U_q(\lie{sl}_n)$-submodule of~$V\tensor V$ isomorphic to $V_{2\varpi_1}$ 
(respectively, $V_{\varpi_2}$), where $\varpi_i$ is the $i$th fundamental weight of~$\lie{sl}_n$. In particular, we have 
\begin{equation}\label{eq:act}
E_{i}^{2\delta_{i,n-1}+1}(w')=0=E_i^{\delta_{i,n-2}+1}(z'),\qquad F_i(w')=F_i(z')=0.
\end{equation}
Using Lemma~\ref{pr:braided semidirect}\eqref{pr:braided semidirect.ii} we immediately conclude that 
$w'$ and~$z'$ satisfy~\eqref{eq:D_n rls.2}, \eqref{eq:D_n rls.3} and the first two relations in~\eqref{eq:D_n rls.1}.
To prove the last relation in~\eqref{eq:D_n rls.1} note that 
$$
[w',z']=[X_{n,n},qX_{n-1,n}-q^2 X_{n,n-1}]=0
$$
since $X_{n,n-1}X_{nn}=q^2 X_{nn}X_{n,n-1}$ and $[X_{n-1,n},X_{nn}]=q^2(q-q^{-1})X_{nn}X_{n,n-1}$ by Proposition~\ref{pr:SqPres}.
To prove the first relation in~\eqref{eq:D_n rls.4}, note that 
\begin{align*}
[w',[w',E_{n-1}]_{q^{-2}}]_{q^2}&=[[E_{n-1},w']_{q^{-2}}],w']_{q^2}=
[X_{n-1,n}+q^{-1}X_{n,n-1},X_{n,n}]_{q^2}\\&=(1-q^2)(X_{n,n}X_{n-1,n}-q X_{n,n}X_{n,n-1})=(q^{-1}-q)w'z'.
\end{align*}
The remaining identities are checked similarly. Using Lemma~\ref{pr:braided semidirect}\eqref{pr:braided semidirect.ii},
we rewrite them in the form $\sum Y_i m_i$, where $m_i\in\{1,E_{n-1},E_{n-2},E_{n-1}E_{n-2},E_{n-2}E_{n-1}\}$ 
and in particularly are linearly independent and
$Y_i\in S_q(V\tensor V)$. Then we check that $Y_i=0$ which can be done either using the presentation from Proposition~\ref{pr:SqPres} or
by observing that $\Im(\Psi-1)=\ker((\Psi+q^2)(\Psi+q^{-2}))$. This is a rather tedious, albeit simple, computation, which
was performed on a computer.

Thus, we proved that $\psi:\mathcal U_{q,n}\to S_q(V\tensor V)\rtimes U_q(\lie{sl}_n)$ is a surjective homomorphism of algebras. 
The same argument shows that we have a surjective homomorphism of algebras $\psi_+:\mathcal U_{q,n}^+\to S_q(V\tensor V)\rtimes U_q^+(\lie{sl}_n)$.

To complete the proof of the proposition, we prove first that $\psi_+$ is an isomorphism. Let $\mathcal F$ be 
the free algebra on the $E_i$, $1\le i\le n-1$, $w$ and~$z$ and define a grading on~$\mathcal F$ 
by $\deg E_i=\deg w=1$, $\deg z=2$.  Let~$\mathcal I_q$ be the kernel of the structural homomorphism $\mathcal F\twoheadrightarrow \mathcal U^+_{q,n}$.
It is easy to see that $\mathcal I_q$ is homogeneous with respect to this grading.
Regard $S_q(V\tensor V)\rtimes U_q^+(\lie{sl}_n)$ as a graded algebra with the 
grading induced by the homomorphism~$\psi_+$. By Lemma~\ref{lem:spec of ideals} we have 
$\dim( \mathcal U^+_{q,n})_k\le \dim(\mathcal F/\mathcal I_1)_k$ for all~$k$ where $\mathcal I_1$ is the specialization of~$\mathcal I_q$
at $q=1$. On the other hand, it is easy to see that $\mathcal F/\mathcal I_1$ is isomorphic to $U(\lie n)$ where 
$\lie n=(V\tensor V)\rtimes (\lie{sl}_n)_+$, which we can regard as a graded Lie algebra with the grading compatible 
with that on~$\mathcal U^+_{q,n}$. Since both $U(\lie n)$ and $S_q(V\tensor V)\rtimes U_q^+(\lie{sl}_n)$ are 
PBW algebras on the set of the same cardinality, it follows that $\dim U(\lie n)_k=\dim(S_q(V\tensor V)\rtimes U_q^+(\lie{sl}_n))_k$
for all~$k$. This and the obvious inequality $\dim (\mathcal U^+_{q,n})_k\ge \dim(S_q(V\tensor V)\rtimes U_q^+(\lie{sl}_n))_k$ proves the
second assertion in part~\eqref{pr:crossprod isom.ii}.

Let $\mathcal U_{q,n}'{}^+$ be the subalgebra of~$\mathcal U_{q,n}$ generated by the $E_i$, $i\in I$ and by $w$, $z$. Clearly,
we have a canonical surjective homomorphism $\pi:\mathcal U_{q,n}^+\to \mathcal U_{q,n}'{}^+$ and $\psi_+=\psi\circ\pi$. Since 
$\psi_+$ is an isomorphism and both $\psi$ and $\pi$ are surjective, it follows that $\pi$ is an isomorphism and proves 
the first assertion in part~\eqref{pr:crossprod isom.i}. To establish the 
remaining assertions, we need the following easy Lemma.
\begin{lemma}
Let $\psi:A\to B$ be a surjective homomorphism of algebras. Let $A^\pm$ (respectively, $B^\pm$) be subalgebras of 
$A$ (respectively, $B$) such that the multiplication map $A^+\tensor A^-\to A$ (respectively, $B^+\tensor B^-\to B$) is surjective (respectively, bijective).
Suppose that the restriction of 
$\psi$ to $A^\pm$ is an isomorphism onto~$B^\pm$. Then $\psi$ is an isomorphism of algebras and $A\cong A^+\tensor A^-$ as vector spaces.
\end{lemma}
Applying this Lemma with $A=\mathcal U_{q,n}$, $B=S_q(V\tensor V)\rtimes U_q(\lie{sl}_n)$,
$A^+=\mathcal U_{q,n}^+$, $A^-=B^-=U_q^{\le 0}(\lie{sl}_n)$ and $B^+=S_q(V\tensor V)\rtimes U_q^+(\lie{sl}_n)$ completes the proof of the Proposition.
\end{proof}

\subsection{Structural homomorphisms}
In this section we prove parts \eqref{th:D_n.iv} and~\eqref{th:D_n.v} of Theorem~\ref{th:D_n}. We use the numeration 
of nodes in the Dynkin diagram of $\lie{so}_{2n+2}$ and $\lie{sp}_{2n}$ introduced in Section~\ref{subs:alg Uqn}.
Note first that part~\eqref{th:D_n.iv} of Theorem~\ref{th:D_n} is trivial since modulo 
the ideal generated by~$z$ its defining relations are precisely the defining relations of~$U_q(\lie{sp}_{2n})$ where
$w$ corresponds to $E_0$. 

To prove part~\eqref{th:D_n.v} of~Theorem~\ref{th:D_n}, let $W=E_0E_{-1}$ and
\begin{align*}
 Z&=(q^2-1)^{-1}([E_0,[E_1,E_{-1}]_q]_q+[E_{-1},[E_1,E_0]_q]_q)\\
 &=(1-q^{-2})^{-1}((\ad^* E_0)(\ad^* E_1)(E_{-1})+
(\ad^*E_{-1})(\ad^* E_1)(E_0))\\
&=(1-q^{-2})^{-1}((\ad^* E_0)(\ad E_{-1})(E_1)+(\ad^*E_{-1})(\ad E_0)(E_1))
\end{align*}
be the images of~$w$ and~$z$ in~$U_q(\lie{so}_{2n+2})$.
Clearly
\begin{equation}\label{eq:quasi-derW}
{}_i r(W)=r_i(W)=0,\quad i>0,\quad {}_0 r(W)=E_{-1},\quad {}_{-1}r(W)=E_0.
\end{equation}
Using Lemma~\ref{lem:quantum levi fact} we obtain
\begin{equation}\label{eq:quasi-derZ}
{}_ir(Z)=0,\quad i>0,\quad {}_0 r(Z)=q (\ad^* E_1)(E_{-1}),\quad {}_{-1} r(Z)=q(\ad^* E_1)(E_0).
\end{equation}
It is easy to check that $Z^*=Z$, hence $r_i(Z)=0$ for all $i>0$. Finally, we have 
\begin{equation}\label{eq:old defn Z}
\begin{split}
Z&= q [W,E_1]_{q^{-2}}-q\, \frac{[2]_q}{q-q^{-1}}(\ad E_0)(\ad E_{-1})(E_1)\\
&=q [E_1,W]_{q^{-2}}-q\, \frac{[2]_q}{q-q^{-1}}(\ad^*E_0)(\ad^*E_{-1})(E_1).
\end{split}
\end{equation}

\begin{proof}[Proof of~Theorem~\ref{th:D_n}\eqref{th:D_n.v}]
We need to show that the elements $W$ and~$Z$ satisfy the relations~\eqref{eq:D_n rls.2}--\eqref{eq:D_n rls.5}. 
The last two identities in~\eqref{eq:D_n rls.2} are trivial, while the first follows from~\eqref{eq:defn-quasi-der} and
\eqref{eq:quasi-derW}, \eqref{eq:quasi-derZ}. Furthermore, observe that 
$$
[E_i,W]=(\ad E_i)(W),\qquad i>1,\qquad [E_1,[E_1,[E_1,W]_{q^{-2}}]]_{q^{2}}=(\ad E_1)^3(W),
$$
while
$$
[E_i,Z]=(\ad E_i)(Z),\qquad i>0,\, i\not=2,\qquad [E_2,[E_2,Z]_{q^{-1}}]_{q}=(\ad E_2)^2 (Z).
$$
The first two identities in~\eqref{eq:D_n rls.1} are now immediate from~\eqref{eq:ad nilp.1}.
The first identity in~\eqref{eq:D_n rls.3} follows from \eqref{eq:ad nilp.2} 
since $(\ad E_1)^2(E_i)=0$, $i\in\{-1,0\}$ by quantum Serre's relations. The second 
is also a consequence of quantum Serre relations since $\ad E_2$ commutes with $\ad^* E_i$, $\ad E_j$, $i,j\in\{0,-1\}$.
To prove the last relation in~\eqref{eq:D_n rls.1}, note that since 
$$
(\ad^* E_i)^2(\ad^* E_1)(E_j)=0,\qquad (\ad E_i)(\ad^* E_j)(\ad E_i)(E_1)=0,\qquad \{i,j\}=\{0,1\}
$$
by quantum Serre relations, it follows that
\begin{align*}
&E_i (\ad^* E_i)(\ad^* E_1)(E_j)=q (\ad^* E_i)(\ad^* E_1)(E_j)E_i,\\& E_i(\ad^* E_j)(\ad^* E_1)(E_i)=q^{-1} (\ad^* E_j)(\ad^*E_1)(E_i)E_i,
\end{align*}
hence 
$$
W Z=Z W.
$$
To prove the first identity in~\eqref{eq:D_n rls.4}, notice that since $[Z,W]=0$ we obtain 
from~\eqref{eq:old defn Z}
\begin{multline*}
[W,[W,E_1]_{q^{-2}}]_{q^2}-(q^{-1}-q)WZ=[W,[W,E_1]_{q^{-2}}-q^{-1}Z]_{q^2}\\=
[2]_q(q-q^{-1})^{-1} [W,(\ad E_0)(\ad E_{-1})(E_1)]_{q^2}.
\end{multline*}
Since for all $x\in U_q^+(\lie{so}_{2n+2})$
$$
[W,x]_{q^2}=E_0 [E_{-1},x]_q+q [E_0,x]_q E_{-1}
$$
and $[E_i,x]_q=(\ad E_i)(x)$ for $x=(\ad E_0)(\ad E_{-1})(E_1)$ and $i\in\{0,-1\}$,
it follows from the quantum Serre relations that
$$
[W,(\ad E_0)(\ad E_{-1})(E_1)]_{q^2}=0,
$$
which together with~\eqref{eq:old defn Z} implies the first relation in~\eqref{eq:D_n rls.4}.
To prove the remaining identities, we use Lemma~\ref{lem:lusztig}\eqref{lem:lusztig.i}. Note that
$$
{}_i r([a,b]_v)=[{}_i r(a),b]_{vq^{-(\delta,\alpha_i)}}+q^{-(\gamma,\alpha_i)} [a,{}_i r(b)]_{v q^{(\gamma,\alpha_i)}},
$$
for all $a\in U_q^+(\lieg)_\gamma$, $b\in U_q^+(\lieg)_\delta$, $\gamma,\delta\in Q$ and
$v\in\CC(q)^\times$. 

To prove~\eqref{eq:D_n rls.4}, note that $$[E_2,[E_1,W]_{q^2}]_q=q^3 (\ad^*E_2)(\ad^*E_1)(W),$$
while 
$$
[E_1,[E_2,Z]_q]_q=q^2(\ad^*E_1)(\ad^*E_2)(Z)
$$
Thus, we want to show that ${}_i r(x)=0$ for all $i\in I$ where 
$$
x= q [Z,(\ad^*E_2)(\ad^*E_1)(W)]_q-[W,(\ad^*E_1)(\ad^*E_2)(Z)]_{q^2}.
$$
This is trivial if~$i>2$. For~$i=2$ we obtain 
\begin{multline*}
{}_2 r(x)=(q-q^{-1}) (q^2[Z,(\ad^* E_1)(W)]-[W,(\ad^*E_1)(Z)]_{q^2})\\=
(q-q^{-1})q^2 (Z (\ad^*E_1)(W)+(\ad^* E_1)(Z) W)-((\ad^*E_1)(W)Z+q^{-2} W(\ad^*E_1)(Z))\\
=(q-q^{-1})q^2 ( (\ad^*E_1)(ZW)-(\ad^* E_1)(WZ))=(q-q^{-1})(\ad^*E_1)[Z,W]=0,
\end{multline*}
where we used already established~\eqref{eq:D_n rls.1} and~\eqref{eq:ad nilp.2}.
Similarly
\begin{multline*}
{}_1 r(x)=(q-q^{-1})(q[Z,(\ad^*E_2)(W)]_q-q^2[W,(\ad^* E_2)(Z)])\\=
(q-q^{-1})q^2(\ad^*E_2)([W,Z])=0.
\end{multline*}
The computation of ${}_i r(x)$ for $i\in\{-1,0\}$ and the ones for the last identity, 
are rather tedious and where performed on a computer.
\end{proof}

\begin{remark}
It can be shown that the kernel of the homomorphism~$\hat\iota:\mathcal U^+_{q,n}\to U_q(\lie{so}_{2n+2})$ is generated by 
an element of degree~3 in $S_q(V\tensor V)$ which is 
a lowest weight vector of a simple $U_q(\lie{sl}_n)$-submodule of~$(V\tensor V)^{\tensor 3}$ isomorphic to~$V_{2\varpi_3}$.
On the other hand, the image of~$\hat\iota$ equals to the subalgebra of $\sigma$-invariant elements in~$U_q^+(\lie{so}_{2n+2})$ graded
by~$Q^\sigma$.
\end{remark}

\subsection{Braid group action on~\texorpdfstring{$\mathcal U_{q,n}$}{U\_q,n}}
\begin{proof}[Proof of part~\eqref{th:D_n.iii} of Theorem~\ref{th:D_n}]
Let $\tilde{\mathcal U}_{q,n}$ be the algebra generated by $U_q(\lie{sl}_n)$ and $w$, $z$
subjects to the relations~\eqref{eq:D_n rls.2}, \eqref{eq:D_n rls.1}, \eqref{eq:D_n rls.3}, except 
the commutativity relation $[w,z]=0$.
Clearly that $\mathcal U_{q,n}$ is a quotient of $\tilde{\mathcal U}_{q,n}$.
First we prove the following
\begin{proposition}
The formulae~\eqref{eq:D_n braid formulae} extend the action of~$Br_{\lie{sl}_n}$ on~$U_q(\lie{sl}_n)$ to
an action on~$\tilde{\mathcal U}_{q,n}$ by algebra automorphisms.
\end{proposition}
\begin{proof}
We note the following useful Lemma
\begin{lemma}\label{lem:braid-tmp}
In $\tilde{\mathcal U}_{q,n}$ we have
\begin{equation*}
\begin{aligned}
&[F_i,T_1(w)]=-\delta_{i,1}[w,E_1]_{q^{-2}}K_1,&& [F_i,T_2(z)]=-\delta_{i,2}q^{-1}z K_2\\
&[T_1(E_1),T_1(w)]_{q^{-2}}=[w,E_1]_{q^{-2}},&& [T_2(E_2),T_2(z)]_{q^{-1}}=z\\
&[T_2(E_1),w]_{q^{-2}}=[[E_1,w]_{q^{-2}},E_2]_{q^{-1}},&& [T_1(E_2),z]_{q^{-1}}=[[E_2,z]_{q^{-1}},E_1]_{q^{-1}}.
\end{aligned}
\end{equation*}
\end{lemma}
Clearly, $[T_i(F_j),T_i(w)]=0$ (respectively, $[T_i(F_j),T_i(z)]=0$) for all $j$ and
for all~$i\not=1$ and (respectively, for all~$i\not=2$). 
Since
\begin{align*}
&[T_1(F_1),T_1(w)]=-[2]_q^{-1} [E_1,[E_1,[E_1,W]_{q^{-2}}]]_{q^2}K_1=0\\
&[T_2(F_2),T_2(z)]=[E_2,[E_2,z]_{q^{-1}}]_qK_2=0.
\end{align*}
we conclude that $[T_1(F_i),T_1(w)]=0$ unless~$i=2$, while by Lemma~\ref{lem:braid-tmp}
\begin{align*}
&[T_1(F_2),T_1(w)]=[[F_1,F_2]_q,T_1(w)]=[[F_1,T_1(w)],F_2]_q=-[w,E_1]_{q^{-2}}[K_1,F_2]_q=0\\
&[T_2(F_j),T_2(z)]=[[F_2,F_j]_q,T_2(z)]=[[F_2,T_2(z)],F_j]_q=-q^{-1} z[K_2,F_j]_q=0,\,\, j=1,3
\end{align*}
The remaining identity in~\eqref{eq:D_n rls.2} is clearly preseved. 
Similarly, for all
$i$ and for all $j\not=1$ (respectively $j\not=2$) we obtain $[T_i(E_j),T_i(w)]=0$ (respectively $[T_i(E_j),T_i(z)]=0$).
The remaining identities follow from Lemma~\ref{lem:braid-tmp} and direct computations. For example, for $i=1,3$ 
\begin{align*}
[T_2(E_i),T_2(z)]&=q^{-2} E_2 E_i E_2 z- q^{-2} E_2 z E_2 E_i-q^{-1} E_i E_2
    E_2 z\\&+E_i E_2 z E_2+q^{-1} z E_2 E_2 E_i-z E_2 E_i
    E_2\\&=[z,(\ad E_2)^{(2)}(E_i)]_{q^{-2}}=0,
\end{align*}
where we used~\eqref{eq:D_n rls.3} and quantum Serre relations.
It is not hard to check, using the above Lemma,
that the maps $T_i$ are invertible with their inverses given on $w$ and~$z$ by
$$
T_1^{-1}(w)=[2]_q^{-1}[E_1,[E_1,w]_{q^{-2}}],\qquad T_2^{-1}(z)=[E_2,z]_{q^{-1}}
$$
while $T_i^{-1}(w)=w$ if~$i\not=1$ and $T_i^{-1}(z)=z$ if~$i\not=2$.
Thus, we conclude that the $T_i$ are automorphisms 
of~$\tilde{\mathcal U}_{q,n}$. Finally, the only braid relations that need to be checked are $T_1 T_2 T_1(w)=T_2 T_1 T_2(w)=T_2 T_1(w)$ and 
$T_2T_iT_2(z)=T_i T_2 T_i(z)=T_i T_2(z)$, where $i=1,3$. This is done by a direct computation. 
\end{proof} 
To complete the proof of part~\eqref{th:D_n.iii} of Theorem~\ref{th:D_n} is suffices to show that 
the kernel of the canonical map $\tilde{\mathcal U}_{q,n}\to \mathcal U_{q,n}$ is preserved 
by the~$T_i$, $1\le i\le n-1$. For example, consider~\eqref{eq:D_n rls.4}. Note that $[w,[w,E_1]_{q^{-2}}]_{q^2}=[[E_1,w]_{q^{-2}},w]_{q^2}$.
Using Lemma~\ref{lem:braid-tmp} we obtain
$$
[[T_i(E_1),T_i(w)]_{q^{-2}},T_i(w)]_{q^2}
=\begin{cases}[[E_1,w]_{q^{-2}},w]_{q^2},& i>2\\
                                           [[w,[w,E_{1}]_{q^{-2}}]_{q^2}],E_2]_{q^{-1}},&i=2\\
                                           [2]_q^{-1} [[[w,[w,E_{1}]_{q^{-2}}]_{q^2},E_1]_{q^{-2}},E_1],&i=1,
                                          \end{cases}
$$
while $T_i(w)T_i(z)=w z$, $i>2$,
$$
T_{1}(w)T_{1}(z)=[2]_q^{-1}[[wz,E_1]_{q^{-2}},E_1]
$$
and
$$
T_{2}(w)T_{2}(z)=[wz,E_2]_{q^{-1}}.
$$
Thus, the first relation in~\eqref{eq:D_n rls.4} is preserved. The computations for the remaining relations 
are rather tedious and where performed on a computer. The relations can be checked in many different ways; perhaps,
the simplest is to use the isomorphism $\mathcal U_{q,n}\cong S_q(V\tensor V)\rtimes U_q(\lie{sl}_n)$, which allows 
us to write any element of~$\mathcal U_{q,n}$ as $\sum_i Y_i m_i$, where $Y_i\in S_q(V\tensor V)$ and
the $m_i$ are linearly independent elements of~$U_q(\lie{sl}_n)$. Writing a relation in this form, we then
check that $(\Psi+q^2)(\Psi+q^{-2})(Y_i)=0$ hence~$Y_i\in\Im(\Psi-1)$.
\end{proof}

\subsection{Liftable quantum foldings 
and \texorpdfstring{$\mathcal U^+_{q,n}$}{U\_q,n+} as a uberalgebra}\label{sub:explicit q.f}
In this section we use the standard numbering of the nodes 
of all Dynkin diagrams.
\begin{theorem}\label{th:images coincide}
In the notation of Theorem~\ref{th:D_n}, $\iota_{\bi}$ for any $\bi\in R(w_\circ)$
is a tame liftable folding with $U(\iota_\bi)=\mathcal U_{q,n}^+$ and $\mu_{\iota_\bi}=\mu$. In particular,
$\tilde\iota_\bi$ splits~$\mu$ and
we have a commutative diagram
\begin{equation}\label{eq:images.commute}
\begin{diagram}
 \node{\mathcal U_{q,n}^+}\arrow{e,t}{\hat\iota} \node{U_q^+(\lie{so}_{2n+2})}\\
 \node{U_q^+(\lie{sp}_{2n})}\arrow{n,l}{\tilde\iota_\bi}\arrow{ne,r}{\iota_\bi}
\end{diagram}
\end{equation}
where all maps commute with the right multiplication with
$U_q^+(\lie{sl}_n)$.
\end{theorem}
\begin{proof}
Let $w_\circ'$ be the longest element in~$W(\lie{sl}_n)=W((\lie{sp}_{2n})_J)$,
where $J=\{1,\dots,n-1\}$ and let $\bi'=(n-1,n-2,n-1,\dots,1,\dots,n-1)\in R(w_\circ')$.
Set
$$
\bi_r=(n,n-1,n,n-2,n-1,n,\dots,r,\dots,n),\qquad 1\le r\le n.
$$
First, we prove the Theorem for $\bi=\bi_\circ$ where $\bi_\circ$ is the concatenation $\bi_1\bi'$.

Given $\bj=(r_1,\dots,r_k)\in (I/\sigma)^k$, write $w_\bj=s_{r_1}\cdots s_{r_k}$
and $T_\bj=T_{w_\bj}$ (respectively, $\hat T_\bj=T_{\hat w_{\bj}}$).
It is easy to check that 
\begin{equation}\label{eq:w simpl}
w_{\bi_{r}}(\alpha_s)=\alpha_{r+n-s-1},\qquad r< s\le n-1,\, 1\le r\le n.
\end{equation}
In particular, $T_{\bi_1}(E_i)=E_{n-i}$, $1\le i\le n-1$ by Lemma~\ref{lem:lus.br.id}, hence $T_{\bi_1}$ acts
as the diagram automorphism~$\tau$ of~$U_q^+(\lie{sl}_n)=U_q^+((\lie{sp}_{2n})_J)$. 
Define the elements $y_{ij}^+\in U_q^+(\lie{sp}_{2n})$, $x_{ij}^+\in U_q(\lie{so}_{2n+2})$, $1\le i\le j\le n$ and $x_{ij}^-\in U_q^+(\lie{sl}_n)=
U_q^+((\lie{so}_{2n+2})_J)=U_q^+((\lie{sp}_{2n})_J)$
\begin{equation}\label{eq:xij yij} 
\begin{aligned}
&X_{\bi}=\{y_{nn}^+,y_{n-1,n}^+,y_{n-1,n-1}^+,\dots,y_{1n}^+,\dots,y_{11}^+,x_{12}^-,\dots,x_{1,n}^-,\dots x_{n-1,n}^-\}
\\
&\hat X_{\bi}=\{x_{nn}^+,x_{n-1,n}^+,x_{n-1,n-1}^+,\dots,x_{1n}^+,\dots,x_{11}^+,x_{12}^-,\dots,x_{1,n}^-,\dots x_{n-1,n}^-\}
\end{aligned}
\end{equation}
as ordered sets, in the notation of Section~\ref{subs:PBW bases}.
In particular for all $a_{ij}^+,a_{ij}^-\in \ZZ_{\ge 0}$, we have
$$
\iota_\bi\Big(\prod^{\to}_{1\le i\le j\le n} (y_{ij}^+)^{a_{ij}^+}\prod^\to_{1\le i<j\le n} (x_{ij}^-)^{a_{ij}^-}\Big)=
\prod^\to_{1\le i\le j\le n} (x_{ij}^+)^{a_{ij}^+}
\prod^\to_{1\le i<j\le n} (x_{ij}^-)^{a_{ij}^-}
$$
where both products are taken in the same order as in~\eqref{eq:xij yij}.

Identify $\mathcal U_{q,n}^+$ with $S_q(V\tensor V)\rtimes U_q^+(\lie{sl}_n)$ using the isomorphism~$\psi_+$
from Proposition~\ref{pr:crossprod isom}. Define $\tilde\iota_{\bi}:U_q^+(\lie{sp}_{2n})\to\mathcal U_{q,n}^+$ 
by 
\begin{equation}\label{eq:tildiota}
\tilde\iota_\bi\Big(\prod^{\to}_{1\le i\le j\le n} (y_{ij}^+)^{a_{ij}^+}\prod^\to_{1\le i<j\le n} (x_{ij}^-)^{a_{ij}^-}\Big)=
\prod^\to_{1\le i\le j\le n} \tilde X_{ji}^{a_{ij}^+}
\prod^\to_{1\le i<j\le n} (x_{ij}^-)^{a_{ij}^-}
\end{equation}
where $\tilde X_{ji}=(q-q^{-1})^{i+j-2n} X_{ji}$ and the product in the right hand side is taken in the total
order defined in the proof of Proposition~\ref{pr:SqPres}.

We will need the following result, where we abbreviate $\tilde E_i:=\ad E_i$, $\tilde E_i^*:=\ad^* E_i$.
\begin{proposition}\label{pr:PBW-def}
The elements $x_{ij}^+\in U_q^+(\lie{so}_{2n+2})$ and $y_{ij}^+\in U_q^+(\lie{sp}_{2n})$ defined in~\eqref{eq:xij yij} are given 
by the following formulae
\begin{alignat}{2}\label{eq:PBW-Dn-def}
&x_{ij}^+=(\hat c_{ij}^+)^{-1} \tilde E_j\cdots \tilde E_{n-1}\tilde E_i\cdots \tilde E_{n-2} 
\tilde E^*_{n}\tilde E^*_{n+1}(E_{n-1}),&\quad& 1\le i<j\le n
\\
&x_{ii}^+=(\hat c_{ii}^+)^{-1}\tilde E_i^{(2)}\cdots \tilde E_{n-1}^{(2)}(E_n E_{n+1}),&&1\le i\le n
\label{eq:PBW-Dn-def.1}\\
&y_{ij}^+=(c_{ij}^+)^{-1}\tilde E_i\cdots \tilde E_{j-1}\tilde E_j^{(2)}\cdots \tilde E_{n-1}^{(2)}(E_n),&& 1\le i\le j\le n
\label{eq:PBWeltsCn-defn}
\end{alignat}
where $c_{ij}^+=(q+q^{-1})^{1-\delta_{ij}} (q-q^{-1})^{2n-i-j}$, $\hat c_{ij}^+=(q-q^{-1})^{2n-i-j+1-\delta_{ij}}$.
\end{proposition}
\begin{proof}
We only prove~\eqref{eq:PBW-Dn-def} and~\eqref{eq:PBW-Dn-def.1}. The argument for~\eqref{eq:PBWeltsCn-defn} is nearly identical and is omitted.
By the definition of the set $\hat X_\bi$ given in Section~\ref{subs:PBW bases}  we have for all $1\le i\le j\le n$
$$
x_{ij}^+= \gamma(\beta_{ij})^{-1}(q-q^{-1})^{1+\delta_{ij}} \hat T_{\bi_{i+1}}T_i\cdots T_{i+n-j-1}(E_{i+n-j}E_{n+1}^{\delta_{ij}}),
$$
where $\beta_{ij}=\deg x_{ij}^+=\hat w_{\bi_{i+1}}s_i\cdots s_{i+n-j-1}(\alpha_{i+n-j}+\delta_{ij}\alpha_{n+1})$. It is easy to see,
using~\eqref{eq:w simpl},  
that~$\beta_{ij}=\alpha_i+\cdots+\alpha_{j-1}+
2(\alpha_j+\cdots+\alpha_{n-1})+\alpha_n+\alpha_{n+1}$, hence $\gamma(\beta_{ij})(q-q^{-1})^{-1-\delta_{ij}}=
\hat c_{ij}^+$.

Since 
\begin{equation}\label{eq:tmp-a}
\hat w_{\bi_i}=s_{n+1}s_n\cdots s_i \hat w_{\bi_{i+1}}, \qquad 1\le i\le n-1 
\end{equation}
and $\ell(w_{\bi_i})=\ell(w_{\bi_{i+1}})+n-i+2$, we have for all $1\le i\le n$
$$
\hat T_{\bi_{i+1}}(E_i)=T_{n+1}\cdots T_{i+1}T_{\bi_{i+2}}(E_i)=T_{n+1}\cdots T_{i+1}(E_i).
$$
Then it is easy to see, using \eqref{eq:braid.ad}, \eqref{eq:ad*-ad} and the obvious observation that 
$\ad E_r$, $\ad^* E_s$ commute if $a_{rs}=0$, that
\begin{equation}\label{eq:tmp-0}
\hat c_{in}^+ x_{in}^+=
\tilde E^*_{n+1}\cdots\tilde E^*_{i+1}(E_i)=\tilde E_i\cdots\tilde E_{n-2}\tilde E_n^*\tilde E_{n+1}^*(E_{n-1}).
\end{equation}
To establish~\eqref{eq:PBW-Dn-def} for $1\le i<j<n$, note that by~\eqref{eq:w simpl} and Lemma~\ref{lem:lus.br.id}
we have $\hat T_{\bi_{i+1}}(E_{i+n-j})=E_{j}$, $i<j<n$. Therefore, 
\begin{multline*}
\hat c_{ij}^+ x_{ij}^+=\hat T_{\bi_{i+1}}T_i\cdots T_{i+n-j-1}(E_{i+n-j})=
\hat T_{\bi_{i+1}}\tilde E_i^*\cdots \tilde E^*_{i+n-j-1}(E_{i+n-j})\\=
\hat T_{\bi_{i+1}}\tilde E_{i+n-j}\cdots \tilde E_{i+1}(E_i)
=\tilde E_j\cdots \tilde E_{n-1}(\hat T_{\bi_{i+1}}(E_i))=\hat c_{in}^+
\tilde E_j\cdots \tilde E_{n-1}x_{in}^+,
\end{multline*}
where we used \citem{BZ}*{Lemma~3.5}, \eqref{eq:tmp-0} and~\eqref{eq:Tw ad}.

Finally, we prove by a downward induction on~$i$ that
\begin{equation}\label{eq:tmp-1}
\hat T_{\bi_{i+1}}T_i\cdots T_{n-1}(E_n)=T_{\sigma^{n-i}(n)}T_{n-1}\cdots T_{i+1}(E_i).
\end{equation}
If $i=n-1$, it follows from Lemma~\ref{lem:lus.br.id} that  $T_n T_{n+1}T_{n-1}(E_n)=T_{n+1}(E_{n-1})$ 
so the induction begins.
For the inductive step, suppose that $n-i$ is odd, the case of $n-i$ even being similar. Using~\eqref{eq:tmp-a}, 
we can write  $$
\hat w_{\bi_{i+1}}s_i\cdots s_{n-1}=s_{n+1}s_n\cdots s_i\hat w_{\bi_{i+2}}s_{i+1}\cdots s_{n-1}$$
and since both expressions are reduced the induction hypothesis yields
$$
\hat T_{\bi_{i+1}}T_i\cdots T_{n-1}(E_n)=T_{n+1}\cdots T_i T_{n}T_{n-1}\cdots T_{i+2}(E_{i+1}).
$$
The inductive step now follows from the braid relations and Lemma~\ref{lem:lus.br.id}. 
Using~\citem{BZ}*{Lemma~3.5} we obtain from~\eqref{eq:tmp-1} that 
\begin{multline*}
\hat c_{ii}^+ x_{ii}^+=(\tilde E_{n+1}^*\tilde E_{n-1}^*\cdots\tilde E_{i+1}^*(E_i))(\tilde E_{n}^*\tilde E_{n-1}^*\cdots\tilde E_{i+1}^*(E_i))
\\=(\tilde E_i\cdots \tilde E_{n-1}(E_n))(\tilde E_i\cdots \tilde E_{n-1}(E_{n+1})).
\end{multline*}
Since by the quantum Serre relations, $\tilde E_i^{(2)}\tilde E_{i+1}\cdots\tilde E_{n-1}(E_r)=0$, $r\in\{n,n+1\}$,
\eqref{eq:PBW-Dn-def.1} follows immediately from~\eqref{eq:ad nilp.3}.
\end{proof}

We can now complete the proof of the Theorem. First, observe that~\eqref{eq:sln Vtens V.1} implies that 
$X_{ii}=E_i^{(2)}\cdots E_{n-1}^{(2)}(X_{nn})$, $1\le i\le n-1$ while
$$
X_{ji}=E_j\cdots E_{n-1} E_{i}\cdots E_{n-2}(X_{n,n-1}),\quad 1\le i<j\le n.
$$
Since $X_{n,n-1}=[2]_q^{-1} (E_{n-1}(w)-q^{-1}z)$, we obtain 
\begin{equation*}
\begin{aligned}
&\tilde\iota_\bi(y_{ij}^+)=\frac{E_j\cdots E_{n-1}E_i\cdots E_{n-2}}{(q-q^{-1})^{2n-i-j-1}}\Big(\frac{E_{n-1}(w)-q^{-1}z}{q^2-q^{-2}}\Big),
&\quad & 1\le i<j\le n\\
&\tilde\iota_\bi(y_{ii}^+)= (c_{ii}^+)^{-1} E_i^{(2)}\cdots E_{n-1}^{(2)}(w),& & 1\le i\le n.
\end{aligned}
\end{equation*}
Note that for all $u\in S_q(V\tensor V)$, $1\le i\le n-1$ we have $\mu(E_i(u))=\tilde E_i(\mu(u))$ and
similarly $\hat\iota(E_i(u))=\tilde E_i(u)$. 
This, together with~\eqref{eq:old defn Z}, Proposition~\ref{pr:PBW-def}
and the multiplicativity of~$\tilde\iota_\bi$ and $\iota_\bi$ immediately implies the assertion for $\bi=\bi_\circ$.

To complete the proof, it remains to apply Lemma~\ref{lem:change of basis} and the argument from the proof of Theorem~\ref{th:folding ii}. 
\end{proof}
This completes the proof of Theorems~\ref{th:first folding} and~\ref{th:Dn Chevalley}. \qed

\section{Diagonal foldings}

\subsection{Folding~\texorpdfstring{$(\lie{sl}_3^{\times n},\lie{sl}_3)$}{(sl\_3\ x\ n,sl\_3)}}\label{subs:diag sl3}
Consider the algebra $\lie s_n:=\lie{sl}_3^{\times n}$ with the diagram automorphism $\sigma$
which is a cyclic permutation of the components. 

Let $\mathcal A_{q,3}^{(n)}$ be the associative $\CC(q)$-algebra generated by
$u_1,u_2$ and $z_k$, $1\le k\le n-1$ subjects to relations given in Theorem~\ref{th:second folding A_2 cube}\eqref{th:second folding A_2 cube.i}.
\begin{theorem}\label{th:second folding A_2 cube refine}
\begin{enumerate}[{\rm(i)}]
 \item\label{th:second folding A_2 cube refine.i} The algebra $\mathcal A_{q,3}^{(n)}$ is PBW on the totally ordered set 
 $$\{u_2,u_{21},u_1\}\cup\{z_k\,:\, 1\le k\le n-1\}, 
 $$
 where
 $$
 u_{21}=u_1 u_2-q^{-n} u_2 u_1-\sum_{k=1}^{n-1} \frac{q-q^{-1}}{q^k-q^{-k}}\, z_k.
 $$
\item\label{th:second folding A_2 cube refine.ii} The assignment $u_i\mapsto E_i$, $z_k\mapsto 0$ defines 
a surjective algebra homomorphism $\mu:\mathcal A_{q,3}^{(n)}\to U_q(\lie{s}_n^\sigma{}^\vee)$.
\end{enumerate}

\end{theorem}

\begin{proof}
Since $u_{21}=[u_1,u_2]_{q^{-n}}-\sum_{r=1}^{n-1} [r]_q^{-1} z_r$, we have 
$$
[u_1, u_{21}]_{q^{n}}=[u_1,[u_1,u_2]_{q^{-n}}]_{q^n}-\sum_{r=1}^{n-1}  \frac{[u_1,z_r]_{q^n}}{[r]_{q}}=
\sum_{r=1}^{n-1} \Big(q^r(q^{-1}-q)+ \frac{q^{2r}-1}{[r]_q}\Big)  u_1 z_r=0.
$$
Similarly, we can write 
$$
u_{21}=-q^{-n}[u_2,u_1]_{q^n}-\sum_{r=1}^{n-1} \tilde z_{n-r}=
-q^{-n} [u_2,u_1]_{q^n}-\sum_{r=1}^{n-1} [n-r]_q^{-1} z_{n-r}
$$
hence 
\begin{multline*}
[u_2,u_{21}]_{q^{-n}}=-q^{-n} [u_2,[u_2,u_1]_{q^n}]_{q^{-n}}-\sum_{r=1}^{n-1} [n-r]_q^{-1} [u_2,z_{n-r}]_{q^{-n}}\\
=-q^{-n}(q^{-1}-q)\sum_{r=1}^{n-1} q^r u_2 z_{n-r}-\sum_{r=1}^{n-1} \frac{1-q^{2(r-n)}}{[n-r]_q}\, u_2 z_{n-r}\\
=(q-q^{-1}) \sum_{r=1}^{n-1} \Big(q^{r-n}+\frac{q^{2(r-n)}-1}{q^{n-r}-q^{r-n}}\Big) u_2 z_{n-r}=0.
\end{multline*}
Since clearly $u_{21}$ commutes with the $z_r$, we obtain the PBW relations from Theorem~\ref{th:second folding A_2 cube}\eqref{th:second folding A_2 cube.ii}
The above computations also show that PBW relations imply Serre-like relations. 
To prove that $\mathcal A_{q,3}^{(n)}$ is PBW, we use Diamond Lemma (Proposition~\ref{prop:diamond lemma}). It is easy to see that the only 
situation which needs to be checked is the monomial $u_1 u_{21} u_2$. We have
$$
(u_1 u_{21})u_2=q^n u_{21}u_1u_2=q^n (u_2 u_{21} u_1+u_{21}^2+\sum_{k=1}^{n-1} [k]_q^{-1} u_{21} z_k)=u_1(u_{21}u_2).
$$
The second assertion is obvious.
\end{proof}

We now proceed to prove that $\mathcal A_{q,3}^{(n)}$ is the desired enhanced uberalgebra for this folding.

Given $x_\alpha\in U_q(\lie{sl}_3)$ we denote its copy in the $i$th component of~$U_q(\lie{sl}_3)^{\tensor n}$ by $x_{\alpha,i}$.
Let $\bi_1=(121)$ and $\bi_2=(212)$. Define the elements $y_i\in U_q(\lie s_n^\sigma)$ and 
$Y_{i,n}\in U_q(\lie{s}_n)$, $i\in\{1,2,12,21\}$ by 
$$
X_{\bi_1}=\{y_1,y_{21},y_2\},\qquad \hat X_{\bi_1}=\{ Y_{1,n},Y_{21,n},Y_{2,n}\}
$$
as ordered sets, in the notation of Section~\ref{subs:PBW bases}, and similarly for~$\bi_2$. It is immediate that 
$y_{1}=E_1$, $y_2=E_2$ and 
$$
y_{12}=(q^n-q^{-n})^{-1} (E_1E_2-q^{-n}E_2 E_1),\quad y_{21}=(q^n-q^{-n})^{-1} (E_2E_1-q^{-n}E_1 E_2)
$$
while $Y_{1,n}=\prod_{i=1}^n E_{1,i}$, $Y_{2,n}=\prod_{i=1}^n E_{2,i}$ and 
$$
Y_{21,n}=\prod_{i=1}^n E_{21,i},\qquad Y_{12,n}=\prod_{i=1}^n E_{12,i},
$$
where 
$$
E_{12,i}=\frac{E_{2,i} E_{1,i}-q^{-1} E_{1,i}E_{2,i}}{q-q^{-1}},\qquad E_{21,i}=\frac{E_{1,i} E_{2,i}-q^{-1} E_{2,i}E_{1,i}}{q-q^{-1}}.
$$
In particular, $\iota_{\bi_1}$ (respectively, $\iota_{\bi_2}$) is given by
$$
\iota_{\bi_1}(y_1^a y_{21}^b y_2^c)=Y_{1,n}^a Y_{21,n}^b Y_{2,n}^c,\qquad 
\iota_{\bi_2}(y_2^a y_{12}^b y_1^c)=Y_{2,n}^a Y_{12,n}^b Y_{1,n}^c,\qquad 
$$

We will need some identities for the elements~$E_{\alpha,i}$.
Clearly 
\begin{equation}\label{eq:A2n rel.1}
q E_{21,i}+E_{12,i}=E_{1,i}E_{2,i},\qquad 
q E_{12,i}+E_{21,i}=E_{2,i}E_{1,i}
\end{equation}
It follows from quantum Serre relations that 
\begin{equation}\label{eq:A2n rel.2}
E_{i,r}E_{ij,s}=q^{-\delta_{r,s}} E_{ij,s}E_{i,r},\qquad E_{j,r}E_{ij,s}=q^{\delta_{r,s}}E_{ij,s}E_{j,r},\qquad i\not=j\in\{1,2\}.
\end{equation}
In particular, this implies that
$$
E_{12,i}E_{21,i}=E_{21,i}E_{12,i}.
$$

Let $Z_{0,1}=E_{21,1}$, $Z_{1,1}=E_{12,1}$ and define inductively
\begin{equation}\label{eq:A2n Zdef}
Z_{i,k}=Z_{i,k-1}E_{21,k}+Z_{i-1,k-1}E_{12,k},\qquad 0\le i\le k,
\end{equation}
where we use the convention that $Z_{i,k}=0$ if~$i<0$ or~$i>k$. In particular, $Z_{0,n}=Y_{21,n}$
and $Z_{n,n}=Y_{12,n}$.

\begin{lemma}\label{lem:comm rels Z}
We have for all $0\le k,l\le n$
\begin{align}\label{eq:comm rels Z.1}
&Y_{1,n} Z_{k,n}=q^{n-2k} Z_{k,n}Y_{1,n},\quad Y_{2,n}Z_{k,n}=q^{-n+2k} Z_{k,n}Y_{2,n},\quad Z_{k,n}Z_{l,n}=Z_{l,n}Z_{k,n}\\
 &Y_{1,n}Y_{2,n}=\sum_{k=0}^n q^{n-k} Z_{k,n},\qquad Y_{2,n}Y_{1,n}=\sum_{k=0}^n q^k Z_{k,n}.\label{eq:comm rels Z.3'}
\end{align}
\end{lemma}
\begin{proof}
The first relation is just~\eqref{eq:A2n rel.2} for~$n=1$. Then, using induction on~$n$, we obtain
\begin{align*}
Y_{1,n}Z_{k,n}&=Y_{1,n-1}E_{1,n}(Z_{k,n-1}E_{21,n}+Z_{k-1,n-1}E_{12,n})\\
&=q^{n-1-2k} Z_{k,n-1} E_{1,n} E_{21,n} Y_{1,n-1}+q^{n+1-2k}Z_{k-1,n-1}E_{1,n}E_{12,n}Y_{1,n-1}\\
&=q^{n-2k}Z_{k,n}Y_{k,n}.
\end{align*}
The second identity in~\eqref{eq:comm rels Z.1} is proved similarly while the last is obvious
since $E_{21,r}$, $E_{12,s}$ commute for all $1\le r,s\le n$. To prove~\eqref{eq:comm rels Z.3'}, we
again use the inductive definition of the~$Z_{k,n}$. For~$n=1$ this relation coincides with~\eqref{eq:A2n rel.1}, while
\begin{align*}
Y_{1,n}Y_{2,n}&=Y_{1,n-1}Y_{2,n-1}E_{1,n}E_{2,n}=\sum_{k=0}^{n-1}q^{n-k-1} Z_{k,n-1}(E_{12,n}+q E_{21,n})\\
&=\sum_{k=0}^n q^{n-k-1} Z_{k,n-1}E_{12,n}+\sum_{k=0}^n q^{n-k} Z_{k,n-1}E_{21,n}=\sum_{k=0}^n q^{n-k} Z_{k,n}.
\end{align*}
The remaining identity is proved similarly.
\end{proof}
As an immediate corollary, we obtain 
\begin{equation}\label{eq:comm rels Z.3}
Y_{1,n}Y_{2,n}=q^{-n}Y_{2,n}Y_{1,n}+(q^n-q^{-n})Y_{21,n}+\sum_{k=1}^{n-1} (q^{n-k}-q^{k-n})Z_{k,n}
\end{equation}
\begin{example}\label{subs:inf generated A(z)}
Let $n=3$ and take~$\iota=\iota_{\bi_1}$. Then in $\lr{U_q^+(\lie{sl}_3)}_{\iota}$ we have 
$$
Y_{1,3}Y_{2,3}=q^{-3}Y_{2,3}Y_{1,3}+(q^3-q^{-3})Y_{21,3}+Y_{21,3}'.
$$
Since the terms in~$Y_{21,3}'$ quasi-commute with $Y_{1,3}$ with different powers of~$q$ and are linearly independent,
we obtain an infinite family of generators by taking $q$-commutators of~$Y_{1,3}$ with $Y_{21,3}'$. Thus, this algebra 
cannot be sub-PBW, since it clearly has polynomial growths.
\end{example}
\begin{lemma}\label{lem:Z0 fract}
The elements $Z_{k,n}$, $1\le k\le n-1$, are contained in $\lr{U_q^+(\lie{sl}_3)}_{\iota_{\bi}}\cap
\Frac\, U_q^+(\lie{sl}_3)^{\tensor n}$ for both reduced expressions~$\bi$ of the longest element 
in the Weyl group of~$\lie{sl}_3$.
\end{lemma}
\begin{proof}
Using Lemma~\ref{lem:comm rels Z}, we obtain for all $s>0$
$$
Y_{1,n}^s Y_{2,n}=\Big(\sum_{k=0}^n q^{(s-1)(n-2k)+n-k} Z_{k,n}\Big)Y_{1,n}^{s-1}.
$$
Note that one of the $Z_{0,n}$, $Z_{n,n}$ is contained in $\lr{U_q^+(\lie{sl}_3)}_{\iota_{\bi}}$.
Taking $1\le s\le n$ yields a system of linear equations for the $Z_{k,n}$ in~$\lr{U_q^+(\lie{sl}_3)}_{\iota_{\bi}}\cap
\Frac\, U_q^+(\lie{sl}_3)^{\tensor n}$ with the matrix $(q^{ns+k(1-2s)})$ where $1\le s\le n+1$, $0\le k\le n$. This matrix is easily
seen to be non-degenerate.
\end{proof}
\begin{proposition}\label{prop:A(z) uber}
The assignment $$u_i\mapsto Y_{i,n}, \quad i=1,2,\qquad z_k\mapsto [k]_q (q^{n-k}-q^{k-n})Z_{k,n}
$$
defines an algebra homomorphism $\hat\iota:\mathcal A_{q,3}^{(n)}\to U_q(\lie{sl}_n)^{\tensor n}$. In particular,
the subalgebra of $U_q^+(\lie{sl}_3)^{\tensor n}$ generated by $\iota_{\bi}(U_q^+((\lie{sl}_3^{\times n})^\sigma{}^\vee)$
and $Z_0=\{Z_{k,n}\,:\, 1\le k\le n-1\}$ is independent of~$\bi$ and is sub-PBW.
\end{proposition}
\begin{proof}
It is sufficient to prove that $Y_{1,n}$, $Y_{2,n}$ and~$Z_{k,n}$, $1\le k\le n-1$ satisfy the relations given in
Theorem~\ref{th:second folding A_2 cube}\eqref{th:second folding A_2 cube.i}.
The first two relations are already obtained in Lemma~\ref{lem:comm rels Z}. Furthermore~\eqref{eq:comm rels Z.1} implies
\begin{align*}
[Y_{1,n},[Y_{1,n},Y_{2,n}]_{q^{-n}}]_{q^n}&=\sum_{k=0}^{n-1} h_{n-k}[Y_{1,n},Z_{k,n}]_{q^n}
=-\sum_{k=1}^{n-1} q^k h_{n-k}h_{k}Y_{1,n}Z_{k,n}.
\\ \intertext{where $h_k=q^k-q^{-k}$. Similarly,}
[Y_{2,n},[Y_{2,n},Y_{1,n}]_{q^{-n}}]_{q^{n}}&=\sum_{k=0}^{n-1} h_{k}[Y_{2,n},Z_{k,n}]_{q^n}=-\sum_{k=0}^{n-1} q^{n-k}h_{n-k}h_{k}Y_{2,n}Z_{k,n}
\end{align*}
The remaining assertions are trivial.
\end{proof}

Define $\tilde\iota_{\bi_1}:U_q(\lie{s}_n^\sigma)\to \mathcal A_{q,3}^{(n)}$ by extending multiplicatively the assignments
$$
y_1\mapsto u_1,\quad y_2\mapsto u_2,\quad y_{21}\mapsto (q^n-q^{-n})^{-1}[u_1,u_2]_{q^{-n}}-\sum_{k=1}^{n-1} [k]_q^{-1} z_k.
$$
The map $\tilde\iota_{\bi_2}$ is defined similarly.
\begin{proposition}\label{pr:diag comm sl3 n}
The maps $\tilde\iota_{\bi_r}$, $r=1,2$ split~$\mu$ and for each of them the diagram~\eqref{eq:cone over iota} commutes with
$\iota=\iota_\bi$. In particular, $\mathcal A_{q,3}^{(n)}$ is the unique uberalgebra for this quantum folding.
\end{proposition}
\begin{proof}
We only show this for $\bi=\bi_1$, the argument for $\bi_2$ being similar. Since $\iota_\bi$ is 
multiplicative on the modified PBW basis, it is enough to check that the diagram~\eqref{eq:cone over iota} commutes 
on $y_{1}$, $y_{2}$ and~$y_{21}$, 
which is 
straightforward from Lemma~\ref{lem:comm rels Z} and from the definition of~$\mu$, $\hat\iota$ and~$\tilde\iota_\bi$.
\end{proof}

\subsection{Folding \texorpdfstring{$(\lie{sl}_4\times \lie{sl}_4,\lie{sl}_4)$}{(sl\_4\ x\ sl\_4,sl\_4) }} Now we turn our attention to the folding $(\lie{sl}_4^{\times 2},\lie{sl}_4)$
with $\sigma$ being the permutation of the components.
Let $\mathcal A_{q,4}=\mathcal A_{q,4}^{(2)}$ be the algebra with generators $u_i$, $1\le i\le 3$, $z_{12}=z_{21}$, $z_{13}=z_{31}$ and
$z_{23}=z_{32}$ and relations given in~Theorem~\ref{th:first folding 2}\eqref{th:first folding 2.ii}. 
\begin{theorem}\label{thm:sl4 diag}
The algebra $\mathcal A_{q,4}$ is PBW on the totally ordered set
$$
\{Y_{2}, Y_{21},Y_{23},
Z_{21},Z_{23},Y_{13}, Z_{123},Z_{321},Z_{1232},
Z_{13},Y_3,Y_1\},
$$
where $Y_i=u_i$, $1\le i\le 3$, $Z_{ij}=(1-q^2)^{-|i-j|} z_{ij}$, $1\le i\not=j\le 3$, and
\begin{alignat*}{2}
&Y_{2i}=\frac{u_iu_2-q^{-2}u_2 u_i+q^{-1} z_{i2}}{q^2-q^{-2}},&\quad& Z_{i2j}=\frac{[u_j,z_{2i}]_{q^{-2}}+q^{-1}z_{13}}{(1-q^4)(1-q^{-2})}\\
&Y_{213}=\frac{[u_j,Y_{2i}]_{q^{-2}}}{q^2-q^{-2}}-\frac{Z_{j2i}}{[2]_q},&
&Z_{2132}=\frac{q Z_{2j}Y_{2i}-q^{-1}Y_{2i} Z_{2j}}{q-q^{-1}},
\end{alignat*}
where $\{i,j\}=\{1,3\}$.
The PBW-type relations are given by the following formulae (where $i\in\{1,3\}$ and $\{i,j\}=\{1,3\}$)
{\footnotesize
\begin{align*}
&Y_{2i}Y_2=q^2 Y_2 Y_{2i},\qquad 
Z_{2i}Y_2=Y_2 Z_{2i},\qquad Z_{1232}Y_2=q^{2}Y_2 Z_{1232},\\
&Y_{13}Y_2=Y_2 Y_{13}+h_2Y_{21}Y_{23}+h_1Z_{1232},\\
&Y_2 Z_{13}=q^2 Z_{13}Y_2+h_1(Y_2 Z_{123}+q^2 Y_2 Z_{321}-q^{-2} Y_{23} Z_{21}-q
    Z_{21} Z_{23}-Y_{21} Z_{23})-h_1h_2Z_{1232}
,\\
&Y_i Y_2=q^{-2}Y_2 Y_i+h_2Y_{2i}+h_1Z_{2i},\quad Y_{2i}Z_{13}=Z_{13}Y_{2i}+h_1(Y_{2i} Z_{i2j}-Z_{2i} Y_{13}),\\
&Y_{23}Y_{21}=Y_{21} Y_{23},\qquad
Z_{2i}Y_{2i}=Y_{2i} Z_{2i},\qquad Y_{13}Y_{2i}=q^{2}Y_{2i} Y_{13},\\
&Z_{j2i}Y_{2i}=Y_{2i} Z_{j2i},\quad Z_{i2j}Y_{2i}=q^2 Y_{2i}Z_{i2j},\quad
Z_{1232}Y_{i2}=Y_{i2} Z_{1232},\quad 
Y_iY_{2i}=q^{2}Y_{2i} Y_i,\\
&Y_{i}Y_{2j}=q^{-2} Y_{2i}Y_j+h_2Y_{13}+h_1Z_{j2i},\quad Z_{2j}Y_{2i}=q^{-2}Y_{2i} Z_{2j}+q^{-1}h_1Z_{1232},\\
&Z_{23}Z_{21}=Z_{21}Z_{23}+h_1(Y_2 Z_{123}-Y_2 Z_{321}+q^{-2} Y_{21} Z_{23}-q^{-2} Y_{23} Z_{21}),\\
&Y_{13}Z_{2i}=Z_{2i}Y_{13}+qh_1Y_{2i} Z_{i2j},\\
&Z_{2i}Z_{j2i}=q^2 Z_{j2i}Z_{2i}+h_1(Z_{2i}Y_{13}-Y_{2i}Z_{i2j}-q Y_{2i}Z_{13}),\\
&Z_{i2j}Z_{2i}=q^2 Z_{2i}Z_{i2j}+h_1(Z_{2i}Y_{13}- Y_{2i} Z_{i2j}-Z_{1232} Y_i),\\
&Z_{1232}Z_{2i}=Z_{2i} Z_{1232}+h_1(q^2 Y_2 Y_{2i} Z_{i2j}+q^{-2} Y_{2i}
    Z_{1232}-q^{-2} Y_{21} Y_{23} Z_{2i}),\\
& Z_{13}Z_{2i}=Z_{2i}Z_{13}+h_1q^{-2}(q Y_2 Z_{j2i} Y_i+q^{-1} Y_{2i} Z_{2j} Y_i-Y_{2i}
    Z_{13}+Z_{2i} Z_{j2i})\\&\qquad\qquad+h_1^2q^{-2}(Z_{1232} Y_i+Y_{2i} Z_{i2j}-Z_{2i} Y_{13}),\\
&Y_iZ_{2j}=q^{-2}Z_{j2}Y_i+h_1Z_{13}+h_2Z_{j2i},\quad Z_{13}Y_{13}=q^{2}Y_{13} Z_{13}-q h_1Z_{123} Z_{321},\\
&Y_iZ_{2i}=Z_{2i} Y_i,\qquad
Z_{i2j}Y_{13}=Y_{13} Z_{i2j},\qquad 
Z_{1232}Y_{13}=q^{-2}Y_{13} Z_{1232},\\
&Y_iY_{13}=q^{2}Y_{13} Y_i,\qquad Z_{321}Z_{123}=Z_{123} Z_{321},\\
&Z_{1232}Z_{i2j}=Z_{i2j} Z_{1232}+h_1(Y_{21} Y_{23} Z_{i2j}-Y_{2j} Z_{2i} Y_{213}+q^{-2} Y_{213} Z_{1232}),\\
&Z_{i2j}Z_{13}=Z_{13}Z_{i2j}+h_1(Y_{213} Z_{213}- Z_{123} Z_{321}+q^{-1} Y_{2j} Z_{j2i}
    Y_i-q^{-1} Z_{2j} Y_{213} Y_i),\\
&Z_{1232}Z_{13}=Z_{13}Z_{1232}+h_1(q^{-1} Y_2 Y_{213} Z_{123}+q Y_2 Y_{213} Z_{321}+Y_2
    Z_{123} Z_{321}-Y_{21} Y_{23} Z_{123}\\&\qquad\qquad-q^{-1}Y_{21} Y_{23}
    Z_{321}-Z_{21} Z_{23} Y_{213})
    +h_1^2(q^{-2} Y_{23} Z_{21} Y_{213}-q^{-4} Y_{213}
    Z_{2132}),\\
&Y_iZ_{1232}=Z_{1232} Y_i-h_1q(Y_{2i} Z_{i2j}-
Z_{2i} Y_{13}),\\
&Y_iZ_{i2j}=Z_{i2j} Y_i,\quad Y_iZ_{j2i}=q^2 Z_{j2i}Y_i,\quad
Z_{13}Y_i=Y_i Z_{13},\quad
Y_3Y_1=Y_1 Y_3.
\end{align*}
}
where we abbreviate $h_k=q^k-q^{-k}$.
\end{theorem}
\begin{proof}
Since the proof is rather computational, we only provide a sketch. First, we define an algebra $\mathcal A'$ with 
generators $Y_\alpha$, $\alpha\in\{1,2,3,21,23,13\}$ and $Z_{\beta}$, $\beta\in\{21,23,13,123,321,1232\}$ and relations 
as above. Using the Diamond Lemma (see Proposition \ref{prop:diamond lemma}) we show that $\mathcal A'$ is PBW on these generators 
with the total order as defined in the theorem. Next, we introduce generators $u_i=Y_i$, $1\le i\le 3$, $z_{ij}=(1-q^2)^{-|i-j|}Z_{ij}$
$1\le i,j\le 3$ and show that they satisfy the relations in Theorem~\ref{th:first folding 2}\eqref{th:first folding 2.ii}. In particular,
this yields a surjective homomorphism of algebras $\mathcal A_{q,4}\to\mathcal A'$. To prove that it is an isomorphism, we use Lemma~\ref{lem:spec of ideals}.
We define a grading on~$\mathcal A_{q,4}$ by $\deg u_i=1$, $\deg z_{12}=\deg z_{23}=2$ and $\deg z_{13}=3$. It is easy to see that
specializations of defining relations of~$\mathcal A_{q,4}$ are defining relations in the universal enveloping algebra of a nilpotent Lie algebra 
$\lie n$ of dimension~$12$ 
generated by $u_i$, $1\le i\le 3$ and $z_{ij}=z_{ji}$, $1\le i<j\le 3$ subject to the relations
\begin{alignat*}{2}
&[u_i,[u_i, u_j]]=0,\qquad |i-j|=1,&\qquad& [z_\alpha,z_\beta]=0, \alpha,\beta\in\{12,13,23\}\\
&[u_i,z_{i2}]=[u_2,z_{i2}]=[u_i,z_{13}]=0,&&i\not=2\\
&[u_2,[u_2,z_{13}]]=[z_{13},[z_{13},u_2]=0,&&[u_1,u_3]=0\\
&[z_{i2},[z_{i2},u_j]]=[u_j,[u_j,z_{i2}]]=0,&&
[z_{i2},[u_j,u_2]]=[u_2,z_{13}],\quad \{i,j\}=\{1,3\}
\end{alignat*}
In particular, $\dim U(\lie n)_k=\dim \mathcal A'_k$ for all~$k$, where we endow $\mathcal A'_k$ with the induced grading.
It remains to apply Lemma~\ref{lem:spec of ideals}.
\end{proof}

Using the above and Proposition~\ref{prop:Specialization}, we immediately obtain
\begin{theorem}\label{thm:sl4 diag Poisson}
The algebra $\mathcal A_{q,4}$ is optimal specializable. In particular, the following formulae define a 
Poisson structure on its specialization (only non-zero brackets are shown), where $i\in\{1,3\}$ and $\{i,j\}=\{1,3\}$.
{\footnotesize
\begin{align*}
&\{Y_2,Y_i\}=2 Y_2 Y_i-4 Y_{2i}-2 Z_{i2}\\
&\{Y_2,Y_{2i}\}=-2 Y_{i2} Y_2\\
&\{Y_2,Y_{13}\}=-4 Y_{12} Y_{32}-2 Z_{2132}\\
&\{Y_2,Z_{i2j}\}=-2 Z_{2i} Y_{2j}+2 Z_{2132}\\
&\{Y_2,Z_{13}\}=2 Y_2 Z_{123}+2 Y_2 Z_{13}+2 Y_2
    Z_{321}-2 Y_{21} Z_{23}-2 Y_{23} Z_{21}-2
    Z_{21} Z_{23}\\
&\{Y_2,Z_{2132}\}=-2 Y_2 Z_{2132}\\
&\{Y_i,Y_{2i}\}=2 Y_i Y_{2i}\\
&\{Y_j,Y_{2i}\}=-2 Y_{2i} Y_j+4 Y_{13}+2 Z_{j2i}\\
&\{Y_{2i},Y_{13}\}=-2 Y_{2i} Y_{13}\\
&\{Y_{2i},Z_{2j}\}=2 Y_{i2} Z_{j2}-2 Z_{2132}\\
&\{Y_{2i},Z_{i2j}\}=-2 Y_{2i} Z_{i2j}\\
&\{Y_{2i},Z_{13}\}=2 Y_{2i} Z_{13}-2 Z_{2i} Y_{13}\\
&\{Y_i,Y_{13}\}=2 Y_{13} Y_i\\
&\{Y_{13},Z_{13}\}=-2 Y_{13} Z_{13}+2 Z_{123} Z_{321}\\
&\{Y_{13},Z_{1232}\}=2 Y_{13} Z_{2132}\\
&\{Y_i,Z_{2j}\}=-2 Z_{2j} Y_i+4 Z_{j2i}+2 Z_{13}\\
&\{Y_{13},Z_{i2}\}=2 Y_{i2} Z_{i2j}\\
&\{Z_{21},Z_{23}\}=-2 Y_2 Z_{123}+2 Y_2 Z_{321}-2 Y_{21}
    Z_{23}+2 Y_{23} Z_{21}\\
&\{Z_{i2},Z_{i2j}\}=2 Z_{1232} Y_i+2 Y_{2i} Z_{i2j}-2 Z_{2i}
    Z_{i2j}-2 Z_{2i} Y_{13}\\
&\{Z_{2i},Z_{13}\}=2 (Y_2 Z_{j2i} Y_i-Y_{2i} Z_{2j} Y_i+
    Y_{2i} Z_{13}-Z_{2i} Z_{321})\\
&\{Z_{2i},Z_{j2i}\}=-2 Y_{2i} Z_{i2j}-2 Y_{2i} Z_{13}+2 Z_{2i}
    Y_{13}+2 Z_{2i} Z_{j2i}\\
&\{Z_{2i},Z_{1232}\}=-2 Y_2 Y_{2i} Z_{i2j}-2 Y_{2i} Z_{1232}+2
    Y_{21} Y_{23} Z_{2i}\\
&\{Z_{13},Z_{i2j}\}=2 Z_{23} Y_{13} Y_1-2 Y_{23} Z_{321} Y_1+2
    Z_{123} Z_{321}-2 Y_{13} Z_{13}\\
&\{Z_{13},Z_{2132}\}= 2(Y_{21} Y_{23} Z_{123}-Y_2 Y_{13} Z_{123}-Y_2 Z_{123} Z_{321}-Y_2 Y_{13} Z_{321}+
    Z_{21} Z_{23} Y_{13}+Y_{21} Y_{23} Z_{321})\\
&\{Y_i,Z_{j2i}\}=2 Y_i Z_{j2i}\\
&\{Z_{1232},Y_i\}=2 Y_{2i} Z_{i2j}-2 Y_{13} Z_{2i}\\
&\{Z_{1232},Z_{i2j}\}= 2 Y_{21} Y_{23} Z_{i2j}-2 Y_{2i} Z_{2j} Y_{13}+2
    Y_{13} Z_{1232}
\end{align*}
}
\end{theorem}
It remains to prove that $\mathcal A_{q,4}$ is the uberalgebra for our quantum folding. 
For, let $\bi=\{2,1,3,2,1,3\}\in R(w_\circ)$ and define the elements $x_\alpha\in U_q^+(\lie{sl}_4)$
by
$$
X_{\bi}=\{ x_2,x_{21},x_{23},x_{13},x_1,x_3\}
$$
as ordered sets, in the notation of Section~\ref{subs:PBW bases}. We identify the $x_\alpha$
with the elements of the first copy of $U_q^+(\lie{sl}_4)$ in $U_q^+(\lie{sl}_4)^{\tensor 2}$ and
denote by $x'_\alpha$ the corresponding elements of the second copy. Then the quantum 
folding $\iota_\bi$ is given by extending multiplicatively
$$
\iota_\bi(x_\alpha)=x_\alpha x_\alpha'.
$$
The following Lemma is checked by direct computations.
\begin{lemma}
Let $\tilde y_{\alpha}=x_\alpha x_\alpha'$ and let 
\begin{align*}
&\tilde z_{2i}=q^{-1}(x_{2i}x_2'x_i'+x_2 x_i x_{2i}'-2 \tilde y_{2i}),\\
&\tilde z_{13}=q^{-2} (4 \tilde y_{13}-2 x_{21} x_3
    x_{13}'-2 x_{23} x_1 x'_{13}-
    2 x_{13} x_{21}' x_3'-2 x_{13} x'_{23}
    x'_1\\&\qquad\qquad+ x_{21} x_3 x'_{23}
    x'_1+x_{23} x_1 x'_{21} x_3'+ x_2 x_1 x_3 x'_{13}+ x_{13} x'_2 x'_1
    x'_3)\\
&\tilde z_{i2j}=q^{-1}(x_{13}x'_{2j}x'_i+x_{2j}x_ix'_{13}-2\tilde y_{13})\\
&\tilde z_{1232}=x_2 x_{13}x'_{21}x'_{23}+x_{21}x_{23}x'_2x'_{13}-2q^{-1} \tilde y_{21}\tilde y_{23}.
\end{align*}
Then the assignment $Y_\alpha\mapsto \tilde y_\alpha$, $Z_\alpha\mapsto\tilde z_\alpha$ defines a 
surjective algebra homomorphism $\mathcal A_{q,4}\to \lr{U_q^+(\lie{sl}_4)}_{\iota_{\bi}}$. Moreover,
the folding $\iota_\bi$ is tame liftable.
\end{lemma}
In this case as well it can be shown that the diagram~\eqref{eq:cone over iota} commutes for 
a suitable choice of~$\tilde\iota_\bi$.
We conclude this section with the following problem.
\begin{problem}
Construct the uberalgebra for the quantum folding $(\lie{sl}_n^{\times k},\lie{sl}_n)$ for all~$n$ and~$k\ge 2$.
\end{problem}

\section{Folding \texorpdfstring{$(\lie{so}_8,G_2)$}{Folding (so\_8,G\_2)}}\label{sect:G2}

In this section we let~$\lie g=\lie{so}_8$ with~$I=\{0,1,2,3\}$ so that $\sigma$ is a cyclic permutation of~$\{1,2,3\}$
and $I/\sigma=\{0,1\}$. In this numbering we have $(\alpha_0,\alpha_0)=2$, $(\alpha_1,\alpha_1)=6$ in~$\lie g^\sigma{}^\vee$
which we abbreviate as $\lie g^\sigma$ since its Langlands dual is obtained by renumbering the simple roots.
Let $\mathcal U_{q,G_2}$ be the associative $\CC(q)$-algebra generated by $U_q(\lie{sl}_2)$ with 
Chevalley generators $E_0,F_0,K^{\pm1}_0$ and $w,z_1,z_2$ satisfying the relations given in Theorem~\ref{th:second folding G2}\eqref{th:G2.iii}
(with $u$ replaced by~$E_0$) as well as
$$
[F_0,w]=0=[F_0,z_j],\qquad K_0 w K_0^{-1}=q^{-3} w,\quad K_0 z_j K_0^{-1}=q^{-1} z_j,\quad j=1,2.
$$
\begin{theorem} \label{th:G_2} 
\noindent
\begin{enumerate}[{\rm(i)}]

\item The algebra ${\mathcal U}_{q,G_2}$ is isomorphic to the cross product $A_q \rtimes U_q(\lie{sl}_2)$,
 where $A_q$ is a flat deformation of the symmetric algebra of the nilpotent Lie algebra $\lie n_{G_2}$ defined 
 in Theorem~\ref{th:second folding G2}\eqref{th:G2.ii}.

\item\label{th:G_2.iv} The assignment 
$w\mapsto E_1$, $z_j\mapsto 0$, $j=1,2$ defines a homomorphism $\hat\iota:{\mathcal U}_{q,n}\to U_q(\lie g^\sigma)$. Its image is 
the (parabolic) subalgebra of $U_q(\lie g^\sigma)$ generated by $U_q^+(\lie g^\sigma)$ and $K_0^{\pm 1}, F_0$.

\item\label{th:G_2.v}
The assignments 
$w\mapsto E_{1}E_2E_3$ and
\begin{align*}
&z_1\mapsto [E_{1}E_2E_3,E_0]_{q^{-3}}-\frac{q^2+1+q^{-2}}{(q-q^{-1})^2} [E_1,[E_2,[E_3,E_0]_{q^{-1}}]_{q^{-1}}]_{q^{-1}}\\
&z_2\mapsto [E_0,E_1E_2E_3]_{q^{-3}}-\frac{q^2+1+q^{-2}}{(q-q^{-1})^2}[[[E_0,E_1]_{q^{-1}},E_2]_{q^{-1}},E_3]_{q^{-1}}
\end{align*}
define an algebra homomorphism 
$\mu:{\mathcal U}_{q,G_2}\to U_q(\lie{g}).$ Its image is contained in
the (parabolic) subalgebra of $U_q(\lie{g})$ generated by $U_q^+(\lie g)$ and $K_0^{\pm 1}, F_0$.
\item\label{th:G_2.iii} The assignments 
$$
T_{0}(w)=((q^2+1+q^{-2})(q+q^{-1}))^{-1} [[w,E_0]_{q^{-3}}]_{q^{-1}}]_{q},
\quad T_0(z_j)=[z_j,E_0]_{q^{-1}}
$$
extend Lusztig's action~\eqref{eq:Lustig braid action} of the braid group $Br_{\lie{sl}_2}$ on $U_q(\lie{sl}_2)$ to an action on~$\mathcal U_{q,G_2}$ by algebra automorphisms. Moreover, $\mu$ and $\hat\iota$ are
$Br_{\lie{sl}_2}$-equvivariant.
\end{enumerate}
\end{theorem}
Most of the computations necessary to prove this theorem were performed on a computer and were involving rather heavy computations (for example,
it took about 22 hours for the UCR cluster to check that the Diamond Lemma holds). Otherwise, the structure of the proof is rather similar 
to the ones discussed above.

\appendix

\section{Naive quantum folding does not exist}\label{sec:non-existent}

In this appendix, we show that the classical additive folding does not admit a quantum deformation even in the simplest possible case of~$\lie{sl}_4$.
We use the standard numbering of the nodes of its Dynkin diagram.
Let $u_1=E_1+E_3$, $u_2=E_2$. We obviously have
$$
u_2^2 u_1-(q+q^{-1})u_2 u_1 u_2+u_1u_2^2=0.
$$
On the other hand, suppose that we have a relation 
\begin{equation}\label{A.10}
\sum_{j=0}^3 c_j u_1^j u_2 u_1^{3-j}=0.
\end{equation}
Retain the notation of~\ref{subs:mod alg semidirect}.
Applying $r_2$ and $r_2 r_1$ to the left hand side of~\eqref{A.10} and considering 
the coefficients of linearly independent monomials we obtain a system of linear equations for the~$c_i$, $0\le i\le 3$ with the 
matrix 
$$
\begin{pmatrix}
1 & q^{-1} & q^{-2} & q^{-3} \\
    3 q^2 & q (q^2+2) & 2 q^2+1 & 3 q \\
    q^2+1+q^{-2} & q+2q^{-1} & 2+q^{-2} &
      q+q^{-1}+q^{-3} \\
    3 (q^2+1) & 5 q+q^{-1} & q^2+5 & 3 (q+q^{-1})
    \end{pmatrix}
$$
of determinant $-(q-q^{-1})^6$. Thus, there is no relation of the form~\eqref{A.10}.
Clearly, replacing $u_1$ and $u_2$ by $E_1+E_3+(q-1)a$, $E_2+(q-1)b$, where $a,b$ are $\sigma$-invariant elements of degree greater than one will not affect the
above calculation. Thus, there is no embedding of $U_q^+(\lie{so}_5)$ into $U_q^+(\lie{sl}_4)$ which deforms the embedding of $\lie{so}_5$ into~$\lie{sl}_4$.

\end{document}